%% file: optimal_ate_main.tex
\documentclass[11pt,twoside]{article}

\usepackage{fullpage}
\usepackage{epsf}
\usepackage{fancyhdr}
\usepackage{graphics}
\usepackage{graphicx}
\usepackage{psfrag}
\usepackage{microtype}
\usepackage{algorithmic}
\usepackage{color,xcolor}
\usepackage[linesnumbered,ruled]{algorithm2e}% http://ctan.org/pkg/algorithm2e
\DontPrintSemicolon
\usepackage{color}

\usepackage{amsthm}
\usepackage{amsfonts}
\usepackage{amsmath}
\usepackage{amssymb,bbm}
\usepackage{comment}
\input{final_macros.tex}

\usepackage{url}% for url's in bib
\usepackage[colorlinks,linkcolor=magenta,citecolor=blue, pagebackref=true,backref=true]{hyperref}
\renewcommand*{\backrefalt}[4]{%
    \ifcase #1 \footnotesize{(Not cited.)}%
    \or        \footnotesize{(Cited on page~#2.)}%
    \else      \footnotesize{(Cited on pages~#2.)}%
    \fi}

\usepackage{nicefrac}
\usepackage{comment}

\usepackage{chngpage}

\usepackage[mathscr]{euscript}

 \usepackage{tabularx}%

\usepackage{enumitem}
\usepackage{booktabs}
\usepackage{caption}

\usepackage{bm,bbm}
\usepackage{mathtools}
\usepackage{cleveref}
\usepackage{setspace}

\newcommand{\diameter}{\mathrm{diam}}

\newcommand{\ctwofour}{M_{2 \rightarrow 4}}

\newcommand{\subgaussian}{\nu}

\newcommand{\datablock}{\mathcal{B}}

\newcommand{\coordinate}{e}
\newcommand{\Wass}{\mathcal{W}}

\newcommand{\actionspace}{\mathbb{A}}

\newcommand{\figwidth}{0.45}

\newcommand{\varbound}{\bar{\sigma}}

\newcommand{\fatshatter}[1]{\mathrm{fat}_{#1}}
\newcommand{\VCdim}{\mathrm{VC}}

\newcommand{\weightedinprod}[2]{\inprod{#1}{#2}_\omega}

\newcommand{\dudley}{\mathcal{J}}

\newcommand{\indicator}{\mathbb{I}}

\newcommand{\statespace}{\mathcal{S}}
\newcommand{\goodendex}{\ensuremath{\clubsuit}}

\newcommand{\lammin}{\lambda_{\min}}

\newcommand{\minimaxRisk}{\mathscr{M}_\numobs}

\newcommand{\smallballcon}{\alpha_1}
\newcommand{\smallballprob}{\alpha_2}

\newcommand{\tauhat}{\widehat{\avgtreat}}
\newcommand{\Ltwospace}{\mathbb{L}^2}

\newcommand{\poissonDistr}{\mathrm{Poi}}

\newcommand{\bernDistr}{\mathrm{Ber}}

\newtheorem{example}{Example}
\newcommand{\metric}{\rho}

\newcommand{\support}{\mathrm{supp}}

\newtheorem{assumption}{Assumption}
\newcommand{\numobs}{\ensuremath{n}}

\newcommand{\usedim}{\ensuremath{d}}

\newcommand{\totalvariation}{d_{\mathrm{TV}}}

\newcommand{\radeComplexity}{\mathcal{R}}

\newcommand{\convSet}{\Omega}

\newcommand{\probx}{\xi}

\newcommand{\conv}{\mathrm{conv}}

\newcommand{\avgtreat}{\tau}

\newcommand{\rade}{\ensuremath{\varepsilon}}

\newcommand{\fstar}{f^*}

\newcommand{\class}{\mathcal{C}}

\newcommand{\talchain}[1]{\gamma_{#1}}

\newcommand{\weightednorm}[1]{\vecnorm{#1}{\omega}}

\newcommand{\Event}{\mathscr{E}}
\newcommand{\Term}{T}

\newcommand{\chisqdiv}[2]{\chi^2 \left( #1 ~||~ #2 \right)}
\long\def\comment#1{}

\long\def\comment#1{}

\newcommand{\funcClass}{\mathcal{F}}
\newcommand{\state}{s}

\newcommand{\probInstance}{\mathcal{I}}

\newcommand{\mubar}{\widebar{\mu}}

\newcommand{\weightfunc}{g}

\newcommand{\propscore}{\pi}
\newcommand{\treateff}{\plaintreateff^*}

\newcommand{\plaintreateff}{\mu}

\newcommand{\treatsig}{\sigma }

\newcommand{\outcome}{Y}

\newcommand{\actbasemsr}{\basemsr}
\newcommand{\actinprod}[2]{\inprod{#1}{#2}_\actbasemsr}

\newcommand{\failprob}{\varepsilon}

\newcommand{\neighborhood}{\mathscr{N}}

\newcommand{\basemsr}{\lambda}

\newcommand{\taustar}{\tau^*}

\usepackage[labelformat=simple]{subcaption}

\newcommand{\lammax}{\ensuremath{\lambda_{\mbox{\tiny{max}}}}}

\newcommand{\Deltil}{\widetilde{\Delta}}

\newcommand{\funcClassTemp}{\mathcal{H}}

\newcommand{\Delhat}{\widehat{\Delta}}

\newcommand{\muhat}{\widehat{\mu}}

\newcommand{\fhat}{\widehat{f}}

\newcommand{\maux}{m}

\newcommand{\vstar}{v_*}

\renewcommand{\state}{\ensuremath{x}}
\newcommand{\action}{\ensuremath{a}}

\newcommand{\State}{\ensuremath{X}}
\newcommand{\Action}{\ensuremath{A}}

\newcommand{\StateSpace}{\Xspace}
\renewcommand{\statespace}{\StateSpace}

\newcommand{\ActionSpace}{\ensuremath{\mathbb{A}}}

\newcommand{\Outcome}{\ensuremath{Y}}

\newcommand{\plainmu}{\ensuremath{\mu}}

\newcommand{\radone}{\ensuremath{s}}
\newcommand{\radtwo}{\ensuremath{r}}
\newcommand{\dradone}{\ensuremath{d}}

\newcommand{\Hclass}{\ensuremath{\mathcal{H}}}
\newcommand{\Fclass}{\ensuremath{\mathcal{F}}}

\newcommand{\SquaredRade}{\ensuremath{\mathcal{S}}}
\newcommand{\SquaredDiffRade}{\ensuremath{\mathcal{D}}}

\usepackage{scalerel}
\newcommand{\smallsuper}[2]{#1^{\scaleto{#2}{4pt}}}
\newcommand{\medsuper}[2]{#1^{\scaleto{#2}{6pt}}}

\newcommand{\DudleyBracket}{\smallsub{\dudley}{\mbox{bra}}}
\newcommand{\CoverBracket}{\smallsub{N}{\mbox{bra}}}

\newcommand{\tauhatipw}[1]{\smallsuper{\tauhat}{IPW}_{#1}}
\newcommand{\tauhatgen}[2]{\medsuper{\tauhat}{#2}_{#1}}

\newcommand{\DelF}{\ensuremath{\partial \Fclass}}
\newcommand{\cmax}{c_{\max}}

\newcommand{\ltwolfour}[1]{\ensuremath{\|#1\|_{2 \rightarrow 4}}}
\newcommand{\Zone}{\ensuremath{Z}}
\newcommand{\Ztwo}{\ensuremath{Z'}}
\newcommand{\treateffzero}{\treateff}

\newcommand{\plainprobx}{\ensuremath{\xi}}
\newcommand{\probxstar}{{\ensuremath{\plainprobx^{*}}}}

\newcommand{\probInstanceStar}{\ensuremath{{{\probInstance^{*}}}}}

\newcommand{\htil}{\ensuremath{\tilde{h}}}

\newcommand{\fatdim}{\ensuremath{D}}

\newcommand{\bigoh}{\ensuremath{\mathcal{O}}}

\newcommand{\Nat}{\ensuremath{\mathbb{N}}}

\newcommand{\pdim}{\ensuremath{p}}

\newcommand{\Qprob}{\ensuremath{\mathbb{Q}}}

\newcommand{\linkfun}{\ensuremath{\varphi}}
\newcommand{\linklower}{\ensuremath{\ell_\varphi}}

\newcommand{\Zprime}{\ensuremath{Z'}}

\newcommand{\utmp}{\ensuremath{u}}

\newcommand{\myassumption}[3]{
  \begin{enumerate}[label={\bf{(#1)}}]
  \item \label{#2} {#3}
  \end{enumerate}
 }

\newcommand{\SymVar}{\ensuremath{Z'}}

\newcommand{\OutNoise}{\ensuremath{W}}

\newcommand{\MomFour}{\ensuremath{M_4}}

\newenvironment{carlist}
 {\begin{list}{$\bullet$}
 {\setlength{\topsep}{0in} \setlength{\partopsep}{0in}
  \setlength{\parsep}{0in} \setlength{\itemsep}{\parskip}
  \setlength{\leftmargin}{0.15in} \setlength{\rightmargin}{0.08in}
  \setlength{\listparindent}{0in} \setlength{\labelwidth}{0.08in}
  \setlength{\labelsep}{0.1in} \setlength{\itemindent}{0in}}}
 {\end{list}}

\newcommand{\bcar}{\begin{carlist}}
\newcommand{\ecar}{\end{carlist}}

\newcommand{\tweak}{\ensuremath{s}}
\newcommand{\probxtweak}{\ensuremath{\probx_\tweak}}

\newcommand{\Nval}[1]{\smallsuper{\neighborhood}{val}_{#1}}
\newcommand{\Nprob}{\smallsuper{\neighborhood}{prob}}

\newcommand{\smallsub}[2]{#1_{\scaleto{#2}{4pt}}}
\newcommand{\htrunc}{\smallsub{h}{tr}}

\newcommand{\Normal}{\ensuremath{\mathcal{N}}}

\newcommand{\RadTwo}{\ensuremath{R_2}}
\newcommand{\RadOne}{\ensuremath{R_1}}

\newcommand{\Qnew}{\ensuremath{\mathbb{Q}'}}

\newcommand{\Sobs}{\ensuremath{\smallsub{S}{obs}}}

\makeatletter
\long\def\@makecaption#1#2{
        \vskip 0.8ex
        \setbox\@tempboxa\hbox{\small {\bf #1:} #2}
        \parindent 1.5em  %% How can we use the global value of this???
        \dimen0=\hsize
        \advance\dimen0 by -3em
        \ifdim \wd\@tempboxa >\dimen0
                \hbox to \hsize{
                        \parindent 0em
                        \hfil 
                        \parbox{\dimen0}{\def\baselinestretch{0.96}\small
                                {\bf #1.} #2
                                } 
                        \hfil}
        \else \hbox to \hsize{\hfil \box\@tempboxa \hfil}
        \fi
        }
\makeatother

\begin{document}

\begin{center}
  {\bf{\Large{Off-policy estimation of linear functionals: \\
        Non-asymptotic theory for semi-parametric efficiency} }}

\vspace*{.2in}
{\large{
 \begin{tabular}{ccc}
  Wenlong Mou$^{ \diamond}$ & Martin J. Wainwright$^{\diamond,
    \dagger, \star}$ &Peter L.  Bartlett$^{\diamond, \dagger, \ddagger}$ \\
  \texttt{wmou@berkeley.edu} & \texttt{wainwrigwork@gmail.com} &
  \texttt{peter@berkeley.edu}
 \end{tabular}
}

\vspace*{.2in}

 \begin{tabular}{c}
 Department of Electrical Engineering and Computer
 Sciences$^\diamond$\\ Department of Statistics$^\dagger$ \\ UC
 Berkeley \\
 \end{tabular}
 
\medskip
 
 \begin{tabular}{c}
    Department of Electrical Engineering and Computer
    Sciences$^\star$ \\
Department of Mathematics$^\star$ \\
    Massachusetts Institute of Technology
 \end{tabular}
}

 \medskip
  \begin{tabular}{c}
Google Research Australia$^\ddagger$
 \end{tabular}

\medskip

\begin{abstract}
The problem of estimating a linear functional based on observational
data is canonical in both the causal inference and bandit
literatures. We analyze a broad class of two-stage procedures that
first estimate the treatment effect function, and then use this
quantity to estimate the linear functional.  We prove non-asymptotic
upper bounds on the mean-squared error of such procedures: these
bounds reveal that in order to obtain non-asymptotically optimal
procedures, the error in estimating the treatment effect should be
minimized in a certain \mbox{weighted $L^2$-norm.}  We analyze a
two-stage procedure based on constrained regression in this weighted
norm, and establish its instance-dependent optimality in finite
samples via matching non-asymptotic local minimax lower bounds.  These
results show that the optimal non-asymptotic risk, in addition to
depending on the asymptotically efficient variance, depends on the
weighted norm distance between the true outcome function and its
approximation by the richest function class supported by the sample
size.
\end{abstract}

\end{center}

%%%%%%%%%%%%%%%%%%%%%%%%%%%%%%%%%%%%%%%%%%%%%%%%%%%%%%%%%%%%%%%%%

\section{Introduction}
\label{SecIntro}

A central challenge in both the casual inference and bandit
literatures is how to estimate a linear functional associated with the
treatment (or reward) function, along with inferential issues
associated with such estimators.  Of particular interest in causal
inference are average treatment effects (ATE) and weighted variants
thereof, whereas with bandits and reinforcement learning, one is
interested in various linear functionals of the reward function
(including elements of the value function for a given policy). In many
applications, the statistician has access to only observational data,
and lacks the ability to sample the treatment or the actions according
to the desired probability distribution.  By now, there is a rich body
of work on this problem
(e.g.,~\cite{robins1995analysis,robins1995semiparametric,chernozhukov2018double,armstrong2021finite,wang2017optimal,ma2022minimax}),
including various types of estimators that are equipped with both
asymptotic and non-asymptotic guarantees.  We overview this and other
past work in the related work section to follow.

In this paper, we study how to estimate an arbitrary linear functional
based on observational data.  When formulated in the language of
contextual bandits, each such problem involves a state space
$\Xspace$, an action space $\actionspace$, and an output space
$\Yspace \subseteq \real$. Given a base measure $\basemsr$ on the
action space $\actionspace$---typically, the counting measure for
discrete action spaces, or Lebesgue measure for continuous action
spaces---we equip each $\state \in \Xspace$ with a probability density
function $\propscore(\state, \cdot)$ with respect to $\basemsr$.  This
combination defines a probability distribution over $\actionspace$,
known either as the \emph{propensity score} (in causal inference) or
the \emph{behavioral policy} (in the bandit literature).  The
conditional mean of any outcome $\Outcome \in \Yspace$ is specified as
$\Exs[\Outcome \mid \state, \action] = \treateff(\state, \action)$,
where the function $\treateff$ is known as the \emph{treatment effect}
or the \emph{reward function}, again in the causal inference and
bandit literatures, respectively.

Given some probability distribution $\probxstar$ over the state space
$\StateSpace$, suppose that we observe $\numobs$ i.i.d. triples
$(\State_i, \Action_i, \Outcome_i)$ in which $\State_i \sim
\probxstar$, and
\begin{align}
\Action_i \mid \State_i \sim \propscore(\State_i, \cdot),
\quad \mbox{and} \quad \Exs \big[ \Outcome_i \mid \State_i, \Action_i
  \big] = \treateff(\State_i, \Action_i), \qquad \mbox{for $i = 1, 2,
  \ldots, \numobs$.}
\end{align}
We also make use of the conditional variance function
\begin{align}
 \label{EqnDefnCondVariance}
\sigma^2(\state, \action) \mydefn \Exs \Big[ \big(\outcome -
  \treateff(\State, \Action)\big)^2 \mid \State = \state, \Action =
  \action \Big],
\end{align}
which is assumed to exist for any $\state \in \Xspace$ and $a
\in \actionspace$.

For a pre-specified weight function $\weightfunc: \Xspace
\times \actionspace \rightarrow \real$, our goal is to estimate the
linear functional
\begin{align}
\label{EqnDefnFunctional}  
\taustar \equiv \tau(\probInstanceStar) \mydefn \int_\actionspace
\Exs_{\probxstar} \Big[ \weightfunc(\State, \action) \cdot
  \treateff(\State, \action) \Big] d \actbasemsr(\action),
\end{align}
With this set-up, the pair $\probInstanceStar \mydefn(\probxstar, \treateff)$ defines a particular \emph{problem instance}.  Throughout
the paper, we focus on the case where both the propensity score
$\propscore$ and the weight function $\weightfunc$ are known to the
statistician. \\

\noindent Among the interesting instantiations of this general
framework are the following: \\

\bcar
\item \textbf{Average treatment effect:} The ATE problem corresponds
  to estimating the linear functional
  \begin{align*}
 \taustar = \Exs_\probxstar \Big[ \treateff(\State, 1) -
   \treateff(\State, 0) \Big].
  \end{align*}
It is a special case of equation~\eqref{EqnDefnFunctional}, obtained
by taking the binary action space $\actionspace = \{0, 1\}$ with
$\actbasemsr$ being the counting measure, along with the weight
function $\weightfunc(\state, \action) \defn 2 \action - 1$. \\

\item \textbf{Weighted average treatment effect:} Again with binary
  actions, suppose that we adopt the weight function
  $\weightfunc(\state, \action) \defn (2 \action - 1) \cdot
  w(\state)$, for some given function $w: \Xspace \rightarrow
  \real_+$. With the choice $w(\state) \defn \propscore(\state, 1)$,
  this corresponds to average treatment effect on the treated
  (ATET). \\
  
\item \textbf{Off-policy evaluation for contextual bandits}: For a
  general finite action space $\actionspace$, a \emph{target policy}
  is a mapping $\state \mapsto \smallsuper{\propscore}{tar}(\state,
  \cdot)$, corresponding to a probability distribution over the action
  space. If we take the weight function $\weightfunc(\state, \action)
  \mydefn \smallsuper{\propscore}{tar}(\state, \action)$ and interpret
  $\treateff$ as a reward function, then the linear
  functional~\eqref{EqnDefnFunctional} corresponds to the value of the
  target policy $\smallsuper{\propscore}{tar}$.  Since the observed
  actions are sampled according to $\propscore$---which can be
  different than the target policy
  $\smallsuper{\propscore}{tar}$---this problem is known as
  \emph{off-policy} evaluation in the bandit and reinforcement
  learning literature.
\ecar

\bigskip

\noindent When the propensity score is known, it is a standard fact
that one can estimate $\tau(\probInstance)$ at a $\sqrt{\numobs}$-rate
via an importance-reweighted plug-in estimator.  In particular, under
mild conditions, the \emph{inverse propensity weighting} (IPW)
estimator, given by
\begin{align}
\label{eq:ipw-estimator-simple}
  \tauhatipw{\numobs} \mydefn \frac{1}{\numobs} \sum_{i = 1}^\numobs
  \frac{\weightfunc(\State_i, \Action_i) }{\propscore(\State_i,
    \Action_i)} \outcome_i,
\end{align}
is $\sqrt{\numobs}$-consistent, in the sense that $\tauhatipw{\numobs}
- \taustar = \bigoh_p(1/\sqrt{\numobs})$.

\medskip

However, the problem is more subtle than might appear at might first:
the IPW estimator $\tauhatipw{\numobs}$ fails to be asymptotically
efficient, meaning that its asymptotic variance is larger than the
optimal one.  This deficiency arises even when the state space
$\Xspace$ and action space $\actionspace$ are both binary; for
instance, see \S 3 in Hirano et al.~\cite{hirano2003efficient}.
Estimators that are asymptotically efficient can be obtained by first
estimating the treatment effect $\treateff$, and then using this
quantity to form an estimate of $\tau(\probInstance)$. Such a
combination leads to a semi-parametric method, in which $\treateff$
plays the role of a nuisance function.  For example, in application to
the ATE problem, Chernozhukov et al.~\cite{chernozhukov2018double}
showed that any consistent estimator of $\treateff$ yields an
asymptotically efficient estimate of
$\avgtreat_{\weightfunc}(\probInstance)$; see \S 5.1 in their paper.
In the sequel, so as to motivate the procedures analyzed in this
paper, we discuss a broad range of semi-parametric methods that are
asymptotically efficient for estimating the linear functional
$\tau(\probInstance)$.

While such semi-parametric procedures have attractive asymptotic
guarantees, they are necessarily applied in finite samples, in
which context a number of questions remain open: \\
\bcar
\item As noted above, we now have a wide array of estimators that are
  known to be asymptotically efficient, and are thus ``equivalent''
  from the asymptotic perspective.  It is not clear, however, which
  estimator(s) should be used when working with a finite collection of
  samples, as one always does in practice. Can we develop theory that
  provides more refined guidance on the choice of estimators in this
  regime? \\

\item As opposed to a purely parametric estimator (such as the IPW
  estimate), semi-parametric procedures involve estimating the
  treatment effect function $\treateff$.  Such non-parametric
  estimation requires sample sizes that scale non-trivially with the
  problem dimension, and induce trade-offs between the estimation and
  approximation error.  In what norm should we measure the
  approximation/estimation trade-offs associated with estimating the
  treatment effect?  Can we relate this trade-off to non-asymptotic
  and instance-dependent lower bounds on the difficulty of estimating
  the linear functional $\tau$? \\
\ecar
  
The main goal of this paper is to give some precise answers to these
questions. On the lower bound side, we establish instance-dependent
minimax lower bounds on the difficulty of estimating $\tau$.  These
lower bounds show an interesting elbow effect, in that if the sample
size is overly small relative to the complexity of a function class
associated with the treatment effect, then there is a penalty in
addition to the classical efficient variance.  On the upper bound
side, we propose a class of weighted constrained least-square
estimators that achieve optimal non-asymptotic risk, even in the
high-order terms. Both the upper and lower bounds are general, with
more concrete consequences for the specific instantiations introduced
previously.

%%%%%%%%%%%%%%%%%%%%%%%%%%%%%%%%%%%%%%%%%%%%%%%%%%%%%%%%%%%%%%%

\paragraph{Related work:}
Let us provide a more detailed overview of related work in the areas of
semi-parametric estimation and more specifically, the literatures on
the treatment effect problem as well as related bandit problems.

In this paper, we make use of the notion of local minimax lower bounds
which, in its asymptotic instantiation, dates back to seminal work of
Le Cam~\cite{lecam1960locally} and H\'{a}jek~\cite{hajek1972local}.
These information-based methods were extended to semiparametric
settings by Stein~\cite{stein1956efficient} and
Levit~\cite{levit1975conditional,levit1978infinite}, among other
authors. Under appropriate regularity assumptions, the optimal
efficiency is determined by the worst-case Fisher information of
regular parametric sub-models in the tangent space; see the
monograph~\cite{bickel1993efficient} for a comprehensive review.

Early studies of treatment effect estimation were primarily
empirical~\cite{ashenfelter1978estimating}. The unconfoundedness
assumption was first formalized by Rosenbaum and
Rubin~\cite{rosenbaum1983central}, thereby leading to the problem
setup described in Section~\ref{SecIntro}.  A series of seminal papers
by Robins and
Rotnitzky~\cite{robins1995analysis,robins1995semiparametric} made
connections with the semi-parametric literature; the first
semi-parametric efficiency bound, using the tangent-based techniques
described in the monograph~\cite{bickel1993efficient}, was formally
derived by Hahn~\cite{hahn1998role}.

There is now a rich body of work focused on constructing valid
inference procedures under various settings, achieving such
semiparametric lower bounds.  A range of methods have been studied,
among them matching
procedures~\cite{rubin1992characterizing,abadie2016matching}, inverse
propensity
weighting~\cite{hahn1998role,hirano2003efficient,hitomi2008puzzling,wang2020debiased},
outcome
regression~\cite{chen2004semiparametric,hirshberg2021augmented}, and
doubly robust
methods~\cite{robins1995semiparametric,chernozhukov2018double,mackey2018orthogonal,foster2019orthogonal}. The
two-stage procedure analyzed in the current paper belongs to the broad
category of doubly robust methods.

In their classic paper, Robins and Ritov~\cite{robins1997toward}
showed that if no smoothness assumptions are imposed on the outcome
model, then the asymptotic variance of the IPW estimator cannot be
beaten.  This finding can be understood as a worst-case asymptotic
statement; in contrast, this paper takes an instance-dependent
perspective, so that any additional structure can be leveraged to
obtain superior procedures.  Robins et
al.~\cite{robins2009semiparametric} derived optimal rates for
treatment effect estimation under various smoothness conditions for
the outcome function and propensity score function. More recent work
has extended this general approach to analyze estimators for other
variants of treatment effect
(e.g.,~\cite{kennedy2022minimax,armstrong2021finite}).  There are some
connections between our proof techniques and the analysis in this line
of work, but our focus is on finite-sample and instance-dependent
results, as opposed to global minimax results.

Portions of our work apply to high-dimensional settings, of which
sparse linear models are one instantiation. For this class of
problems, the recent
papers~\cite{bradic2019minimax,bradic2019sparsity,wang2020debiased}
study the relation between sample size, dimension and sparsity level
for which $\sqrt{\numobs}$-consistency can be obtained.  This body of
work applies to the case of unknown propensity scores, which is
complementary to our studies with known behavioral policies.  To be
clear, obtaining $\sqrt{\numobs}$-consistency is always possible under
our set-up via the IPW estimator; thus, our focus is on the more refined
question of non-asymptotic sample size needed to obtain optimal
instance-dependent bounds.

Our work is also related to the notion of second-order efficiency in
classical asymptotics.  Some past
work~\cite{dalalyan2006penalized,dalalyan2006penalized,castillo2007semi}
has studied some canonical semi-parametric problems, including
estimating the shift or period of one-dimensional regression
functions, and established second-order efficiency asymptotic upper
and lower bounds in the exact asymptotics framework.  Our
instance-dependent lower bounds do not lead to sharp constant factors,
but do hold in finite samples.  We view it as an important direction
for future work to combine exact asymptotic theory with our
finite-sample approach so as to obtain second-order efficiency lower
bounds with exact first-order asymptotics.

There is also an independent and parallel line of research on the
equivalent problem of off-policy evaluation (OPE) in bandits and
reinforcement learning.  For multi-arm bandits, the
paper~\cite{li2015toward} established the global minimax optimality of
certain OPE estimators given a sufficiently large sample size.  Wang
et al.~\cite{wang2017optimal} proposed the ``switch'' estimator, which
switches between importance sampling and regression estimators; this
type of procedure, with a particular switching rule, was later shown
to be globally minimax optimal for any sample
size~\cite{ma2022minimax}. Despite desirable properties in a
worst-case sense, these estimators are known to be asymptotically
inefficient, and the sub-optimality is present even ignoring constant
factors (see Section 3 of the paper~\cite{hirano2003efficient} for
some relevant discussion).  In the more general setting of
reinforcement learning, various efficient off-policy evaluation
procedures have been proposed and
studied~\cite{jiang2016doubly,yin2020asymptotically,ZanWaiBru_neurips21,
  kallus2022efficiently}.  Other
researchers~\cite{zhou2022offline,zhan2021policy,athey2021policy,ZanWaiBru_neurips21}
have studied procedures that are applicable to adaptively collected
data.  It is an interesting open question to see how the perspective
of this paper can be extended to dynamic settings of this type.

\paragraph{Notation:}
Here we collect some notation used throughout the paper. Given a pair
of functions $h_1, h_2: \actionspace \rightarrow \real$ such that
$|h_1 h_2| \in \mathbb{L}^1 (\actbasemsr)$, we define the inner
product
\begin{align*}
\actinprod{h_1}{h_2} \mydefn \int h_1(\action) h_2(\action) ~ d\actbasemsr(\action)
\end{align*}
Given a set $A$ in a normed vector space with norm
$\vecnorm{\cdot}{c}$, we denote the diameter $\diameter_c(A) \mydefn
\sup_{x, y \in A} \vecnorm{x - y}{c}$.  For any $\alpha > 0$, the
\emph{Orlicz norm} of a scalar random variable $X$ is given by
\begin{align*}
  \vecnorm{X}{\psi_\alpha} \mydefn \sup \left\{ u > 0 \, \mid \, \Exs
  \big[ e^{( |X| / u)^\alpha} \big] \leq 1 \right\}.
\end{align*}
The choices $\alpha = 2$ and $\alpha = 1$ correspond, respectively,
to the cases of sub-Gaussian and sub-exponential tails, respectively.

Given a metric space $(\mathbb{T}, \metric)$ and a set $\Omega
\subseteq \mathbb{T}$, we use $N(\Omega, \metric; s)$ to denote the
cardinality of a minimal $s$-covering of set $\Omega$ under the metric
$\metric$. For any scalar $q \geq 1$ and closed interval $[\delta, D]$
we define the Dudley entropy integral
\begin{align*}
 \dudley_q (\Omega, \metric; [\delta, D]) \mydefn \int_\delta^D \big[\log
   N(\Omega, \metric; s) \big]^{1/q} \; ds.
\end{align*}
Given a domain $\statespace$, a bracket $[\ell, u]$ is a pair of
real-valued functions on $\statespace$ such that $\ell (x) \leq u (x)$
for any $x \in \statespace$, and a function $f$ is said to lie in the
bracket $[\ell, u]$ if $f (x) \in [\ell (x), u (x)]$ for any $x
\in \statespace$. Given a probability measure $\mathbb{Q}$ over
$\statespace$, the size of the bracket $[\ell, u]$ is defined as
$\vecnorm{u - \ell}{\Ltwospace (\mathbb{Q})}$. For a function class
$\funcClass$ over $\statespace$, the bracketing number
$\CoverBracket\big( \funcClass, \Ltwospace (\mathbb{Q}); s \big)$
denotes the cardinality of a minimal bracket covering of the set
$\funcClass$, with each bracket of size \mbox{smaller than $s$.} Given
a closed interval $[\delta, D]$, the bracketed chaining integral is
given by
\begin{align*}
 \DudleyBracket (\funcClass, \Ltwospace (\mathbb{Q}); [\delta, D]) \mydefn
 \int_\delta^D \sqrt{\log \CoverBracket (\funcClass, \Ltwospace
   (\mathbb{Q}); s) } \; ds.
\end{align*}

%%%%%%%%%%%%%%%%%%%%%%%%%%%%%%%%%%%%%%%%%%%%%%%%%%%%%%%%%%%%%%%%%%%%%

\section{Non-asymptotic and instance-dependent upper bounds}
\label{SecUpper}

We begin with a non-asymptotic analysis of a general class of
two-stage estimators of the functional $\tau(\probInstance)$.  Our
upper bounds involve a certain weighted $L^2$-norm---see
equation~\eqref{EqnWeighted}---which, as shown by our lower bounds in
the sequel, plays a fundamental role.

\subsection{Non-asymptotic risk bounds on two-stage procedures}

We first provide some intuition for the class of two-stage estimators
that we analyze, before turning to a precise description.

\subsubsection{Some elementary intuition}

We consider two-stage estimators obtained from simple perturbations of
the IPW estimator~\eqref{eq:ipw-estimator-simple}.  Given an auxiliary
function $f:
\Xspace \times \actionspace \rightarrow \real$ and the data set
$\{(\State_i, \Action_i, \Outcome_i) \}_{i=1}^\numobs$, consider the
estimate
\begin{align}
\label{eq:general-form-of-unbiased-estimator}  
\tauhatgen{\numobs}{f} = \frac{1}{\numobs} \sum_{i = 1}^\numobs \Big\{
\frac{\weightfunc(\State_i, \Action_i)}{\propscore(\State_i,
  \Action_i)} \outcome_i - f(\State_i, \Action_i) + \actinprod{f
 (\State_i, \cdot)}{\propscore(\State_i, \cdot)} \Big\}.
\end{align}
By construction, for any choice of $f \in \Ltwospace(\probxstar \times
\propscore)$, the quantity $\tauhatgen{\numobs}{f}$ is an unbiased
estimate of $\tau$, so that it is natural to choose $f$ so as to
minimize the variance $\var(\tauhatgen{\numobs}{f})$ of the induced
estimator.  As shown
in~\Cref{app:sec:proof-basic-exp-and-var-of-general-form}, the minimum
of this variational problem is achieved by the function
\begin{subequations}
\begin{align}
\label{eq:defn-fstar-function}  
 \fstar(\state, \action) \mydefn \frac{\weightfunc(\state, \action)
   \treateff(\state, \action)}{\propscore(\state, \action)} -
 \actinprod{\weightfunc(\state, \cdot)}{\treateff(\state, \cdot)},
\end{align}
where in performing the minimization, we enforced the constraints
$\actinprod{f(\state, \cdot)}{\propscore(\state, \cdot)} = 0$ for any
$\state \in \Xspace$.  We note that this same function $\fstar$ also
arises naturally via consideration of Neyman orthogonality.

The key property of the optimizing function $\fstar$ is that it
induces an estimator $\tauhatgen{\numobs}{\fstar}$ with asymptotically optimal variance---viz.
\begin{align}
\label{eq:defn-optimal-variance}  
\vstar^2 \mydefn \var \Big( \actinprod{\weightfunc(\State
  \cdot)}{\treateff(\State, \cdot)} \Big) + \int_{\actionspace}
\Exs_{\probxstar} \Big[ \frac{\weightfunc^2(\State, \action)
  }{\propscore(\State, \action)} \treatsig^2(\State, \action) \Big]
d \basemsr(\action),
\end{align}
\end{subequations}
where $\sigma^2(\state, \action) \mydefn \var( \outcome \mid \state,
\action)$ is the conditional variance~\eqref{EqnDefnCondVariance} of
the outcome.
See~\Cref{app:sec:proof-basic-exp-and-var-of-general-form} for details
of this derivation.

%%%%%%%%%%%%%%%%%%%%%%%%%%%%%%%%%%%%%%%%%%%%%%%%%%%%%%%%%%%%%%%%%%%%%%%%%%%%

\subsubsection{A class of two-stage procedures}

The preceding set-up naturally leads to a broad class of two-stage
procedures, which we define and analyze here.  Since the treatment
effect $\treateff$ is unknown, the optimal function $\fstar$ from
equation~\eqref{eq:defn-fstar-function} is also unknown to us.  A
natural approach, then, is the two-stage one: (a) compute an estimate
$\muhat$ using part of the data; and then (b) substitute this estimate
in equation~\eqref{eq:defn-fstar-function} so as to construct an
approximation to the ideal estimator $\tauhatgen{\numobs}{\fstar}$.  A
standard cross-fitting approach (e.g., ~\cite{chernozhukov2018double})
allows one to make full use of data while avoiding the
self-correlation bias. \\

\medskip

In more detail, we first split the data into two disjoint subsets
$\datablock_1 \mydefn(\State_i, \Action_i, \outcome_i)_{i =
  1}^{\numobs / 2}$ and $\datablock_2 \mydefn(\State_i, \Action_i,
\outcome_i)_{i = \numobs / 2 + 1}^{\numobs}$. We then perform the
following two steps:
\begin{subequations}
\label{eq:estimator-framework}
\paragraph{Step I:} For $j \in \{1,2\}$, compute an estimate
$\muhat_{\numobs / 2}^{(j)}$ of $\treateff$ using the data subset
$\datablock_j$, and compute
\begin{align}
\label{eq:defn-fhat}    
\fhat_{\numobs/2}^{(j)}(\state, \action) \mydefn
\frac{\weightfunc(\state, \action) \muhat_{\numobs/2}^{(j)}(\state,
  \action)}{\propscore(\state, \action)} - \actinprod{\weightfunc
(\state, \cdot)}{\muhat_{\numobs / 2}^{(j)}(\state, \cdot)}.
\end{align}
\paragraph{Step II:} Use the auxiliary functions
$\fhat_{\numobs/2}^{(1)}$ and $\fhat_{\numobs / 2}^{(2)}$ to construct
the estimate
\begin{align}
\label{eq:split-plugin-stage}     
\tauhat_{\numobs} \mydefn \frac{1}{\numobs} \sum_{i = 1}^{\numobs/2}
\Big\{ \frac{\weightfunc(\State_i, \Action_i)}{\propscore(\State_i,
  \Action_i)} \outcome_i - \fhat_{\numobs / 2}^{(2)}(\State_i,
\Action_i) \Big\} + \frac{1}{\numobs} \sum_{i = \numobs / 2 +
  1}^{\numobs} \Big\{ \frac{\weightfunc(\State_i,
  \Action_i)}{\propscore(\State_i, \Action_i)} \outcome_i -
\fhat_{\numobs / 2}^{(1)}(\State_i, \Action_i) \Big\}.
\end{align}
\end{subequations}
As described, these two steps should be understood as defining a
meta-procedure, since the choice of auxiliary estimator
$\muhat_{\numobs / 2}^{(j)}$ can be arbitrary.\\

The main result of this section is a non-asymptotic upper bound on the
MSE of any such two-stage estimator.  It involves the \emph{weighted
$L^2$-norm} $\weightednorm{\cdot}$ given by
\begin{subequations}
\begin{align}
\label{EqnWeighted}  
  \weightednorm{h}^2 \mydefn \int_{\actionspace} \Exs_{\probxstar} \Big[
    \frac{\weightfunc^2(\State, \action)}{\propscore(\State, \action)}
    h^2(\State, \action) \Big] d \basemsr(\action),
\end{align}
which plays a fundamental role in both upper and lower bounds for the
problem.  With this notation, we have:
\begin{theorem}
\label{thm:upper-sample-split-general}
For any estimator $\muhat_{\numobs/2}$ of the treatment effect, the
two-stage estimator~\eqref{eq:estimator-framework} has MSE bounded as
\begin{align}
\label{EqnSimpleBound}  
\Exs \Big[ \abss{\tauhat_\numobs - \taustar}^2 \Big] \leq
\frac{1}{\numobs} \Big \{ \vstar^2 + 2 \Exs \big[
  \weightednorm{\muhat_{\numobs / 2} - \treateff}^2 \big] \Big \}.
\end{align}
\end{theorem}
\noindent See~\Cref{subsec:proof-upper-sample-split-general} for the
proof of this claim.  \\

Note that the upper bound~\eqref{EqnSimpleBound} consists of two
terms, both of which have natural intepretations.  The first term
$\vstar^2$ corresponds to the asymptotically efficient
variance~\eqref{eq:defn-optimal-variance}; in terms of the weighted
norm~\eqref{EqnWeighted}, it has the equivalent expression
\begin{align}
\vstar^2 & = \var \Big( \actinprod{\weightfunc(\State
  \cdot)}{\treateff(\State, \cdot)} \Big) +
\weightednorm{\treatsig}^2.
\end{align}
\end{subequations}
The second term corresponds to twice the average estimation error
$\Exs \big[ \weightednorm{\muhat_{\numobs / 2} - \treateff}^2 \big]$,
again measured in the weighted squared norm~\eqref{EqnWeighted}.
Whenever the treatment effect can be estimated consistently---so that
this second term is vanishing in $\numobs$---we see that the
estimator $\tauhat_{\numobs}$ is asymptotically efficient, as is known
from past work~\cite{chernozhukov2018double}.  Of primary interest to
us is the guidance provided by the bound~\eqref{EqnSimpleBound} in the
finite sample regime: in particular, in order to minimize this upper
bound, one should construct estimators $\muhat$ of the treatment
effect that are optimal in the weighted norm~\eqref{EqnWeighted}.

\subsection{Some non-asymptotic analysis}

With this general result in hand, we now propose some explicit
two-stage procedures that can be shown to be finite-sample optimal.
We begin by introducing the classical idea of an oracle inequality,
and making note of its consequences when combined
with~\Cref{thm:upper-sample-split-general}. We then analyze a class of
non-parametric weighted least-squares estimators, and prove that they
satisfy an oracle inequality of the desired type.

\subsubsection{Oracle inequalities and finite-sample bounds}

At a high level,~\Cref{thm:upper-sample-split-general} reduces our
problem to an instance of non-parametric regression, albeit one
involving the weighted norm $\weightednorm{\cdot}$ from
equation~\eqref{EqnWeighted}.  In non-parametric regression, there are
many methods known to satisfy an attractive ``oracle'' property (e.g.,
see the books~\cite{tsybakov2008introduction,wainwright2019high}).  In
particular, suppose that we construct an estimate $\muhat$ that takes
values in some function class $\funcClass$.  It is said to satisfy an
\emph{oracle inequality} for estimating $\treateff$ in the norm
$\weightednorm{\cdot}$ if
\begin{align}
\label{EqnOracle}  
\Exs \big[\weightednorm{\muhat - \treateff}^2 \big] & \leq c
\inf_{\plainmu \in \funcClass} \Big \{ \weightednorm{\plainmu -
  \treateff}^2 + \delta^2_\numobs(\plainmu; \funcClass) \Big \}
\end{align}
for some universal constant $c \geq 1$.  Here the functional $\plainmu
\mapsto \delta^2_\numobs(\plainmu; \funcClass)$ quantifies the
$\weightednorm{\cdot}^2$-error associated with estimating some
function $\plainmu \in \funcClass$, whereas the quantity
$\weightednorm{\plainmu - \treateff}^2$ is the squared approximation
error, since the true function $\treateff$ need not belong to the
class.  We note that the oracle inequality stated here is somewhat
more refined than the standard one, since we have allowed the
estimation error to be instance-dependent (via its dependence on the
choice of $\plainmu$).

Given an estimator $\muhat$ that satisfies such an oracle inequality,
an immediate consequence of~\Cref{thm:upper-sample-split-general} is
that the associated two-stage estimator of $\taustar \equiv
\tau(\probInstance)$ has MSE upper bounded as
\begin{align}
\Exs \big[ \abss{\tauhat_\numobs - \taustar}^2 \big] \leq
\frac{1}{\numobs} \Big ( \vstar^2 + 2 c \inf_{\plainmu \in \funcClass}
\Big \{ \weightednorm{\plainmu - \treateff}^2 +
\delta^2_\numobs(\plainmu; \funcClass) \Big \} \Big ).
\end{align}
This upper bound is explicit, and given some assumptions on the
approximability of the unknown $\treateff$, we can use to it choose
the ``complexity'' of the function class $\funcClass$ in a
data-dependent manner.  See~\Cref{SecExamples} for discussion and
illustration of such choices for different function classes.

\subsubsection{Oracle inequalities for non-parametric weighted least-squares}

Based on the preceding discussion, we now turn to the task of
proposing a suitable estimator of $\treateff$, and proving that it
satisfies the requisite oracle inequality~\eqref{EqnOracle}.  Let
$\funcClass$ be a given function class used to approximate the
treatment effect $\treateff$.  Given our goal of establishing bounds
in the weighted norm~\eqref{EqnWeighted}, it is natural to analyze the
\emph{non-parametric weighted least-squares} estimate
\begin{align}
\label{eq:least-square-estimator-in-stage-1}  
\muhat_\maux \mydefn \arg\min_{\plainmu \in \funcClass} \Big \{
\frac{1}{\maux} \sum_{i = 1}^\maux \frac{\weightfunc^2(\State_i,
  \Action_i)}{\propscore^2(\State_i, \Action_i)} \big \{
\plainmu(\State_i, \Action_i) - \outcome_i \big \}^2 \Big \},
\end{align}
where $\{(\State_i, \Action_i, \outcome_i) \}_{i=1}^\maux$ constitute
an observed collection of state-action-outcome triples.

Since the pairs $(\State, \Action)$ are drawn from the distribution
$\probxstar(\state) \propscore(\state, \action)$, our choice of
weights ensures that
\begin{align*}
\Exs \Big[\frac{\weightfunc^2(\State, \Action)}{\propscore^2(\State,
    \Action)} \big \{ \plainmu(\State, \Action) - \outcome \big \}^2
  \Big] & = \weightednorm{\plainmu - \treateff}^2 + \Exs \Big[
  \frac{\weightfunc^2(\State, \Action)}{\propscore^2(\State, \Action)}
  \sigma^2(\State, \Action) \Big].
\end{align*}
so that (up to a constant offset), we are minimizing an unbiased
estimate of $\weightednorm{\plainmu - \treateff}^2$.

\noindent In our analysis, we impose some natural conditions on the
function class:
\myassumption{CC}{assume:convset}{ The function class $\funcClass$ is
  a \emph{convex and compact} subset of the Hilbert space
  $\Ltwospace_\omega$.}
\noindent We also require some tail conditions on functions $h$ that
belong to the difference set
\begin{align*}
\DelF \defn \{f_1 - f_2 \mid f_1, f_2 \in \Fclass \}.
\end{align*}
There are various results in the non-parametric literature that rely
on functions being uniformly bounded, or satisfying other sub-Gaussian
or sub-exponential tail conditions
(e.g.,~\cite{wainwright2019high}). Here we instead leverage the less
restrictive learning-without-concentration framework of
Mendelson~\cite{mendelson2015learning}, and require that the following
\emph{small probability condition} holds:
\myassumption{SB}{assume:small-ball}{There exists a pair
  $(\smallballcon, \smallballprob)$ of positive scalars such that
  \begin{align}
\label{EqnSmallBall}    
  \Prob \Big[ \abss{\tfrac{\weightfunc(\State,
        \Action)}{\propscore(\State, \Action)} h(\State, \Action)}
    \geq \smallballcon \weightednorm{h} \Big] & \geq \smallballprob
  \qquad \mbox{for all $h \in \DelF$.}
 \end{align}
}
If we introduce the shorthand $\htil = \tfrac{\weightfunc}{\propscore}
h$, then condition~\eqref{EqnSmallBall} can be written equivalently as
$\Prob \big[ |\htil(\State, \Action)| \geq \smallballcon \|\htil\|_2
  \big] \geq \smallballprob$, so that it is a standard small-ball
condition on the function $\htil$; see the
papers~\cite{koltchinskii2015bounding,mendelson2015learning} for more
background. \\

\medskip

As with existing theory on non-parametric estimation, our risk bounds
are determined by the suprema of empirical processes, with
``localization'' so as to obtain optimal rates. Given a function class
$\Hclass$ and a positive integer $\maux$, we define the
\emph{Rademacher complexities}
\begin{subequations}
  \begin{align}
\label{eq:defn-rade-effective-noise}     
\SquaredRade^2_\maux(\Hclass) & \mydefn \Exs \Big[ \sup_{f \in
    \Hclass} \Big\{ \frac{1}{\maux} \sum_{i = 1}^\maux \frac{\rade_i
    \weightfunc^2(\State_i, \Action_i)}{\propscore^2(\State_i,
    \Action_i)} \big( \outcome_i - \treateff(\State_i, \Action_i)
  \big) f(\State_i, \Action_i) \Big\}^2 \Big], \quad \mbox{and} \\
\label{eq:defn-rade-intrinsic}
\radeComplexity_\maux(\Hclass) & \mydefn \Exs \Big[ \sup_{f \in
    \Hclass} \frac{1}{\maux} \sum_{i = 1}^\maux \frac{\rade_i
    \weightfunc(\State_i, \Action_i)}{\propscore(\State_i, \Action_i)}
  f(\State_i, \Action_i) \Big],
\end{align}
\end{subequations}
where $(\rade_i)_{i = 1}^\numobs$ are $\mathrm{i.i.d.}$ Rademacher
random variables independent of the data.

With this set-up, we are now ready to state some oracle inequalities
satisfied by the weighted least-squares
estimator~\eqref{eq:least-square-estimator-in-stage-1}.  As in our
earlier statement~\eqref{EqnOracle}, these bounds are indexed by some
$\plainmu \in \funcClass$, and our risk bound involves the solutions
\begin{subequations}
\begin{align}
\label{eq:defn-critical-radius-squared}     
\tfrac{1}{\radone} \; \SquaredRade_\maux \big((\funcClass - \plainmu)
\cap \ball_\omega(\radone) \big) & \leq \radone \, , \qquad \mbox{and}
\\
\label{eq:defn-critical-radius-plain}
\tfrac{1}{\radtwo} \; \radeComplexity_\maux \big((\funcClass -
\plainmu) \cap \ball_\omega(\radtwo) \big) & \leq \tfrac{\smallballcon
  \smallballprob}{32}.
\end{align}
\end{subequations}
Let $\radone_\maux(\plainmu)$ and $\radtwo_\maux(\plainmu)$,
respectively, be the smallest non-negative solutions to these
inequalities; see~\Cref{prop:existence-critical-radii}
in~\Cref{app:subsec-proof-existence-critical-radii} for their
guaranteed existence.
\begin{theorem}
\label{cor:least-sqr-split-estimator}
Under the convexity/compactness condition~\ref{assume:convset} and
small-ball condition~\ref{assume:small-ball}, the two-stage
estimate~\eqref{eq:estimator-framework} based on the non-parametric
least-squares estimate~\eqref{eq:least-square-estimator-in-stage-1}
satisfies the oracle inequality
\begin{subequations}
\begin{align}
\label{eq:cor-least-sqr-split-estimator-risk}  
\Exs \big[ \big(\tauhat_\numobs - \taustar \big)^2 \big] & \leq
\frac{1}{\numobs} \Big \{ \vstar^2 + c \inf_{\plainmu \in \funcClass}
\Big( \weightednorm{\plainmu - \treateff}^2 +
\delta^2_\numobs(\plainmu; \funcClass) \Big) \Big \}
\end{align}
where the instance-dependent estimation error is given by
\begin{align}
\delta^2_\numobs(\plainmu; \funcClass) & \mydefn \radone_{\numobs /
  2}^2(\plainmu) + \radtwo_{\numobs / 2}^2(\plainmu) + e^{- c'
  \numobs} \; \diameter^2_\omega(\funcClass \cup \{\treateff\}),
\end{align}
\end{subequations}
for a pair $(c, c')$ of constants depending only on the small-ball
parameters $(\smallballcon, \smallballprob)$.
\end{theorem}
\noindent See~\Cref{subsec:proof-least-sqr-split-estimator} for the
proof of this theorem.

\medskip

A few remarks are in order.  The
bound~\eqref{eq:cor-least-sqr-split-estimator-risk} arises by
combining the general bound from~\Cref{thm:upper-sample-split-general}
with an oracle inequality that we establish for the weighted
least-squares
estimator~\eqref{eq:least-square-estimator-in-stage-1}. Compared to
the efficient variance $\vstar^2$, this bound includes three
additional terms: (i) the critical radii
$\radone_{\numobs/2}(\plainmu)$ and $\radtwo_{\numobs/2}(\plainmu)$
that solve the fixed point equations; (ii) the approximation error
under $\weightednorm{\cdot}$ norm; and (iii) an exponentially decaying
term. For any fixed function class $\Fclass$, if we take limits as the
sample size $\numobs$ tends to infinity, we see that the asymptotic
variance of $\tauhat_\numobs$ takes the form
\begin{align*}
\vstar^2 + c \, \inf_{\plainmu \in \Fclass} \weightednorm{\plainmu -
  \treateff}^2.
\end{align*}
Consequently, the estimator may suffer from an efficiency loss
depending on how well the unknown treatment effect $\treateff$ can be
approximated (in the weighted norm) by a member of $\Fclass$.  When
the outcome noise $\outcome_i - \treateff (\State_i, \Action_i)$ is of
constant order, inspection of
equations~\eqref{eq:defn-critical-radius-squared}
and~\eqref{eq:defn-critical-radius-plain} reveals that---as $\numobs$
tends to infinity---the critical radius
$\radone_{\numobs/2}(\plainmu)$ decays at a faster rate than
$\radtwo_{\numobs/2}(\plainmu)$.  Therefore, the non-asymptotic excess
risk---that is, any contribution to the MSE \emph{in addition to} the
efficient variance $\vstar^2$---primarily depends on two quantities:
(a) the approximation error associated with approximating $\treateff$
using a given function class $\funcClass$, and (b) the (localized)
metric entropy of this function class.  Interestingly, both of these
quantities turn out to be information-theoretically optimal in an
instance-dependent sense. More precisely, in~\Cref{SecLower}, we show
that an efficiency loss depending on precisely the same approximation
error is unavoidable; we further show that a sample size depending on
a local notion of metric entropy is also needed for such a bound to be
valid.

%%%%%%%%%%%%%%%%%%%%%%%%%%%%%%%%%%%%%%%%%%%%%%%%%%%%%%%%%%%%%%%%%%%%%

\subsection{A simulation study}
\label{subsec:simulation}

We now describe a simulation study that helps to illustrate the elbow
effect predicted by our theory, along with the utility of using
reweighted estimators of the treatment effect.  We can model a missing
data problem by using $\Action \in \actionspace \defn \{0,1 \}$ as a
binary indicator variable for ``missingness''---that is, the outcome
$\outcome$ is observed if and only if $\Action = 1$.  Taking $\probx$
as the uniform distribution on the state space $\Xspace \defn [0, 1]$,
we take the weight function $\weightfunc(\state, \action) = \action$,
so that our goal is to estimate the quantity
$\Exs_\probx[\treateff(\State, 1)]$.  Within this subsection, we abuse
notation slightly by using $\plainmu$ to denote the function
$\plainmu(\cdot, 1)$, and similarly $\treateff$ for $\treateff(\cdot,
1)$.

We allow the treatment effect to range over the first-order Sobolev
smoothness class
\begin{align*}
\funcClass \mydefn \Big\{ f: [0, 1] \rightarrow \real \; \mid \; f(0)
= 0, ~ \vecnorm{f}{\mathbb{H}^1}^2 \mydefn \int_0^1 \big(f'(x) \big)^2
dx \leq 1 \Big\},
\end{align*}
corresponding (roughly) to functions that have a first-order
derivative $f'$ with bounded $L^2$-norm.  The function class
$\funcClass$ is a particular type of reproducing kernel Hilbert space
(cf. Example 12.19 in the book~\cite{wainwright2019high}), so it is
natural to consider various forms of kernel ridge regression.

\paragraph{Three possible estimators:}

So as to streamline notation, we let $\Sobs \subseteq \{1, \ldots,
\maux \}$ denote the subset of indices associated with observed
outcomes---that is, $\action_i = 1$ if and only if $i \in \Sobs$.  Our
first estimator follows the protocol suggested by our theory: more
precisely, we estimate the function $\treateff$ using a
\emph{reweighted} form of kernel ridge regression (KRR)
\begin{subequations}
  \begin{align}
\label{EqnWeightedKRR}    
    \muhat_{\maux, \omega} \mydefn \arg\min_{\plainmu \in \Ltwospace
      ([0, 1])} \left \{ \sum_{i \in \Sobs}
    \frac{\weightfunc^2(\State_i, 1)}{\propscore^2(\State_i, 1)} \big
    \{ \outcome_i - \plainmu (\State_i) \big \} ^2 + \lambda_\maux
    \vecnorm{\plainmu}{\mathbb{H}^1}^2 \right \},
  \end{align}
  where $\lambda_\maux \geq 0$ is a regularization parameter (to be
  chosen by cross-validation).  Let $\tauhat_{\numobs, \omega}$ be the
  output of the two-stage procedure~\eqref{eq:estimator-framework}
  when the reweighted KRR estimate is used in the first stage.

So as to isolate the effect of reweighting, we also implement the
standard (unweighted) KRR estimate, given by
\begin{align}
\label{EqnStandardKRR}    
    \muhat_{\maux, \Ltwospace} \mydefn \arg \min_{\plainmu
      \in \Ltwospace ([0, 1])} \Big\{ \sum_{i \in \Sobs} \big \{
    \outcome_i - \plainmu (\State_i) \big \}^2 + \lambda_\maux
    \vecnorm{\plainmu}{\mathbb{H}^1}^2 \Big\}.
\end{align}
Similarly, we let $\tauhat_{\numobs, \Ltwospace}$ denote the estimate
obtained by using the unweighted KRR estimate as a first-stage
quantity.

Finally, so as to provide an (unbeatable) baseline for comparison, we
compute the \emph{oracle estimate}
\begin{align}
    \tauhat_{\numobs, \mathrm{oracle}} \mydefn \frac{1}{\numobs}
    \sum_{i = 1}^\numobs \Big\{ \frac{ (\outcome_i - \treateff
      (\State_i)) \Action_i}{\propscore (\State_i, 1)} + \treateff
    (\State_i) \Big\}.
\end{align}
\end{subequations}
Here the term ``oracle'' refers to the fact that it provides an answer
given the unrealistic assumption that the true treatment effect
$\treateff$ is known.  Thus, this estimate cannot be computed based
purely on observed quantities, but instead serves as a lower bound for
calibrating.  For each of these three estimators, we compute its
$\numobs$-rescaled mean-squared error
\begin{align}
\label{eq:normalized-mse-in-simulation}  
\numobs \cdot \Exs \big[ \abss{\tauhat_{\numobs, \diamond} -
    \taustar}^2 \big] \quad \mbox{with $\diamond \in \big\{
  \omega, \Ltwospace, \mathrm{oracle} \big\}$}.
\end{align}

\paragraph{Variance functions:}
Let us now describe an interesting family of variance functions
$\sigma^2$ and propensity scores $\propscore$.  We begin by observing
that if the standard deviation function $\sigma$ takes values of the
same order as the treatment effect $\treateff$, then the simple IPW
estimator has a variance of the same order as the asymptotic efficient
limit $\vstar^2$. Thus, in order to make the problem non-trivial and
illustrate the advantage of semiparametric methods, we consider
variance functions of the following type: for a given propensity score
$\propscore$ and exponent $\gamma \in [0, 1]$, define
\begin{align}
\label{eq:sigma-choice-in-simulation}  
\sigma^2(\state, 1) & \defn \sigma_0^2 \; \big[\propscore (\state, 1)
  \big]^\gamma,
\end{align}
where $\sigma_0 > 0$ is a constant pre-factor.  Since the optimal
asymptotic variance $\vstar$ contains the term $\Exs \big[
  \tfrac{\sigma^2(\State, 1)}{\propscore (\State, 1)} \big]$, this
family leads to a term of the form
\begin{align*}
\sigma_0^2 \; \Exs \left[\frac{1}{\{ \propscore(\State, 1)\}^{1 -
      \gamma}} \right],
\end{align*}
showing that (in rough terms) the exponent $\gamma$ controls the
influence of small values of the propensity score $\propscore(\State,
1)$.  At one extreme, for $\gamma = 1$, there is no dependence on
these small values, whereas the other extreme $\gamma=0$, it will be
maximally sensitive to small values of the propensity score.

%%%%%%%%%%%%%%%%%%%%%%%%%%%%%%%%%%%%%%%%%%%%%%%%%%%%%%%%%%%%%%%%%%%%%%%%%%%%%%%%%%%%%%
\paragraph{Propensity and treatment effect:}

We consider the following two choices of propensity scores
\begin{subequations}
  \begin{align}
\propscore_1(\state, 1) & \defn \tfrac{1}{2} - \big( \tfrac{1}{2} -
\propscore_{\min}\big) \sin( \pi \state), \quad \mbox{and} \\
\propscore_2(\state, 1) & \defn \tfrac{1}{2} - \{ \tfrac{1}{2}
-\propscore_{\min}) \sin ( \pi \state / 2 ) \},
  \end{align}
\end{subequations}
where $\propscore_{\min} \defn 0.005$.  At the same time, we take the
treatment effect to be the ``tent'' function
\begin{align}
  \label{EqnDefnTent}
\treateff(\state) & = \frac{1}{2} - \abss{\state - \tfrac{1}{2}}
\qquad \mbox{for $\state \in [0,1]$.}
\end{align}

Let us provide the rationale for these choices.  Both propensity score
functions take values in the interval $[\propscore_{\min}, 0.5]$, but
achieve the minimal value $\propscore_{\min}$ at different points
within this interval: $x = 1/2$ for $\propscore_1$ and at $x= 1$ for
$\propscore_2$.  Now observe that for the missing data problem, the
risk of the na\"{i}ve IPW estimator~\eqref{eq:ipw-estimator-simple}
contains a term of the form $\Exs
[\tfrac{\treateff(\State)^2}{\propscore(\State, 1)}]$.  Since our
chosen treatment effect function~\eqref{EqnDefnTent} is maximized at
$\state = 1/2$, this term is much larger when we set $\propscore =
\propscore_1$, which is minimized at $\state = 1/2$.  Thus, the
propensity score $\propscore_1$ serves as a ``hard'' example.  On the
other hand, the treatment effect is minimized at $\state = 1$, where
$\propscore_2$ achieves its minimum, so that this represents an
``easy'' example.
\begin{figure}[!h]
\centering
    \begin{subfigure}{\figwidth\linewidth}
    \centering
    \includegraphics[width=\linewidth]{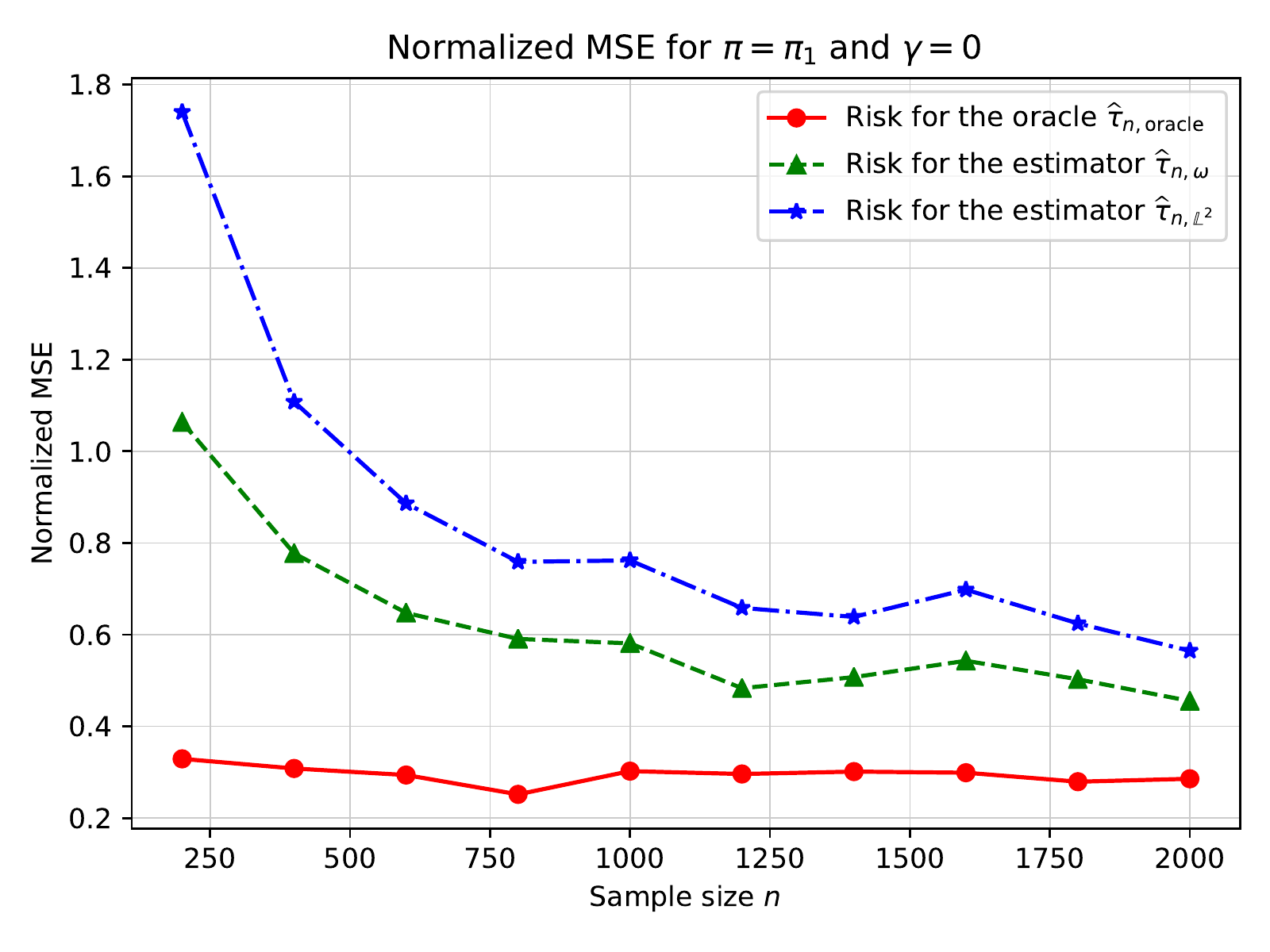}
    \caption{$\gamma = 0, ~ \propscore = \propscore_1$}
    \end{subfigure}
     \begin{subfigure}{\figwidth\linewidth}
    \centering
    \includegraphics[width=\linewidth]{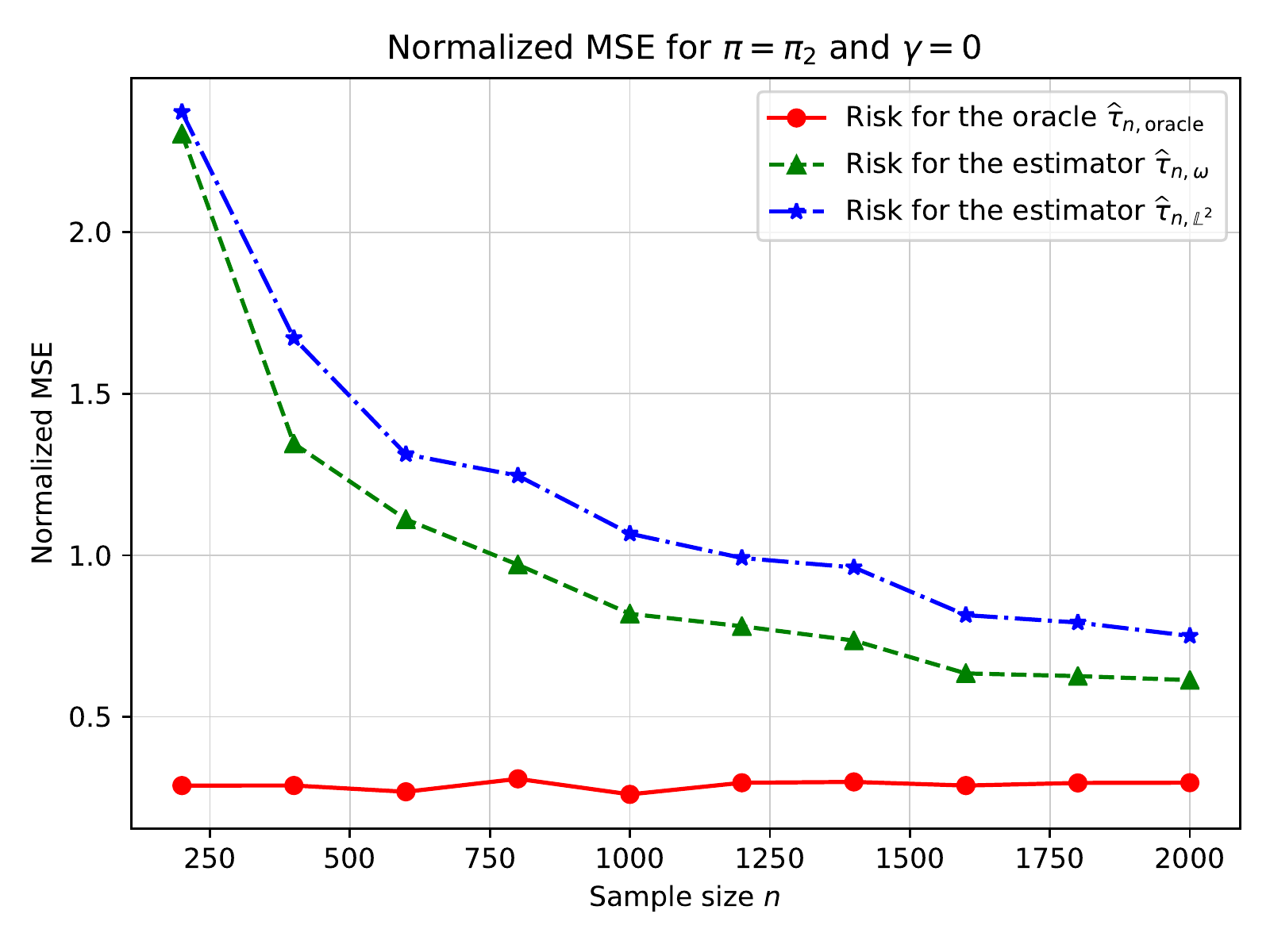}
    \caption{$\gamma = 0, ~ \propscore = \propscore_2$}
    \end{subfigure}
      \begin{subfigure}{\figwidth\linewidth}
    \centering
    \includegraphics[width=\linewidth]{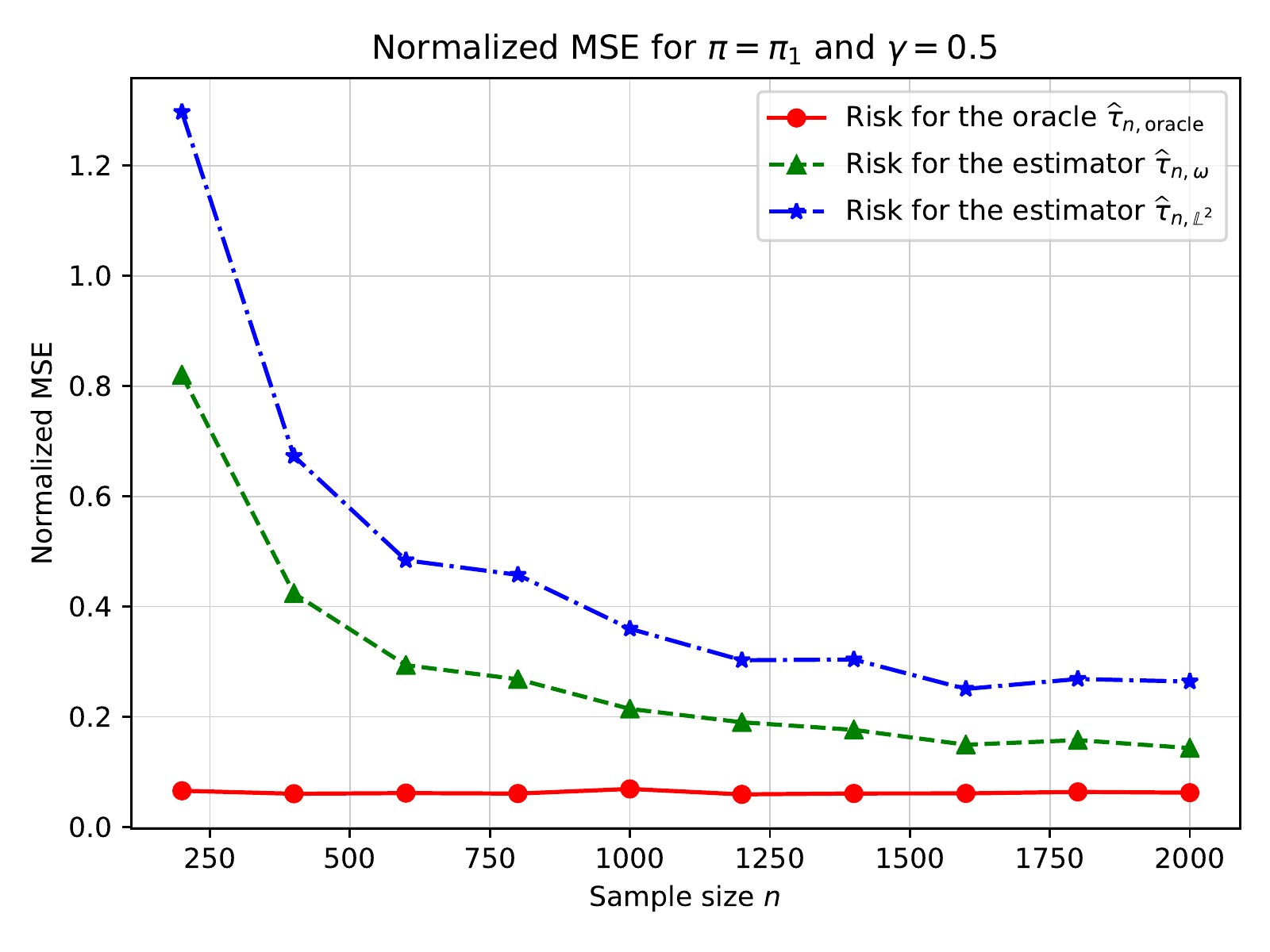}
    \caption{$\gamma = 0.5, ~ \propscore = \propscore_1$}
    \end{subfigure}
     \begin{subfigure}{\figwidth\linewidth}
    \centering
    \includegraphics[width=\linewidth]{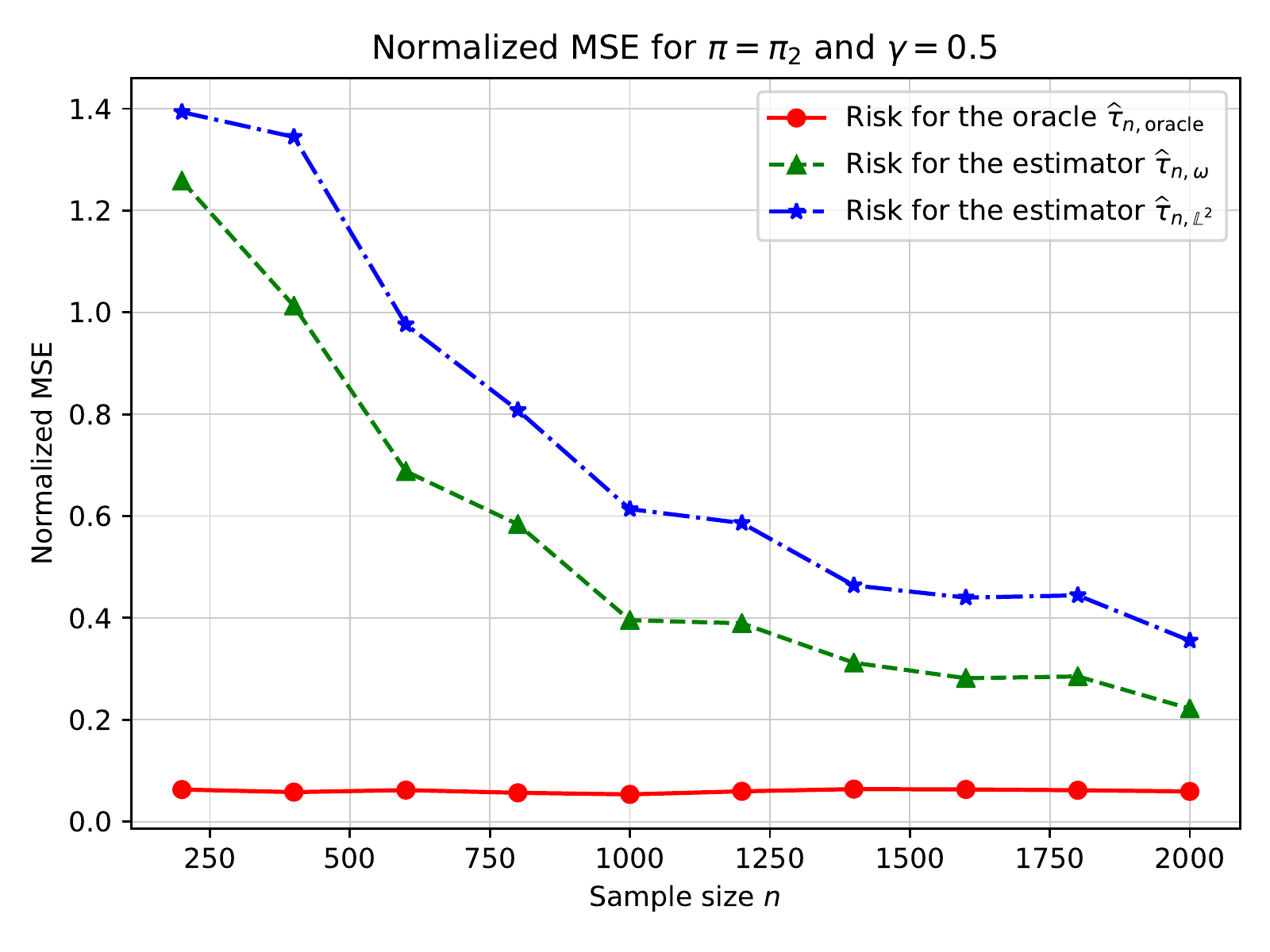}
    \caption{$\gamma = 0.5, ~ \propscore = \propscore_2$}
    \end{subfigure}
      \begin{subfigure}{\figwidth\linewidth}
    \centering
    \includegraphics[width=\linewidth]{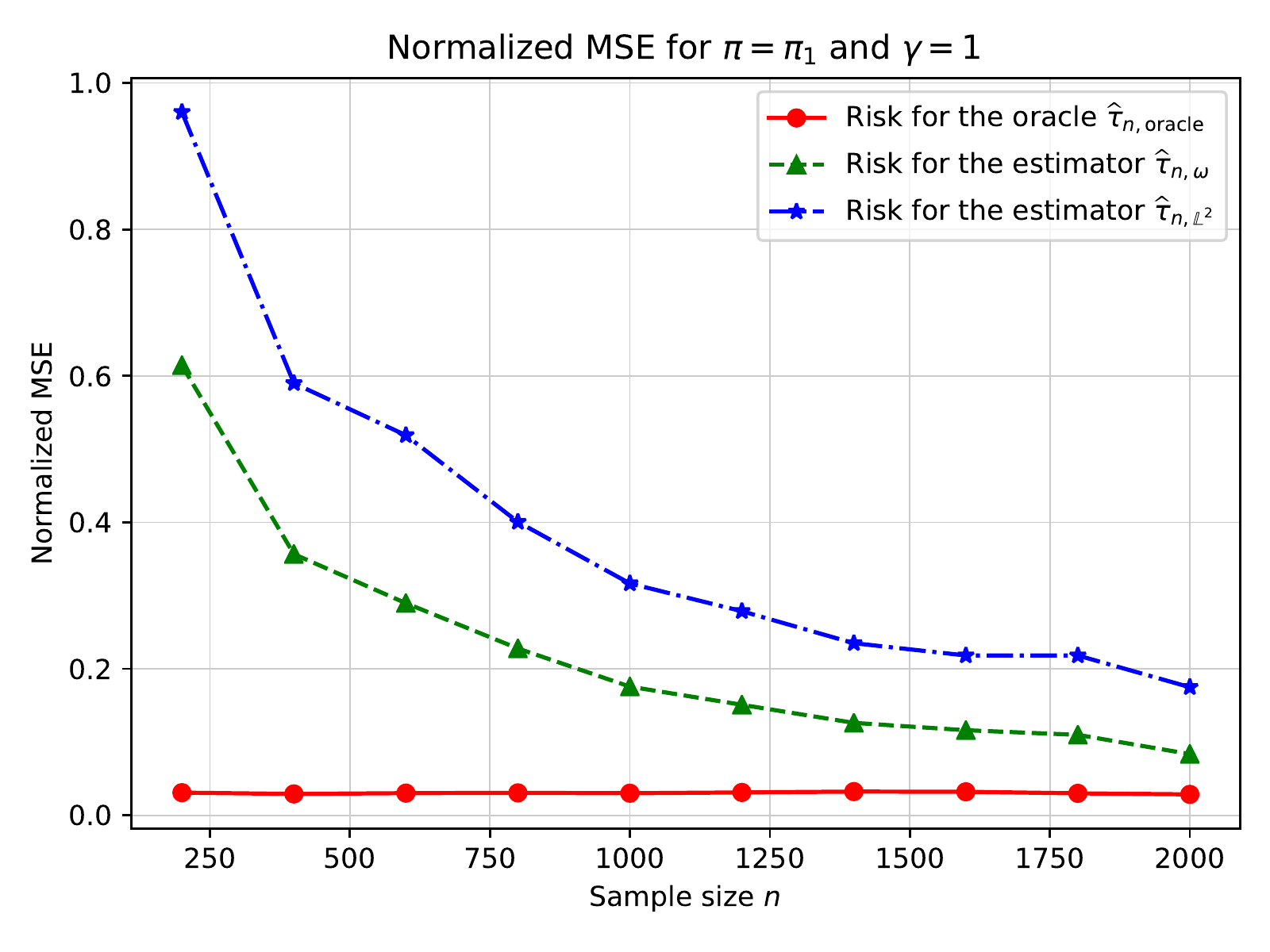}
    \caption{$\gamma = 1, ~ \propscore = \propscore_1$}
    \end{subfigure}
     \begin{subfigure}{\figwidth\linewidth}
    \centering
    \includegraphics[width=\linewidth]{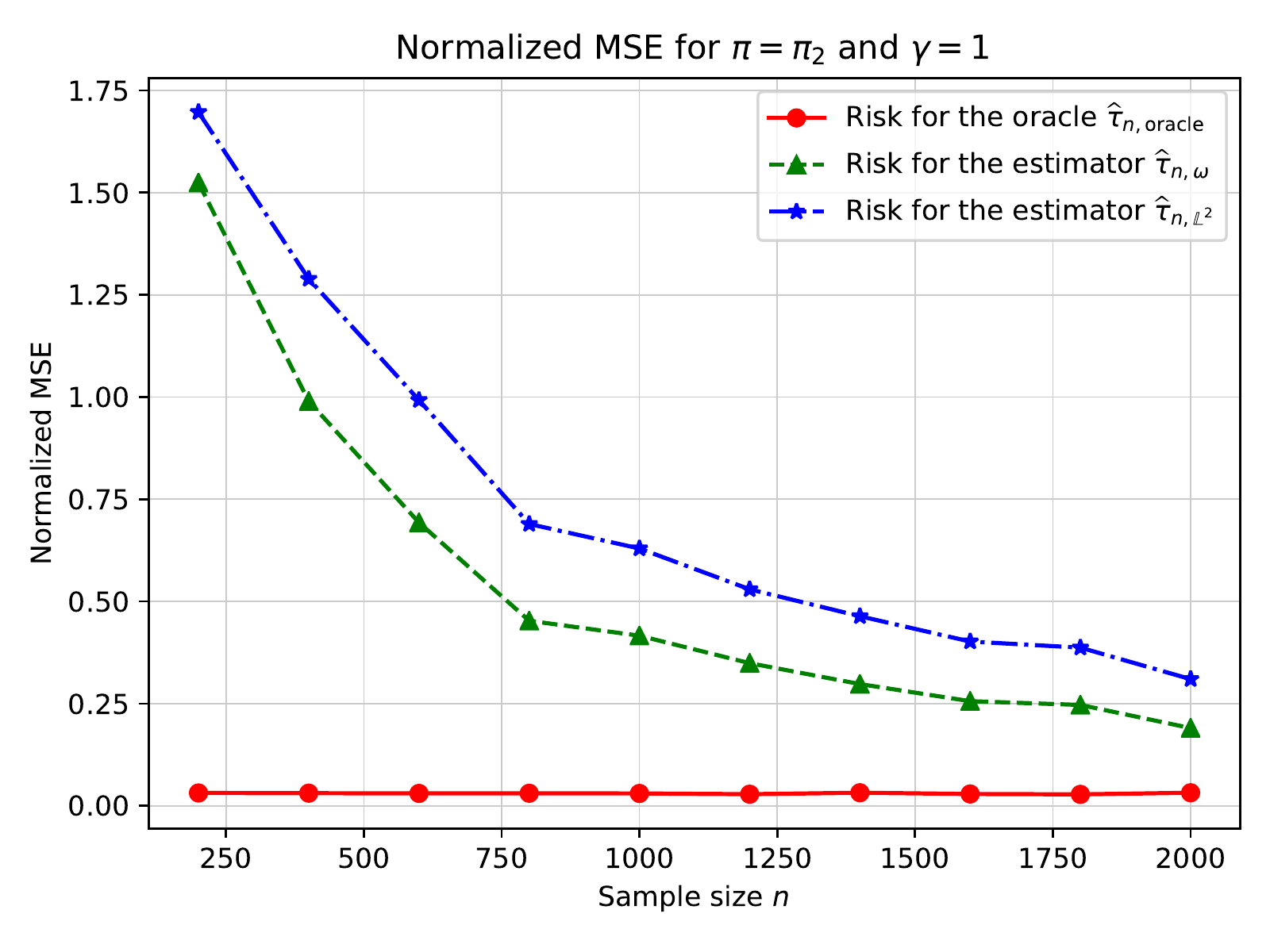}
    \caption{$\gamma = 1, ~ \propscore = \propscore_2$}
    \end{subfigure}
    \caption{Plots of the normalized MSE $\numobs \cdot \Exs
      [\abss{\tauhat_{\numobs, \diamond} - \taustar}]$ for $\diamond
      \in \{\omega, \Ltwospace, \mathrm{oracle}\}$ versus the sample
      size.  Each marker corresponds to a Monte Carlo estimate based
      on the empirical average of $1000$ independent runs.  As
      indicated in the figure titles, panels (a--f) show the
      normalized MSE of estimators for combinations of parameters:
      exponent $\gamma \in \{0, 0.5, 1 \}$ in the top, middle and
      bottom rows respectively, and propensity scores $\propscore_1$
      and $\propscore_2$ in the left and right columns,
      respectively. For each run, we used $5$-fold cross validation to
      choose the value of regularization parameter $\lambda_\numobs
      \in [10^{-1}, 10^2]$.}
    \label{fig:simulation}
\end{figure}

\paragraph{Simulation set-up and results:}

For each choice of exponent $\gamma \in \{0, 0.5, 1 \}$ and each
choice of propensity score $\propscore \in \{\propscore_1,
\propscore_2 \}$, we implemented the reweighted estimator
$\tauhat_{\numobs, \omega}$, the standard estimator
$\tauhat_{\numobs, \Ltwospace}$ and the oracle estimator
$\tauhat_{\numobs, \mathrm{oracle}}$.  For each simulation, we varied
the sample size over the range $\numobs \in \{2000, 4000, \ldots,
18000, 20000\}$.  For each run, we use $5$-fold cross validation to
choose the value of regularization parameter $\lambda_\numobs \in
[10^{-1}, 10^2]$.  For each estimator and choice of simulation
parameters, we performed a total of $1000$ independent runs, and used
them to form a Monte Carlo estimate of the true MSE.

\Cref{fig:simulation} provides plots of the $\numobs$-rescaled
mean-squared error~\eqref{eq:normalized-mse-in-simulation} versus the
sample size $\numobs$ for each of the three estimators in each of the
six set-ups (three choices of $\gamma$, crossed with two choices of
propensity score).  In order to interpret the results, first note that
consistent with the classical theory, the $\numobs$-rescaled MSE of
the oracle estimator stays at a constant level for different sample
sizes.  (There are small fluctuations, to be expected, since the
quantity $\tauhat_{\numobs, \mathrm{oracle}}$ itself is an empirical
average over $\numobs$ samples.) Due to the design of our problem
instances, the na\"{i}ve IPW estimator~\eqref{eq:ipw-estimator-simple}
has much larger mean-squared error; in fact, it is so large that we do
not include it in the plot, since doing so would change the scaling of
the vertical axis. On the other hand, both the reweighted KRR
two-stage estimate $\tauhat_{\numobs, \omega}$ and the standard KRR
two-stage estimate $\tauhat_{\numobs, \Ltwospace}$ exhibit the elbow
effect suggested by our theory: when the sample size is relatively
small, the high-order terms in the risk dominate, yielding a large
normalized MSE.  However, as the sample size increases, these
high-order terms decay at a faster rate, so that the renormalized MSE
eventually converges to the asymptotically optimal limit (i.e., the
risk of the oracle estimator $\tauhat_{\numobs, \mathrm{oracle}}$). In
all our simulation instances, the weighted estimator
$\tauhat_{\numobs, \omega}$, which uses a \emph{reweighted}
non-parametric least-squares estimate in the first stage, outperforms
the standard two-stage estimator $\tauhat_{\numobs, \Ltwospace}$ that
does not reweight the objective.  Again, this behavior is to be
expected from our theory: in our bounds, the excess MSE due to errors
in estimating the treatment effect is measured using the weighted
norm.

%%%%%%%%%%%%%%%%%%%%%%%%%%%%%%%%%%%%%%%%%%%%%%%%%%%%%%%%%%%%%%%%%%%%%%%%%%%%%%%%%%%%%%%

\subsection{Implications for particular models}
\label{SecExamples}

We now return to our theoretical thread, and illustrate the
consequences of our general theory for some concrete classes of
outcome models.

\subsubsection{Standard linear functions}
\label{sec:consequence:linear}
We begin with the simplest case, namely that of linear outcome
functions. For each $j = 1, \ldots, \usedim$, let $\phi_j: \Xspace
\times \actionspace \rightarrow \real$ be a basis function, and
consider functions that are linear in this
representation---viz. $f_\theta(\state, \action) = \sum_{j=1}^\usedim
\theta_j \phi_j(\state, \action)$ for some parameter vector $\theta
\in \real^\usedim$.  For a radius\footnote{We introduce this radius
only to ensure compactness; in our final bound, the dependence on
$\RadTwo$ is exponentially decaying, so that it is of little
consequence.}  $\RadTwo > 0$, we define the function class
\begin{align*}
\funcClass & \mydefn \Big \{ f_\theta \mid \: \|\theta\|_2 \leq
\RadTwo \Big \}.
\end{align*}
Our result assumes the existence of the following moment matrices:
\begin{align*}
    \Sigma \mydefn \Exs \left[ \frac{\weightfunc^2(\State,
        \Action)}{\propscore^2(\State, \Action)} \phi(\State, \Action)
      \phi (\State, \Action)^\top \right], \quad \mbox{and} \quad
    \Gamma_\sigma \mydefn \Exs \left[ \frac{\weightfunc^4(\State,
        \Action)}{\propscore^4(\State, \Action)} \sigma^2(\State,
      \Action) \phi(\State, \Action) \phi(\State, \Action)^\top
      \right].
\end{align*}
With this set-up, we have:
\begin{corollary}
\label{cor:example-linear}
Under the small-ball condition~\ref{assume:small-ball}, given a sample
size satisfying the lower bound \mbox{$\numobs \geq c_0 \big\{ \usedim
  + \log (\RadTwo \lammax (\Sigma)) \big\}$,} the estimate
$\tauhat_\numobs$ satisfies the bound
\begin{align}
  \Exs \Big[ \abss{\tauhat_\numobs - \taustar}^2 \Big] & \leq
  \frac{1}{\numobs} \left \{ \vstar^2 + c \inf_{\plainmu \in
    \funcClass} \weightednorm{\plainmu - \treateff}^2 \right \} +
  \frac{c}{\numobs^2} \trace \Big( \Sigma^{-1} \Gamma_\sigma \Big),\label{eq:cor-example-linear}
\end{align}
where the constants $(c_0, c)$ depend only on the small-ball
parameters $(\smallballcon, \smallballprob)$.
\end{corollary}
\noindent See Appendix~\ref{app:subsec-proof-cor-example-linear} for
the proof of this corollary.

A few remarks are in order. First, Corollary~\ref{cor:example-linear}
is valid in the regime $\numobs \gtrsim \usedim$, and the higher order
term scales as $\order{\usedim / \numobs^2}$ in the worst
case. Consequently, the optimal efficiency $\vstar^2 +
\inf_{\plaintreateff \in \funcClass} \weightednorm{\plaintreateff -
  \treateff}^2$ is achieved when the sample size $\numobs$ exceeds the
dimension $\usedim$ for linear models.\footnote{We note in passing
that the constant pre-factor $c$ in front of the term $\numobs^{-1}
\inf_{\plainmu \in \funcClass} \weightednorm{\plainmu - \treateff}^2$
can be reduced to $1$ using the arguments in
Corollary~\ref{thm:normal-approx-least-sqr}.}

It is worth noting, however, that in the well-specified case with
$\treateff \in \funcClass$, the high-order term $\tfrac{c}{\numobs^2}
\trace (\Sigma^{-1} \Gamma_\sigma)$ in
equation~\eqref{eq:cor-example-linear} does not necessarily correspond
to the optimal risk for estimating the function $\treateff$ under the
weighted norm $\weightednorm{\cdot}$. Indeed, in order to estimate the
function $\treateff$ with the optimal semi-parametric efficiency under
a linear model, an estimator that reweights samples with the function
$\tfrac{1}{\sigma^2(\State_i, \Action_i)}$ is the optimal choice,
leading to a higher order term of the form $\frac{c}{\numobs^2}\trace
(\widebar{\Sigma}^{-1} \Sigma)$, where $\widebar{\Sigma} \mydefn \Exs
\left[ \frac{1}{\sigma^2(\State, \Action)} \phi(\State, \Action)
  \phi(\State, \Action)^\top \right]$.\footnote{Note that
$\begin{bmatrix} \bar{\Sigma} & \Sigma \\ \Sigma&
  \Gamma_\sigma \end{bmatrix} = \cov \Big( \begin{bmatrix}
  \sigma(\State, \Action)^{-1} \phi (\State, \Action)
  \\ \tfrac{\weightfunc (\State, \Action)}{\propscore (\State,
    \Action)} \sigma (\State, \Action) \phi (\State,
  \Action) \end{bmatrix} \Big) \succeq 0$. Taking the Schur complement
we obtain that $\Gamma_\sigma \succeq \Sigma \widebar{\Sigma}^{-1}
\Sigma$, which implies that $\trace(\Sigma^{-1} \Gamma_\sigma) \geq
\trace (\widebar{\Sigma}^{-1} \Sigma)$.}  In general, the question of
achieving optimality with respect to both the approximation error and
high-order terms (under the $\weightednorm{\cdot}$-norm) is currently
open.

%%%%%%%%%%%%%%%%%%%%%%%%%%%%%%%%%%%%%%%%%%%%%%%%%%%%%%%%%%%%%%%%%

\subsubsection{Sparse linear models}

Now we turn to sparse linear models for the outcome function. Recall
the basis function set-up from~\Cref{sec:consequence:linear}, and
the linear functions $f_\theta = \sum_{j=1}^\usedim \theta_j
\phi_j(\state, \action)$.  Given a radius $\RadOne > 0$, consider the
class of linear functions induced by parameters with bounded
$\ell_1$-norm---viz.
\begin{align*}
  \funcClass \mydefn \Big\{ f_\theta \; \mid \; \vecnorm{\theta}{1} \leq
  \RadOne \Big\}.
\end{align*}
Sparse linear models of this type arise in many applications, and have
been the subject of intensive study (e.g., see the
books~\cite{HasTibWai15, wainwright2019high} and references therein).

We assume that the basis functions and outcome noise $\outcome -
\treateff(\State, \Action)$ satisfy the moment bounds
\begin{subequations}
\label{eq:subgaussian-in-sparse-linear-regression}
\begin{align}
\Exs \Big[ \abss{\frac{\weightfunc (\State, \Action)}{\propscore
      (\State, \Action)} \big(\outcome - \treateff (\State, \Action)
    \big) }^\ell \Big] & \leq (\varbound \sqrt{\ell})^\ell, \quad
\mbox{for any $\ell = 1, 2, \ldots$, and}\\
\max_{j = 1, \ldots, \usedim} \Exs \Big[ \abss{\frac{\weightfunc
      (\State, \Action)}{\propscore (\State, \Action)}\phi_j(\State,
    \Action)}^\ell \Big] & \leq (\subgaussian \sqrt{\ell})^\ell, \quad
\mbox{for any $\ell = 1, 2, \ldots$.}
\end{align}
\end{subequations}
Under these conditions, we have the following guarantee:
\begin{corollary}
\label{cor:example-sparse}
Under the small-ball condition~\ref{assume:small-ball} and the moment
bounds~\eqref{eq:subgaussian-in-sparse-linear-regression}, for any
sparsity level $k = 1, \ldots, \usedim$ and sample size $\numobs$ such
that \mbox{$\numobs \geq c_0 \Big\{ \tfrac{ \subgaussian^2 k
    \log(\usedim) }{\lammin (\Sigma)} + \log^2(\usedim) + \log
  (\RadOne \cdot \lammax(\Sigma)) \Big\}$,} we have
\begin{align*}
\Exs \Big[ \abss{\tauhat_\numobs - \taustar}^2 \Big] \leq
\frac{\vstar^2}{\numobs} + \frac{c}{\numobs}~ \inf_{\substack{
    \vecnorm{\thetabar}{1} = \RadOne \\ \vecnorm{\thetabar}{0} \leq
    k}} \left\{ \weightednorm{\treateff -
  \inprod{\thetabar}{\phi(\cdot, \: \cdot)}}^2 + \frac{
  \varbound^2\vecnorm{\thetabar}{0} \log (\usedim)}{\numobs} \cdot
\frac{\subgaussian^2}{\lammin (\Sigma)} \right\},
\end{align*}
where the constants $(c_0, c)$ depend only on the small ball
parameters $(\smallballcon, \smallballprob)$.
\end{corollary}
\noindent See Appendix~\ref{app:subsec-proof-cor-example-sparse} for
the proof of this corollary.

A few remarks are in order. First, the additional risk term compared
to the semiparametric efficient limit $\vstar^2 / \numobs$ is similar
to existing oracle inequalities for sparse linear regression (e.g., \S
7.3 in the book~\cite{wainwright2019high}).  Notably, it adapts to the
sparsity level of the approximating vector $\thetabar$. The complexity
of the auxiliary estimation task is characterized by the sparsity
level $\vecnorm{\thetabar}{0}$ of the target function, which appears
in both the high-order term of the risk bound and the sample size
requirement. On the other hand, note that the
$\weightednorm{\cdot}$-norm projection of the function $\treateff$ to
the set $\funcClass$ may not be sparse. Instead of depending on the
(potentially large) local complexity of such projection, the bound in
Corollary~\ref{cor:example-sparse} is adaptive to the trade-off
between the sparsity level $\vecnorm{\thetabar}{0}$ and the
approximation error $\weightednorm{\treateff - \inprod{
    \thetabar}{\phi(\cdot, \: \cdot)}}$.

%%%%%%%%%%%%%%%%%%%%%%%%%%%%%%%%%%%%%%%%%%%%%%%%%%%%%%%%%%%%%%%%%

\subsubsection{H\"{o}lder smoothness classes}

Let us now consider a non-parametric class of outcome functions.  With
state space $\Xspace = [0, 1]^{d_x}$ and action space $\actionspace =
[0, 1]^{d_a}$, define the total dimension $\pdim \defn d_x +
d_a$. Given an integer order of smoothness $k > 0$, consider the class
\begin{align*}
\funcClass_k \mydefn\Big\{ \plainmu: [0, 1]^\pdim \rightarrow \real
\; \mid \; \sup_{(\state, \action) \in [0, 1]^\pdim} |\partial^\alpha
\plainmu (\state, \action)| \leq 1 \quad \mbox{for any multi-index
  $\alpha$ satisfying $\vecnorm{\alpha}{1} \leq k$} \Big\}.
\end{align*}
Here for a multi-index $\alpha \in \Nat^p$, the quantity
$\partial^\alpha f$ denotes the mixed partial derivative
\begin{align*}
\partial^\alpha f(\state, \action) & \defn \Big(\prod_{j=1}^p
\frac{\partial^{\alpha_j}}{\partial x_j^{\alpha_j}} \Big) f(\state,
\action).
\end{align*}
We impose the following assumptions on the likelihood ratio and random
noise
\begin{subequations}
\label{eq:subgaussian-in-nonparametric}
\begin{align}
\Exs \Big[ \abss{\frac{\weightfunc (\State, \Action)}{\propscore
      (\State, \Action)} \big(\outcome - \treateff (\State, \Action)
    \big) }^\ell \Big] &\leq (\varbound \sqrt{\ell})^\ell \quad
\mbox{and}\\
\Exs \Big[ \abss{\frac{\weightfunc (\State, \Action)}{\propscore
      (\State, \Action)} }^\ell \Big] & \leq (\subgaussian
\sqrt{\ell})^\ell, \quad \mbox{for any $\ell \in \mathbb{N}_+$}.
\end{align}
Additionally, we impose the $L_2-L_4$ hypercontractivity condition
\begin{align}
\label{eq:hypercontractivity-in-nonparam}  
  \sqrt{\Exs \Big[\big( \tfrac{\weightfunc (\State,
        \Action)}{\propscore (\State, \Action)} f (\State, \Action)
      \big)^4 \Big]} \leq \ctwofour \; \Exs \Big[\big( \tfrac{\weightfunc
      (\State, \Action)}{\propscore (\State, \Action)} f (\State,
    \Action) \big)^2 \Big] \quad \mbox{for any $f \in \funcClass_k$,}
\end{align}
\end{subequations}
which is slightly stronger than the small-ball
condition~\ref{assume:small-ball}.

Our result involves the sequences $\widebar{\radtwo}_\numobs \defn
c_{\subgaussian, \pdim/k} \; \numobs^{- k / \pdim} \log \numobs$, and
\begin{align*}
  \widebar{\radone}_\numobs & \defn c_{\subgaussian, \pdim/k}
  \varbound \cdot \begin{cases} \numobs^{- \frac{k}{2 k + \pdim}} &
    \mbox{if $\pdim < 2k$} \\
    \numobs^{-1/4} \; \sqrt{\log \numobs} & \mbox{if $\pdim = 2k$,} \\
    \numobs^{- \frac{k}{2 \pdim}} & \mbox{if $\pdim > 2k$,}
  \end{cases}
\end{align*}
where the constant $c_{\subgaussian, \pdim/k}$ depends on the tuple
$(\subgaussian, \pdim/k, \ctwofour)$.

With this notation, when the outcome function is approximated by the
class $\funcClass_k$, we have the following guarantee for treatment
effect estimation:
\begin{corollary}
\label{cor:example-holder}
Under the small-ball condition~\ref{assume:small-ball} and the moment
bounds~\eqref{eq:subgaussian-in-nonparametric}(a)--(c), we have
\begin{align}
\label{EqnHolderBound}  
  \Exs \Big[ \abss{\tauhat_\numobs - \taustar}^2 \Big] \leq
  \frac{1}{\numobs} \left \{ \vstar^2 + c \inf_{\plainmu \in
    \funcClass_k} \weightednorm{\treateff - \plainmu}^2 \right \} +
  \frac{c}{\numobs} \Big \{ \widebar{\radone}_\numobs^2 +
  \widebar{\radtwo}_\numobs^2 \Big\},
\end{align}
where the constant $c$ depends only on the small-ball parameters
$(\smallballcon, \smallballprob)$.
\end{corollary}
\noindent See Appendix~\ref{app:subsec-proof-cor-example-holder} for
the proof of this corollary.\\

It is worth making a few comments about this result.  First, in the
high-noise regime where $\varbound \gtrsim 1$, the term
$\widebar{\radone}_{\numobs}$ is dominant.  This particular rate is
optimal in the Donsker regime ($\pdim < 2k$), but is sub-optimal when
$\pdim > 2 k$.  However, this sub-optimality only appears in
high-order terms, and is of lower order for a sample size\footnote{In
fact, this lower bound cannot be avoided, as shown by our analysis in
Section~\ref{subsubsec:lower-bound-examples}.}  $\numobs$ such that
$\log \numobs \gg (\pdim / k)$.  Indeed, even if the least-square
estimators is sub-optimal for nonparametric estimation in non-Donsker
classes, the reweighted least-square
estimator~\eqref{eq:least-square-estimator-in-stage-1} may still be
desirable, as it is able to approximate the projection of the function
$\treateff$ onto the class $\funcClass_k$, under the weighted norm
$\weightednorm{\cdot}$.

%%%%%%%%%%%%%%%%%%%%%%%%%%%%%%%%%%%%%%%%%%%%%%%%%%%%%%%%%%%%%%%%%

\subsubsection{Monotone functions}

We now consider a nonparametric problem that involves shape
constraints---namely, that of monotonic functions. Let $\phi: \Xspace
\times \actionspace \rightarrow [0, 1]$ be a one-dimensional feature
mapping. We consider the class of outcome functions that are monotonic
with respect to this feature---namely, the function class
\begin{align*}
\funcClass \mydefn \Big\{ (\state, \action) \rightarrow f \big( \phi
(\state, \action) \big) \; \mid \: f: [0,1] \rightarrow [0, 1] \mbox{
  is non-decreasing} \Big\}.
\end{align*}
We assume the outcome and likelihood ratio are uniformly
bounded---specifically, that
\begin{align}
\label{eq:boundedness-in-monotone-example}  
  \abss{\outcome_i} \leq 1 \quad \mbox{and} \quad
  \abss{\frac{\weightfunc (\State_i, \Action_i)}{\propscore (\State_i,
      \Action_i)}} \leq b \quad \mbox{almost surely for $i = 1, 2,
    \ldots, \numobs$.}
\end{align}   
Under these conditions, we have the following result:
\begin{corollary}
  \label{cor:example-monotonic}
Under the small-ball condition~\ref{assume:small-ball} and boundedness
condition~\eqref{eq:boundedness-in-monotone-example}, we have
\begin{align}
  \Exs \Big[ \abss{\tauhat_\numobs - \taustar}^2 \Big] & \leq
  \frac{1}{\numobs} \left \{ \vstar^2 + c \inf_{\plainmu \in
    \funcClass} \weightednorm{\plainmu - \treateff}^2 \right\} +
  \frac{c}{\numobs} \Big( \frac{b^2}{\numobs} \Big)^{2/3},
\end{align}
where the constants $(c_0, c)$ depend only on the small-ball
parameters $(\smallballcon, \smallballprob)$.
\end{corollary}
\noindent See Appendix~\ref{app:subsec-proof-cor-example-monotonic}
for the proof of this corollary.

Note that compared to
Corollaries~\ref{cor:example-linear}--~\ref{cor:example-holder},
Corollary~\ref{cor:example-monotonic} requires a stronger uniform
bound on the likelihood ratio $\weightfunc / \propscore$: it is
referred to as the \emph{strict overlap condition} in the causal
inference literature.  In our analysis, this condition is required to
make use of existing bracketing-based localized entropy
control. Corollary~\ref{cor:example-monotonic} holds for any sample
size $\numobs \geq 1$, and we establish a matching lower bound as a
consequence of~\Cref{thm:worst-case-shattering-dim} to be stated in
the sequel.  It should be noted that the likelihood ratio bound $b$ might
be large, in which case the high-order term
in~\Cref{cor:example-monotonic} could be dominant (at least for small
sample sizes).  As with previous examples, optimal estimation of the
scalar $\taustar$ requires optimal estimation of the function
$\treateff$ under $\weightednorm{\cdot}$-norm.  How to do so optimally
for isotonic classes appears to be an open question.

%%%%%%%%%%%%%%%%%%%%%%%%%%%%%%%%%%%%%%%%%%%%%%%%%%%%%%%%%%%%%%%%%

\subsection{Non-asymptotic normal approximation}

Note that the oracle inequality
in~\Cref{cor:least-sqr-split-estimator} involves an approximation
factor depending on the small-ball condition in
Assumption~\ref{assume:small-ball}, as well as other universal
constants. Even with sample size $\numobs$ tending to infinity, the
result of~\Cref{cor:least-sqr-split-estimator} does not ensure that
the auxiliary estimator $\muhat_{\numobs / 2}$ converges to a limiting
point. This issue, while less relevant for the mean-squared error
bound in~\Cref{cor:least-sqr-split-estimator}, assumes importance in
the inferential setting.  In this case, we do need the auxiliary
estimator to converge so as to be able to characterize the
approximation error.

In order to address this issue, we first define the orthogonal
projection within the class
\begin{align}
\mubar \mydefn \arg \min_{\plainmu \in \funcClass}
\weightednorm{\plainmu -\treateff}.
\end{align}
Our analysis also involves an additional squared Rademacher
complexity, one which involves the \emph{difference} $\treateff -
\mubar$.  It is given by
\begin{subequations}
\begin{align}
\SquaredDiffRade^2_\maux(\Hclass) \mydefn \Exs \Big[ \sup_{f \in
    \Hclass} \Big\{ \frac{1}{\maux} \sum_{i = 1}^\maux \rade_i
  \frac{\weightfunc^2(\State_i, \Action_i)}{\propscore^2(\State_i,
    \Action_i)} \big[ \treateff(\State_i, \Action_i) -
    \mubar(\State_i, \Action_i) \big] \: f(\State_i, \Action_i)
  \Big\}^2 \Big].
\end{align}
We let $\dradone_\maux > 0$ be the unique solution to the fixed
point equation
\begin{align}
\tfrac{1}{\dradone} \, \SquaredDiffRade_\maux \big((\funcClass -
\mubar) \cap \ball_\omega(\dradone) \big) & = \dradone.
\end{align}
\end{subequations}
The existence and uniqueness is guaranteed by an argument analogous to
that used in the proof of~\Cref{prop:existence-critical-radii}.

In order to derive a non-asymptotic CLT, we need a finite fourth
moment
\begin{align*}
\MomFour & \mydefn \Exs \Big[ \Big \{ \frac{\weightfunc(\State,
    \Action)}{\propscore(\State, \Action)} \big( \outcome - \mubar
  (\State, \Action) \big) + \actinprod{\weightfunc(\State,
    \cdot)}{\mubar(\State, \cdot)} \Big \}^4 \Big].
\end{align*}
The statement also involves the excess variance
\begin{align*}
v^2(\mubar) & \mydefn \Exs \Big[ \var \Big( \frac{\weightfunc(\State,
    \Action)}{\propscore(\State, \Action)} \cdot \big \{
  \treateff(\State, \Action) - \mubar(\State, \Action) \big \} \mid
  \State \Big) \Big].
\end{align*}

\medskip

\noindent With these definitions, we have the following guarantee:
\begin{corollary}
\label{thm:normal-approx-least-sqr}
Under Assumptions~\ref{assume:convset} and~\ref{assume:small-ball},
the two-stage estimator~\eqref{eq:estimator-framework} satisfies the
Wasserstein distance bound
\begin{align}
\label{eq:normal-approx-least-sqr}  
\Wass_1 \big( \sqrt{\numobs} \tauhat_\numobs, Z \big) \leq \tfrac{4
  \sqrt{\MomFour}}{[\vstar + v(\mubar) ]} \; \frac{1}{\sqrt{\numobs}}
+ c \big \{ \radtwo_{\numobs/2} + \radone_{\numobs/2} +
\dradone_{\numobs} \big \} + \diameter_\omega(\funcClass \cup
\{\treateff\}) \cdot e^{- c' \numobs },
\end{align}
\mbox{where $Z \sim \mathcal{N} \big(0, \vstar^2 + v^2(\mubar)
  \big)$,} and the pair $(c, c')$ of constants depend only on the
small-ball parameters $(\smallballcon, \smallballprob)$.
\end{corollary}
\noindent See~\Cref{subsec:proof-thm-normal-approx} for the proof of
this corollary. \\

\medskip

A few remarks are in order.  First, in the limit $\numobs \rightarrow
+ \infty$, ~\Cref{thm:normal-approx-least-sqr} guarantees asymptotic
normality of the estimate $\tauhat_{\numobs}$, with asymptotic
variance $\vstar^2 + v^2(\mubar)$. In contrast, the non-asymptotic
result given here makes valid inference possible at a finite-sample
level, by taking into account the estimation error for auxiliary
functions. Compared to the risk bound
in~\Cref{cor:least-sqr-split-estimator}, the right-hand-side of
equation~\eqref{eq:normal-approx-least-sqr} contains two terms: the
first term $\tfrac{4 \MomFour^2}{(\vstar + v(\mubar) ) \sqrt{\numobs}
}$ is the Berry--Esseen error, and an additional critical radius
$\dradone_{\numobs/ 2}$ depending on the localized multiplier
Rademacher complexity. When the approximation error $\treateff -
\mubar$ is of order $o(1)$, the multiplier Rademacher complexity
$\SquaredDiffRade_{\numobs / 2}$ becomes (asymptotically) smaller than
the Rademacher complexity $\SquaredRade_{\numobs / 2}$, resulting in a
critical radius $\dradone_{\numobs / 2}(\mubar)$ smaller than
$\radone_{\numobs / 2}(\mubar)$. On the other hand, the efficiency
loss in~\Cref{thm:normal-approx-least-sqr} is the exact variance
$v^2(\mubar)$ with unity pre-factor, which exhibits a smaller
efficiency loss compared to~\Cref{cor:least-sqr-split-estimator}.

\paragraph{Excess variance compared to approximation error:}

It should be noted that the excess variance term $v^2(\mubar)$
in~\Cref{thm:normal-approx-least-sqr} is smaller than the best
approximation error $\inf_{\treateff \in \funcClass}
\weightednorm{\mubar - \treateff}^2$. Indeed, the difference
$\Delta \defn \weightednorm{\mubar - \treateff}^2 - v^2(\mubar)$
can be written as
\begin{align}
\Delta & = \Exs \Big[ \Big(\frac{\weightfunc(\State,
    \Action)}{\propscore(\State, \Action)} \cdot \big( \treateff -
  \mubar \big)(\State, \Action) \Big)^2 \Big] - \Exs \Big[ \var \Big(
  \frac{\weightfunc(\State, \Action)}{\propscore(\State, \Action)}
  \cdot \big( \treateff - \mubar \big)(\State, \Action) \mid \State
  \Big) \Big] \nonumber \\
\label{eq:weightnorm-vs-vmubar}
& = \Exs_{\probxstar} \big[ \actinprod{\weightfunc(\State,
    \cdot)}{(\treateff - \mubar)(\State, \cdot)}^2 \big].
\end{align}
When considering minimax risk over a local neighborhood around the
function $\treateff$, the difference term computed above is dominated
by the supremum of the asymptotic efficient variance $\vstar^2$
evaluated within this neighborhood. Consequently, the upper bound
induced by~\Cref{thm:normal-approx-least-sqr} does not contradict the
local minimax lower bound in~\Cref{thm:minimax-local}; and since the
difference $\weightednorm{\mubar - \treateff}^2 - v^2(\mubar)$ does
not involve the importance weight ratio $\weightfunc / \propscore$,
this term is usually much smaller than the weighted norm term
$\weightednorm{\mubar - \treateff}^2$.

On the other hand, ~\Cref{thm:normal-approx-least-sqr} and
equation~\eqref{eq:weightnorm-vs-vmubar} provide guidance on the way
of achieving the optimal pointwise exact asymptotic variance. In
particular, when we choose a function class $\funcClass$ such that
$\actinprod{h(\state, \cdot)}{\weightfunc(\state, \cdot)} = 0$ for any
$h \in \funcClass$ and $\state \in \Xspace$, the
expression~\eqref{eq:weightnorm-vs-vmubar} becomes a constant
independent of the choice of $\mubar$. For such a function class, a
function $\mubar$ that minimizes the approximation error
$\weightednorm{\mu - \treateff}^2$ will also minimize the variance
$v^2(\mu)$. Such a class can be easily constructed from any function
class $\Hclass$ by taking a function $h \in \Hclass$ and replacing
it with $f(\state, \action) \defn h(\state, \action) -
\actinprod{h(\state, \cdot)}{\weightfunc(\state, \cdot)}$. And the
optimal variance can still be written in the form of approximation
error:
\begin{align*}
v(\mubar) = \weightednorm{\mubar - \widetilde{\mu}^*}, \quad
\mbox{where $\widetilde{\mu}^*(\state, \action) \mydefn
  \treateff(\state, \action) - \actinprod{\treateff(\state,
    \cdot)}{\weightfunc(\state, \cdot)}$.}
\end{align*}
Indeed, the functional $v$ can be seen as the induced norm of
$\weightednorm{\cdot}$ in the quotient space generated by
$\Ltwospace_\omega$ modulo the subspace $\Ltwospace(\probxstar)$ that
contains functions depending only on the state but not action.

%%%%%%%%%%%%%%%%%%%%%%%%%%%%%%%%%%%%%%%%%%%%%%%%%%%%%%%%%%%%%%

\section{Minimax lower bounds}
\label{SecLower}

Thus far, we have derived upper bounds for particular estimators of
the linear functional $\avgtreat(\probInstance)$, ones that involve
the weighted norm~\eqref{EqnWeighted}.  In this section, we turn to
the complementary question of deriving local minimax lower bounds for
the problem.  Recall that any given problem instance is characterized
by a quadruple of the form $(\probxstar, \propscore, \treateff,
\weightfunc)$.  In this section, we state some lower bounds that hold
uniformly over all estimators that are permitted to know both the
policy $\propscore$ and the weight function $\weightfunc$.  With
$(\propscore, \weightfunc)$ known, the instance is parameterized by
the pair $(\probxstar, \treateff)$, and we derive two types of lower
bounds: \\
\bcar
\item In~\Cref{thm:minimax-local}, we study local minimax bounds in
  which the unknown probability distribution $\probxstar$ and
  potential outcome function are allowed to range over suitably
  defined neighborhoods of a given target pair $(\probxstar,
  \treateff)$, respectively, but without structural conditions on the
  function classes.
\item In~\Cref{thm:worst-case-shattering-dim}, we impose structural
  conditions on the function class $\funcClass$ used to model
  $\treateff$, and prove a lower bound that involves the complexity of
  $\funcClass$---in particular, via its fat shattering dimension.
  This lower bound shows that if the sample size is smaller than the
  function complexity, then any estimator has a mean-squared error
  larger than the efficient variance.
\ecar

\subsection{Instance-dependent bounds under mis-specification}

Given a problem instance $\probInstanceStar =( \probxstar, \treateff)$
and an error function $\delta: \Xspace \times \actionspace \rightarrow
\real$, we consider the local neighborhoods
\begin{subequations}
\begin{align}
\Nval{\delta}(\treateff) & \mydefn \Big\{ \plainmu \; \mid \;\; \abss{
  \plainmu(\state, \action) - \treateff(\state, \action)} \leq
\delta(\state, \action) \quad \mbox{for $(\state, \action) \in \Xspace
  \times \ActionSpace$} \Big\}, \\
\Nprob(\probxstar) & \mydefn \Big\{ \probx \; \mid \; \;
\kull{\probx}{\probxstar} \leq \tfrac{1}{\numobs} \Big\}.
\end{align}
\end{subequations}
Our goal is to lower bound the \emph{local minimax risk}
\begin{align}
\label{EqnLocalMinimax}  
\minimaxRisk \big(\class_\delta(\probInstanceStar) \big) \mydefn
\inf_{\tauhat_\numobs} \sup_{\probInstance \in
  \class_{\delta}(\probInstanceStar) } \Exs\abss{\avgtreat -
  \tauhat_\numobs}^2 \: \; \mbox{where
  $\class_{\delta}(\probInstanceStar) \mydefn \big \{(  \probx, \plainmu) \in  \Nprob(\probxstar)
  \big\}\times \Nval{\delta}(\treateffzero)$.}
\end{align} 

\noindent Let us now specify the assumptions that underlie our lower bounds.
\paragraph{Assumptions for lower bound:}
First, we require some tail control on certain random variables,
stated in terms of the \emph{$(2,4)$-moment-ratio} $\ltwolfour{Y}
\defn \tfrac{\sqrt{\Exs[Y^4]}}{\Exs[Y^2]}$.
\myassumption{MR}{AssEll2Ell4}{ The random variables
  \begin{align}
\label{EqnDefnZone}
\Zone(\State, \Action) \defn \frac{\delta(\State, \Action)
  \weightfunc(\State, \Action)}{\propscore(\State, \Action)}, \quad
\mbox{and} \quad \Ztwo(\State, \Action) \defn
\actinprod{\treateff(\State, \cdot)}{\weightfunc(\State, \cdot)} -
\avgtreat (\probInstanceStar)
\end{align}
have finite $(2,4)$-moment ratios $\ctwofour \mydefn
\ltwolfour{\Zone}$ and $\ctwofour' \mydefn \ltwolfour{\Ztwo}$.}

\noindent Second, we require the existence of a constant $\cmax > 0$
such that the distribution $\probxstar$ satisfies the following
\emph{compatibility condition}.
\myassumption{COM}{AssCompatibility}{ For a finite state space
  $\Xspace$, we require $\probxstar(\state) \leq c_{\max} / |\Xspace|$
  for all $\state \in \Xspace$.  If $\Xspace$ is infinite, we require
  that $\probxstar$ is non-atomic (i.e., $\probxstar(\{\state\}) = 0$
  for all $\state \in \Xspace$), and set $\cmax = 1$ for concreteness.
}
\noindent Finally, we impose a lower bound on the \emph{local
neighborhood size}: \myassumption{LN}{AssLocalNeigh}{ The neighborhood
  function $\delta$ satisfies the lower bound
  \begin{align}
\label{eq:size-of-neighborhood}    
\sqrt{\numobs} \; \delta(\state, \action) & \geq
\frac{\weightfunc(\state, \action) \sigma^2(\state,
  \action)}{\propscore(\state, \action) \weightednorm{\sigma}} \quad
\mbox{for any $(\state, \action) \in \StateSpace
  \times \ActionSpace$.}
    \end{align}
}

\medskip

\noindent In the following statement, we use $c$ and $c'$ to denote
universal constants.
\begin{theorem}
\label{thm:minimax-local}
Under Assumptions~\ref{AssEll2Ell4},~\ref{AssCompatibility}
and~\ref{AssLocalNeigh}, given a sample size lower bounded as $\numobs
\geq c' \max \{ (\ctwofour')^2, \ctwofour^2 \}$, the local minimax
risk over the class $\class_{\delta}(\probInstanceStar)$ is lower
bounded as
\begin{align}
\label{EqnMinimaxLocal}  
\minimaxRisk \big( \class_{\delta}(\probInstanceStar) \big) & \geq
\frac{c}{\numobs} \begin{cases} \vstar^2 & \mbox{if $\numobs \geq
    \frac{|\Xspace|}{\cmax}$} \\
\vstar^2 + \weightednorm{\delta}^2 & \mbox{otherwise.}
\end{cases}
\end{align}
\end{theorem}
\noindent We prove this claim
in~\Cref{subsec:proof-thm-minimax-local}. \\

It is worth understanding the reasons for each of the assumptions
required for this lower bound to hold.  The compatibility
condition~\ref{AssCompatibility} is needed to ensure that no single
state can take a significant proportion of probability mass under
$\probxstar$.  If this condition is violated, then it could be
possible to construct a low MSE estimate of the outcome function via
an empirical average, which would then break our lower bound.  The
neighborhood condition~\ref{AssLocalNeigh} ensures that the set of
problems considered by the adversary is large enough to be able to
capture the term $\weightednorm{\sigma}^2$ in the optimal variance
$\vstar^2$. Without this assumption, the ``local-neighborhood''
restriction on certain states-action pairs could be more informative
than the data itself. \\

Now let us understand some consequences of~\Cref{thm:minimax-local}.
First, it establishes the information-theoretic optimality
of~\Cref{cor:least-sqr-split-estimator}
and~\Cref{thm:normal-approx-least-sqr} in an instance-dependent
sense. Consider a function class $\funcClass$ that approximately
contains the true outcome function $\treateff$; more formally,
consider the $\delta$-approximate version of $\funcClass$ given by
\begin{align*}
\funcClass_\delta \mydefn \Big \{ \widetilde{\plainmu}
\in \Ltwospace_\omega \mid \; \exists \plainmu \in \funcClass \mbox{
  such that } \abss{\plainmu(\state, \action) -
  \widetilde{\plainmu}(\state, \action)} \leq \delta(\state, \action)
\quad \mbox{for all $(\state, \action) \in \Xspace
  \times \actionspace$} \Big \},
\end{align*}
and let us suppose that $\treateff \in \funcClass_\delta$.  With this
notation, ~\Cref{thm:minimax-local} implies a lower bound of the form
\begin{align}
\inf_{\tauhat_\numobs} \sup_{\substack{\plaintreateff \in
    \funcClass_\delta \\ \probx \in \Nprob(\probxstar)}} \Exs \big[
  \abss{\avgtreat - \tauhat_\numobs}^2 \big] \geq \frac{c}{\numobs}
\Big\{ \sup_{\plaintreateff \in \funcClass} \var \big(
\actinprod{\weightfunc(\State, \cdot)}{\plaintreateff(\State, \cdot)}
+ \weightednorm{\sigma}^2 + \weightednorm{\delta}^2 \Big\}.
\end{align}
Thus, we see that the efficiency loss due to errors in estimating the
outcome function is unavoidable; moreover, this loss is measured in
the weighted norm $\weightednorm{\cdot}$ that also appeared centrally
in our upper bounds.

It is also worth noting that for a finite cardinality state space
$\Xspace$, ~\Cref{thm:minimax-local} exhibits a ``phase transition''
in the following sense: for a sample size $\numobs \gg |\Xspace|$, the
lower bound is simply a non-asymptotic version of the semi-parametric
efficiency lower bound (up to the pre-factor\footnote{Using slightly
more involved argument, this pre-factor can actually be made
arbitrarily close to unity.} $c > 1$). On the other hand, when
$\numobs < |\Xspace|$, then the term $\weightednorm{\delta}^2 /
\numobs$ starts to play a significant role.  For an infinite state
space $\Xspace$ without atoms, the lower bound~\eqref{EqnMinimaxLocal}
holds for any sample size $\numobs$.

By taking $\treateffzero = 0$ and $\delta(\state, \action) = 1$ for
all $(\state, \action)$, equation~\eqref{EqnMinimaxLocal} implies the
global minimax lower bound
\begin{align}
\label{eq:global-minimax-instantiation}  
\inf_{\tauhat_\numobs} \sup_{\vecnorm{\plainmu}{\infty} \leq 1,
  ~\probx = \probxstar} \Exs \big[ \abss{\avgtreat -
    \tauhat_\numobs}^2 \big] \geq \frac{c}{\numobs}
\int_{\actionspace} \Exs_{\probxstar} \Big[
  \frac{\weightfunc^2(\State, \action)}{\propscore(\State,
    \action)}\Big] d \actbasemsr(\action),
\end{align}
valid whenever $\numobs \leq
|\Xspace|$.

The $\chi^2$-type term on the right-hand side of this bound is related
to---but distinct from---results from past work on off-policy
evaluation in bandits~\cite{wang2017optimal,ma2022minimax}.  In this
past work, a term of this type arose due to noisiness of the
observations.  In contrast, our lower
bound~\eqref{eq:global-minimax-instantiation} is valid even if the
observed outcome is noiseless, and the additional risk depending on
the weighted norm arises instead from the impossibility of estimating
$\treateff$ itself.

%%%%%%%%%%%%%%%%%%%%%%%%%%%%%%%%%%%%%%%%%%%%%%%%%%%%%%%%%%%%%%%%

\subsection{Lower bounds for structured function classes}

As we have remarked, in the special case of a finite state space
($|\StateSpace| < \infty$), ~\Cref{thm:minimax-local} exhibits an
interesting transition at the boundary $\numobs \asymp |\StateSpace|$.
On the other hand, for an infinite state space, the stronger lower
bound in~\Cref{thm:minimax-local}---namely, that involving
$\weightednorm{\delta}^2$---is always in force.  It should be noted,
however, that this strong lower bound depends critically on the fact
that~\Cref{thm:minimax-local} imposes \emph{no conditions} on the
function class $\funcClass$ of possible treatment effects, so that the
error necessarily involves the local perturbation $\delta$.

In this section, we undertake a more refined investigation of this
issue.  In particular, when some complexity control is imposed upon
$\funcClass$, then the lower bounds again exhibit a transition: any
procedure pays a price only when the sample size is sufficiently small
relative to the complexity of $\funcClass$.  In doing so, we assess
the complexity of $\funcClass$ using the \emph{fat-shattering
dimension}, a scale-sensitive version of the VC
dimension~\cite{KEARNS1994464,alon1997scale}.
\myassumption{FS}{DefnFatShatter}{ A collection of data points
  $(\state_i)_{i = 1}^N$ is \emph{shattered at scale $\delta$} by a
  function class $\funcClassTemp: \Xspace \rightarrow \real$ means
  that for any subset $S \subseteq \{1, \ldots, N\}$, there exists a
  function $f \in \funcClassTemp$ and a vector $t \in \real^N$ such
  that
\begin{align}
\label{eq:defn-fat-shattering}  
  f(\state_i) \geq t_i + \delta \quad \mbox{for all $i \in S$, and}
  \quad f(\state_i) \leq t_i - \delta \quad \mbox{for all $i \notin
    S$.}
\end{align}
The fat-shattering dimension $\fatshatter{\delta}(\funcClassTemp)$ is
the largest integer $N$ for which there exists some sequence
$(\state_i)_{i = 1}^N$ shattered by $\funcClassTemp$ at scale
$\delta$.  }

In order to illustrate a transition depending on the fat shattering
dimension, we consider the minimax risk
\begin{align*}
\minimaxRisk(\funcClass) & \mydefn \inf_{\tauhat_\numobs}
\sup_{\substack{ \plainmu \in \funcClass \\ \probx \in
    \mathcal{P}(\Xspace) }} \Exs \big[ \abss{\tauhat_\numobs -
    \avgtreat (\probx, \plainmu)}^2 \big],
\end{align*}
specializing to the case of a finite action space $\actionspace$
equipped with the counting measure $\basemsr$.  We further assume that
the class $\funcClass$ is a product of classes associated to each
action, i.e., $\funcClass = \bigotimes_{\action \in \actionspace}
\funcClass_\action$, with $\funcClass_\action$ being a \emph{convex
subset} of real-valued functions on the state space $\Xspace$.  We
also assume the existence\footnote{Thus, per force, we have $\fatdim
\leq \fatshatter{\delta_\action}{(\funcClass_\action)}$ for each
$\action \in \actionspace$.}  of a sequence $\{\state_j\}_{j =
  1}^\fatdim$ that, for each action $\action \in
\actionspace$, is shattered by $\funcClass_\action$ at scale
$\delta_\action$. Analogous to the moment ratio assumption~\ref{AssEll2Ell4}, we need an additional assumption that
\begin{align}
    \ctwofour \mydefn \ltwolfour{\frac{\weightfunc (\State, \Action)}{\propscore (\State, \Action)} \delta_\Action} < + \infty, \quad \mbox{for } \State \sim \mathcal{U} (\{x_j\}_{j = 1}^\fatdim) ~ \mbox{and} ~ \Action \sim \propscore (\State, \cdot).
\end{align}
\begin{proposition}
\label{thm:worst-case-shattering-dim}
With the set-up given above, there are universal constants $(c, c')$
such that for any sample size satisfying $\numobs \geq \ctwofour^2$ and $\numobs \leq c' \fatdim$, we have the
lower bound
\begin{align}
  \label{eq:minimax-bound-funcclass-small-sample}
  \minimaxRisk(\funcClass) \geq \frac{c}{\numobs} \; \Big \{
  \frac{1}{\fatdim} \sum_{j = 1}^\fatdim \sum_{\action
    \in \actionspace} \frac{\weightfunc^2(\state_j,
    \action)}{\propscore(\state_j, \action)} \delta_\action^2 \Big \}.
    \end{align}
\end{proposition}
\noindent See~\Cref{subsec:proof-thm-minimax-shattering-dim} for the
proof of this claim. \\
  
A few remarks are in order. First, if we take $\probx$ to be the
uniform distribution over the sequence $\{x_j\}_{j = 1}^\fatdim$, the
right-hand-side of the
bound~\eqref{eq:minimax-bound-funcclass-small-sample} is equal to
$\frac{c}{\numobs} \weightednorm{\delta}^2$.  Thus,
\Cref{thm:worst-case-shattering-dim} is the analogue of our earlier
lower bound~\eqref{EqnLocalMinimax} under the additional restriction
that the treatment effect function $\treateff$ belong the given
function class $\funcClass$.  This lower bound holds as long as
$\numobs \leq c' \fatdim$, so that the fat shattering dimension
$\fatdim$ as opposed to the state space cardinality $|\Xspace|$ (for a
discrete state space) demarcates the transition between different
regimes.

An important take-away of~\Cref{thm:worst-case-shattering-dim} is that
the sample size must exceed the ``complexity'' of the function class
$\funcClass$ in order for the asymptotically efficient variance
$\vstar^2$ to be dominant.  More precisely, suppose that---for some
scale $\delta > 0$---the sample size is smaller than the
fat-shattering dimension $\fatshatter{\delta}(\funcClass)$.  In this
regime, the na\"{i}ve IPW estimator~\eqref{eq:ipw-estimator-simple} is
actually instance-optimal, even when there is no noise.  Observe that
its risk contains a term of the form $\sum_{\action \in \actionspace}
\Exs \big[ \tfrac{\weightfunc^2 (\State, \action)}{\propscore (\State,
    \action)} \big]$, which is \emph{not present} in the
asymptotically efficient variance $\vstar^2$.

By contrast, suppose instead that the sample size exceeds the
fat-shattering dimension.  In this regime, it is possible to obtain
non-trivial estimates of the treatment effect, so that superior
estimates of $\taustar$ are possible.  From the point of view of our
theory, one can use the fat shattering dimension
$\fatshatter{\delta}(\funcClass)$ to control the $\delta$-covering
number~\cite{mendelson2002entropy}, and hence the Rademacher
complexities that arise in our theory.  Doing so leads to non-trivial
radii $(\radone_{\numobs/ 2}, \radtwo_{\numobs/ 2})$
in~\Cref{cor:least-sqr-split-estimator}, and consequently, the
asymptotically efficient variance will become the dominant term.  We
illustrate this line of reasoning via various examples in
Section~\ref{SecExamples}.

It should be noted that a sample size scaling with the fat-shattering
dimension is also known to be necessary and sufficient to learn the
function $\treateff$ with $o(1)$
error~\cite{KEARNS1994464,bartlett1994fat,alon1997scale}. These classical results,
in combination with our
Proposition~\ref{thm:worst-case-shattering-dim} and
Theorem~\ref{cor:least-sqr-split-estimator}, exhibit that necessary
conditions on the sample size for consistent estimation of the
function $\treateff$ are equivalent to those requiring for achieving
the asymptotically efficient variance in estimating the scalar
$\taustar$.

\paragraph{Worst-case interpretation:}
It is worthwhile interpreting the
bound~\eqref{eq:minimax-bound-funcclass-small-sample} in a worst-case
setting.  Consider a problem with binary action space $\actionspace =
\{0, 1\}$ and $\weightfunc (\state, \action) = 2 \action - 1$. Suppose
that we use a given function class $\funcClassTemp$ (consisting of
functions from the state space $\Xspace$ to the interval $[0, 1]$) as
a model\footnote{We write $\treateff \in \funcClassTemp$ as a
shorthand for this set-up.}  of both of the functions $\treateff
(\cdot, 0)$ and $\treateff (\cdot, 1)$. Given a scalar
$\propscore_{\min} \in (0, 1/2)$, let $\Pi (\propscore_{\min})$ be the
set of propensity score functions such that $\propscore (\state, 1)
\in [\propscore_{\min}, 1 - \propscore_{\min}]$ for any $\state
\in \Xspace$. By taking the worst-case over this class, we find that
there are universal constants $c, c' > 0$ such that
\begin{align}
\label{eq:worst-case-instantiation-of-funcclass-lower}  
\sup_{\propscore \in \Pi (\propscore_{\min}) } \inf_{\tauhat_\numobs}
\sup_{\treateff \in \funcClassTemp} \Exs \big[ \abss{\tauhat_\numobs -
    \tau }^2 \big] & \geq c \; \begin{dcases} \frac{1}{\numobs} +
  \frac{ \widebar{\delta}^2}{\numobs \propscore_{\min} } & \mbox{for
    $\numobs \leq c' \fatshatter{\bar{\delta}} (\funcClass)$,} \\
  \frac{1}{\numobs}  & \mbox{otherwise,}
\end{dcases}
\end{align}
for any $\widebar{\delta} \in (0, 1)$.  The validity of this lower
bound does not depend on noise in the outcome observations (and
therefore applies to noiseless settings).  Since $\propscore_{\min} \in
(0, 1)$, any scalar $\widebar{\delta} \gg \sqrt{\propscore_{\min}}$
yields a non-trivial risk lower bound for sample sizes $\numobs$ below
the threshold $ \fatshatter{\bar{\delta}} (\funcClass)$.

\paragraph{Relaxing the convexity requirement:}

\Cref{thm:worst-case-shattering-dim} is based on the assumption each
function class $\funcClass_\action$ is convex.  This requirement can
be relaxed if we require instead that the sequence $\{x_i \}_{i =
  1}^\fatdim$ be shattered with the
inequalities~\eqref{eq:defn-fat-shattering} all holding with
equality---that is, for any subset $S$, there exists a function $f \in
\funcClassTemp$ and a vector $t \in \real^{\fatdim}$ such that
\begin{align}
\label{EqnStrongShatter}
  f(\state_i) = t_i + \delta \quad \mbox{for all $i \in S$, and} \quad
  f(\state_i) = t_i - \delta \quad \mbox{for all $i \notin S$.}
\end{align}
For example, any class of functions mapping $\Xspace$ to the binary
set $\{0, 1\}$ satisfies this condition with $\fatdim =
\VCdim(\funcClass)$ and $\delta = 1/2$. In the following, we provide
additional examples of non-convex function classes that satisfy
equation~\eqref{EqnStrongShatter}.

%%%%%%%%%%%%%%%%%%%%%%%%%%%%%%%%%%%%%%%%%%%%%%%%%%%%%%%%%%%%%%%%%%%%%%%%%

\subsubsection{Examples of fat-shattering lower bounds}
\label{subsubsec:lower-bound-examples}

We discuss examples of the fat-shattering lower
bound~\eqref{eq:minimax-bound-funcclass-small-sample} in this
section. We first describe some implications for convex classes.  We
then treat some non-convex classes using the strengthened shattering
condition~\eqref{EqnStrongShatter}.

\begin{example}[Smoothness class in high dimensions]
\upshape
We begin with a standard H\"{o}lder class on the domain $\Xspace =
[-1, 1]^p$. For some index $k = 1, 2, \ldots$, we consider functions
that are $k$-order smooth in the following sense
\begin{align}
  \funcClass_{k}^{\mathrm{(Lip)} } & \mydefn \Big\{ f:[-1, 1]^p
  \rightarrow \real \; \mid \; \sup_{x \in \Xspace} ~\max_{\alpha \in
    \Nat^p, ~ \vecnorm{\alpha}{1} \leq k} \abss{\partial^\alpha
    f(\state)} \leq 1 \Big\}.
\end{align}
By inspection, the class $\funcClass$ is convex. We can lower bound
its fat shattering dimension by combining classical results on
$L^2$-covering number of smooth
functions~\cite{kolmogorov1959varepsilon} with the relation between
fat shattering dimension and covering
number~\cite{mendelson2002entropy}, we conclude that
\begin{align}
\fatshatter{t} \big( \funcClass_{k}^{\mathrm{(Lip)} } \big) \geq
2^{\pdim / k}, \quad \mbox{for a sufficiently small scale $t > 0$.}
\end{align}
Consequently, for a function class with a constant order of smoothness
(i.e., not scaling with the dimension $\pdim$), the sample size
required to approach the asymptototically optimal efficiency scales
exponentially in $\pdim$.  \hfill \goodendex
\end{example}

\begin{example}[Single index models]\upshape
Next we consider a class of single index models with domain $\Xspace =
[-1, 1]^\pdim$. Since our main goal is to understand scaling issues,
we may assume that $\pdim$ is an integer power of $2$ without loss of
generality. Given a differentiable function $\linkfun: \real
\rightarrow \real$ such that $\linkfun(0) = 0$ and $\linkfun'(\state)
\geq \linklower > 0$ for all $x \in \real$, we consider ridge
functions of the form $g_\beta(\state) \defn
\linkfun\big(\inprod{\beta}{x} \big)$.  For a radius $R > 0$, we
define the class
\begin{align}
\smallsuper{\funcClass}{GLM}_R & \mydefn \Big\{ g_\beta \; \mid \;
\|\beta\|_2 \leq R \Big\}.
\end{align}

Let us verify the strengthened shattering
condition~\eqref{EqnStrongShatter}.  Suppose that the vectors $\{ x_j
\}_{j=1}^\pdim$ define the Hadamard basis in $\pdim$ dimensions, and
so are orthonormal.  Taking $t_j = 0$ for $j = 1, \ldots, \pdim$,
given any binary vector $\zeta \in \{ -1, 1 \}^\pdim$, we define the
$\pdim$-dimensional vector
\begin{align*}
\beta(\zeta) = \frac{1}{\pdim} \sum_{j = 1}^\pdim \linkfun^{-1} \big(
\zeta_j a R \big)  \: x_j,
\end{align*}
Given the orthonormality of the vectors $\{x_j\}_{j=1}^\pdim$, we have
\begin{align*}
\inprod{\beta(\zeta)}{x_\ell} & = \linkfun^{-1} \big( \zeta_\ell a R
\big) \qquad \mbox{for each $\ell = 1, \ldots, \pdim$,}
\end{align*}
and thus $g_{\beta(\zeta)}(\state_\ell) = \zeta_\ell a R$ for each
$\ell = 1, 2, \ldots, \pdim$. Consequently, the function class
$\smallsuper{\funcClass}{GLM}_R$ satisfies the strengthened shattering
condition~\eqref{EqnStrongShatter} with fat shattering dimension
$\fatdim = \pdim$ and scale $\delta = a R$. So when the outcome
follows a generalized linear model, a sample size must be at least of
the order $\pdim$ in order to match the optimal asymptotic efficiency.
\hfill\goodendex
\end{example}

\begin{example}[Sparse linear models]\upshape
  \label{ExaSparseLinear}
Once again take the domain $[-1, 1]^\pdim$, and consider linear
functions of the form $f_\beta(\state) = \inprod{\beta}{\state}$ for
some parameter vector $\beta \in \real^\pdim$. Given a positive
integer $s \in \{1, \ldots, \pdim \}$, known as the \emph{sparsity
index}, we consider the set of $s$-sparse linear functions
\begin{align}
\smallsuper{\funcClass}{sparse}_s \mydefn \Big\{ f_\beta \; \mid \;
\abss{\mathrm{supp}(\beta)} \leq s, \; \; \mbox{and} \; \;
\vecnorm{\beta}{\infty} \leq 1 \Big\}.
\end{align}
As noted previously, sparse linear models of this type have a wide
range of applications (e.g., see the book~\cite{HasTibWai15}).

In~\Cref{subsec:strong-shatter-sparse}, we prove that the strong
shattering condition~\eqref{EqnStrongShatter} holds with fat
shattering dimension $\fatdim \asymp s \log \big(\tfrac{e \pdim}{s}
\big)$.  Consequently, if the outcome functions $\treateff$ follow a
sparse linear model, at least $\Omega \Big( s \log \big( \tfrac{e
  \pdim}{s} \big) \Big)$ samples are needed to make use of this
fact. \hfill \goodendex
\end{example}

%%%%%%%%%%%%%%%%%%%%%%%%%%%%%%%%%%%%%%%%%%%%%%%%%%%%%%%%%%%%%%%

\section{Proofs of upper bounds}

In this section, we prove the upper bounds on the estimation error
(\Cref{thm:upper-sample-split-general}
and~\Cref{cor:least-sqr-split-estimator}), along with corollaries for
specific models.

%%%%%%%%%%%%%%%%%%%%%%%

\subsection{Proof of~\Cref{thm:upper-sample-split-general}}
\label{subsec:proof-upper-sample-split-general}

The error can be decomposed into three terms as $\tauhat_\numobs -
\taustar = \Term_* - \Term_1 - \Term_2$, where
\begin{align*}
\Term_* \mydefn \frac{1}{\numobs} \sum_{i = 1}^{\numobs} \Big\{
\tfrac{\weightfunc(\State_i, \Action_i)}{\propscore(\State_i,
  \Action_i)} \outcome_i - \taustar - \fstar(\State_i, \Action_i)
\Big\}, & \\
\Term_1 \mydefn \frac{1}{\numobs} \sum_{i = 1}^{\numobs / 2} \big(
\fhat_{\numobs / 2}^{(2)}(\State_i, \Action_i) - \fstar(\State_i,
\Action_i) \big), \quad & \mbox{and} \quad \Term_2 \mydefn
\frac{1}{\numobs} \sum_{i = \numobs / 2 + 1}^{\numobs} \big(
\fhat_{\numobs / 2}^{(1)}(\State_i, \Action_i) - \fstar(\State_i,
\Action_i) \big).
\end{align*}
Since the terms in the summand defining $\Term_*$ are
$\mathrm{i.i.d.}$, a straightforward computation yields
\begin{align*}
\Exs[\Term_*^2] = \frac{1}{\numobs} \Exs \Big[ \Big(
  \tfrac{\weightfunc(\State_i, \Action_i)}{\propscore(\State_i,
    \Action_i)} \outcome_i - \taustar - \fstar(\State_i, \Action_i)
  \Big)^2 \Big] = \frac{\vstar^2}{\numobs},
\end{align*}
corresponding to the optimal asymptotic variance.  For the cross term
$\Exs[T_1 T_2]$, applying the Cauchy-Schwarz inequality yields
\begin{align*}
 \abss{\Exs[T_1 T_2]} \leq \sqrt{\Exs[T_1^2]} \cdot \sqrt{\Exs
   [T_2^2]} \leq \tfrac{1}{2 \numobs} \Exs \big[
   \weightednorm{\muhat_{\numobs/2} - \treateff}^2 \big].
\end{align*}
Consequently, in order to complete the proof, it suffices to show that
\begin{subequations}
  \begin{align}
\label{EqnBoundOne}    
\Exs[\Term_1^2] \; = \; \Exs[\Term_2^2] & = \tfrac{1}{2 \numobs} \Exs
\big[ \weightednorm{\muhat_{\numobs/2} - \treateff}^2 \big], \quad
\mbox{and} \\
\label{EqnBoundTwo}
\Exs[\Term_1 \Term_*] & = \Exs[\Term_2 \Term_*] = 0.
  \end{align}
\end{subequations}

\paragraph{Proof of equation~\eqref{EqnBoundOne}:}
We begin by observing that $\Exs \big[ \Term_1^2 \mid \datablock_2
  \big] = \frac{1}{2 \numobs} \vecnorm{\fhat_{\numobs / 2}^{(2)} -
  \fstar}{\probx \times \propscore}^2$.  Now recall
equations~\eqref{eq:defn-fstar-function} and~\eqref{eq:defn-fhat} that
define $\fstar$ and $\fhat_{\numobs / 2}^{(2)}$ respectively.  From
these definitions, we have
\begin{align*}
 \vecnorm{\fhat_{\numobs / 2}^{(2)} - \fstar}{\probx \times
   \propscore}^2 &= \Exs_{X \sim \probx} \Big[ \var_{\Action \sim
     \propscore(\State, \cdot)} \Big( \frac{\weightfunc(\State,
     \Action)}{\propscore(\State, \Action)} \big( \muhat_{\numobs /
     2}^{(2)}(\State, \Action) - \treateff(\State, \Action) \big) \mid
   X \Big) \mid \datablock_2 \Big]\\ &\leq \Exs_{(\State, \Action)
   \sim \probx \times \propscore} \Big[ \frac{\weightfunc^2(\State,
     \Action) }{\propscore^2(\State, \Action) } \big( \muhat_{\numobs
     / 2}^{(2)}(\State, \Action) - \treateff(\State, \Action) \big)^2
   \mid \datablock_2 \Big] = \weightednorm{ \muhat_{\numobs / 2}^{(2)}
   - \treateff}^2.
\end{align*}
Putting together the pieces yields $\Exs[\Term_1^2] \leq \tfrac{1}{2
  \numobs} \Exs[\weightednorm{ \muhat_{\numobs / 2}^{(2)} -
    \treateff}^2]$ as claimed.  A similar argument yields the same
bound for $\Exs[\Term_2^2]$.

\paragraph{Proof of equation~\eqref{EqnBoundTwo}:}
We first decompose the term $T_*$ into two parts:
\begin{align*}
 T_{*, j} \mydefn \frac{1}{\numobs} \sum_{i = \numobs(j - 1)/ 2 + 1}^{
   \numobs j / 2 } \Big\{ \frac{\weightfunc(\State_i,
   \Action_i)}{\propscore(\State_i, \Action_i)} \outcome_i - \taustar
 - \fstar(\State_i, \Action_i) \Big\}, \quad \mbox{for $j \in \{1,
   2\}$}.
\end{align*}
Since for any $x \in \Xspace$, the functions $\fstar(\state, \cdot)$
and $\fhat_{\numobs / 2}^{(2)}(\state, \cdot)$ are both zero-mean
under $\propscore(\state, \cdot)$, we have the following identity.
\begin{align*}
\Exs \big[ T_{*, 2} T_1 \mid \datablock_2 \big] = \frac{1}{\numobs}
\sum_{i = 1}^{\numobs / 2} \Exs \Big[ T_{*, 2} \cdot \Exs \big[
    \fhat_{\numobs / 2}^{(2)}(\State_i, \Action_i) - \fstar(\State_i,
    \Action_i) \mid \State_i \big] \mid \datablock_2 \Big] = 0.
\end{align*}
Similarly, we have $\Exs \big[T_{*, 1} T_2 \big] = 0$. It remains to
study the terms $\Exs \big[T_{*, j} T_j \big]$ for $j \in \{1,
2\}$. We start with the following expansion:
\begin{align*}
T_{*, 1} \cdot T_1 & = \frac{1}{\numobs^2} \sum_{ i = 1}^{\numobs / 2}
\Big\{ \frac{\weightfunc(\State_i, \Action_i)}{\propscore(\State_i,
  \Action_i)} \outcome_i - \taustar - \fstar(\State_i, \Action_i)
\Big\} \cdot \big( \fhat_{\numobs / 2}^{(2)}(\State_i, \Action_i) -
\fstar(\State_i, \Action_i) \big) \\
& \qquad + \frac{1}{\numobs^2} \sum_{1 \leq i \neq \ell \leq \numobs /
  2} \Big\{ \frac{\weightfunc(\State_i, \Action_i)}{\propscore
  (\State_i, \Action_i)} \outcome_i - \taustar - \fstar(\State_i,
\Action_i) \Big\} \cdot \big( \fhat_{\numobs / 2}^{(2)}(\State_\ell,
\Action_\ell) - \fstar(\State_\ell, \Action_\ell) \big).
\end{align*}
For $i \neq \ell$, by the unbiasedness of $T_*$, we note that:
\begin{align*}
 \Exs \Big[\Big\{ \frac{\weightfunc(\State_i,
     \Action_i)}{\propscore(\State_i, \Action_i)} \outcome_i -
   \taustar - \fstar(\State_i, \Action_i) \Big\} \cdot \big(
   \fhat_{\numobs / 2}^{(2)}(\State_\ell, \Action_\ell) - \fstar
   (\State_\ell, \Action_\ell) \big) \mid \datablock_2, \State_\ell
   \Big] = 0.
\end{align*}
So we have that:
\begin{align*}
\Exs \big[ T_{*, 1} T_1 \big] &= \frac{1}{2 \numobs} \Exs \Big[ \Big\{
  \frac{\weightfunc(\State, \Action)}{\propscore(\State, \Action)}
  \treateff(\State, \Action) - \taustar - \fstar(\State, \Action)
  \Big\} \cdot \big( \fhat_{\numobs / 2}^{(2)}(\State, \Action) -
  \fstar(\State, \Action) \big) \Big]\\ &= \frac{1}{2 \numobs} \Exs
\Big[ \Big( \inprod{\weightfunc(\State, \cdot)}{\treateff(\State,
    \cdot)} - \taustar \Big) \cdot \big( \fhat_{\numobs /
    2}^{(2)}(\State, \Action) - \fstar(\State, \Action) \big) \Big]
\\ &= \frac{1}{2 \numobs} \Exs \Big[ \Big( \inprod{\weightfunc(\State,
    \cdot)}{\treateff(\State, \cdot)} - \taustar \Big) \cdot \Exs
  \big[ \fhat_{\numobs / 2}^{(2)}(\State, \Action) - \fstar(\State,
    \Action) \mid \State, \datablock_2 \big] \Big] = 0.
\end{align*}

%%%%%%%%%%%%%%%%%%%%%%%%%%%%%%%%%%%%%%%%%%%%%%%%%%%%%%%%%%%%%%%%%%%

\subsection{Proof of~\Cref{cor:least-sqr-split-estimator}}
\label{subsec:proof-least-sqr-split-estimator}

Based on~\Cref{thm:upper-sample-split-general} and the discussion
thereafter, it suffices to prove an oracle inequality on the squared
error $\Exs \big[ \weightednorm{\muhat_{\numobs} - \treateff}^2
  \big]$.  So as to ease the notation, for any pair of functions $f,
g: \Xspace \times \actionspace \rightarrow \real$, we define the
empirical inner product
\begin{align*}
\inprod{f}{g}_\maux \mydefn \frac{1}{\maux} \sum_{i = 1}^\maux
\frac{\weightfunc^2(\State_i, \Action_i)}{\propscore^2(\State_i,
  \Action_i)} f(\State_i, \Action_i) g(\State_i, \Action_i), \quad
\mbox{and the induced norm $\vecnorm{f}{\maux} \mydefn
  \sqrt{\inprod{f}{f}_\maux}$.}
\end{align*}
With this notation, observe that our weighted least-squares estimator
is based on minimizing the objective $\vecnorm{\outcome -
  \plainmu}{\maux}^2 = \frac{1}{\maux} \sum_{i = 1}^\maux
\frac{\weightfunc^2(\State_i, \Action_i)}{\propscore^2(\State_i,
  \Action_i)} \big( \outcome_i - \plainmu(\State_i, \Action_i)
\big)^2$, where we have slightly overloaded our notation on $Y$---
viewing it as a function such that $\outcome(\State_i, \Action_i) =
\outcome_i$ for each $i$.

By the convexity of $\convSet$ and the optimality condition that
defines $\muhat_\maux$, for any function $\plainmu \in \funcClass$ and
scalar $\beta \in (0, 1)$, we have $\vecnorm{\outcome -
  \plainmu}{\maux}^2 \leq \vecnorm{\outcome_i - \big( t \plainmu +(1 -
  t) \muhat_\maux \big)}{\maux}^2$.  Taking the limit $t \rightarrow
0^+$ yields the basic inequality
\begin{align}
\label{eq:basic-ineq-for-constrained-least-sqr-proof}  
\vecnorm{\Delhat_\maux}{\maux}^2 & \leq \inprod{\treateff -
  \outcome}{\Delhat_\maux}_\maux +
\inprod{\Delhat_\maux}{\Deltil}_\maux,
\end{align}
where define the estimation error $\Delhat_\maux \mydefn \muhat_\maux
- \plainmu$, and the approximation error $\Deltil \mydefn \treateff -
\plainmu$.  By applying the Cauchy--Schwarz inequality to the last
term in
equation~\eqref{eq:basic-ineq-for-constrained-least-sqr-proof}, we
find that
\begin{align*}
\inprod{\Delhat_\maux}{\Deltil}_\maux \leq
\vecnorm{\Delhat_\maux}{\maux} \cdot \vecnorm{\Deltil}{\maux} \leq
\frac{1}{2} \vecnorm{\Delhat_\maux}{\maux}^2 + \frac{1}{2}
\vecnorm{\Deltil}{\maux}^2.
\end{align*}
Combining with
inequality~\eqref{eq:basic-ineq-for-constrained-least-sqr-proof}
yields the bound
\begin{align}
\label{eq:basic-ineq-modified-form-non-sharp}  
\vecnorm{\Delhat_\maux}{\maux}^2 \leq \frac{2}{\maux} \sum_{i =
  1}^\maux \OutNoise_i \frac{\weightfunc^2(\State_i, \Action_i)}{
  \propscore^2(\State_i, \Action_i)} \Delhat_\maux(\State_i,
\Action_i) + \vecnorm{\Deltil}{\maux}^2,
\end{align}
where $\OutNoise_i \defn \treateff(\State_i, \Action_i) - \outcome_i$
is the \emph{outcome noise} associated with observation $i$.

The remainder of our analysis involves controlling different terms in
the bound~\eqref{eq:basic-ineq-modified-form-non-sharp}.  There are
two key ingredients in the argument:
\begin{itemize}
\item First, we need to relate the empirical $\Ltwospace$-norm
  $\vecnorm{\cdot}{\maux}$ with its population counterpart
  $\weightednorm{\cdot}$.
  \Cref{lemma:relate-emp-norm-to-weighted-norm} stated below provides
  this control.
\item Second, using the Rademacher complexity $\SquaredRade_{\maux}$
  from equation~\eqref{eq:defn-rade-effective-noise}, we upper bound
  the weighted empirical average term associated with the outcome
  noise $\OutNoise_i = \treateff(\State_i, \Action_i) - \outcome_i$ on
  the right-hand-side of
  equation~\eqref{eq:basic-ineq-modified-form-non-sharp}.  This bound
  is given
  in~\Cref{lemma:relate-empirical-process-to-rade-complexity}.
\end{itemize}

Define the event
\begin{align}
  \label{EqnDefnEventOmega}
\Event_\omega & \defn \Big \{ \vecnorm{f}{\maux}^2 \geq
\frac{\smallballprob \smallballcon^2}{16} \weightednorm{f}^2 \quad
\mbox{for all $f \in \funcClass^* \setminus
  \ball_\omega(\radtwo_\maux)$} \Big \}.
\end{align}
The following result provides tail control on the complement of this
event.
\begin{lemma}
\label{lemma:relate-emp-norm-to-weighted-norm}
There exists a universal constant $c' > 0$ such that
\begin{align}
  \Prob(\Event_\omega^c) \leq \exp \big( -
  \tfrac{\smallballprob^2}{c'} \maux \big).
\end{align}
\end{lemma}
\noindent
See~\Cref{subsubsec:proof-lemma-relate-emp-norm-to-weighted-norm} for
the proof.

\medskip

For any (non-random) scalar $r > 0$, we also define the event
\begin{align*}
\Event(r) \mydefn \big\{ \weightednorm{\Delhat_\maux} \geq r \big\}.
\end{align*}
On the event $\Event_\omega \cap \Event(\radtwo_\maux)$, our original
bound~\eqref{eq:basic-ineq-modified-form-non-sharp} implies that
\begin{align}
\label{eq:basic-ineq-after-small-ball-argument}  
\weightednorm{\Delhat_\maux}^2 \leq \frac{32}{\smallballprob
  \smallballcon^2 \maux} \sum_{i = 1}^\maux \OutNoise_i
\frac{\weightfunc^2(\State_i, \Action_i)}{ \propscore^2(\State_i,
  \Action_i)} \Delhat_\maux(\State_i, \Action_i) +
\frac{16}{\smallballprob \smallballcon^2}\vecnorm{\Deltil}{\maux}^2.
\end{align}

In order to bound the right-hand-side of
equation~\eqref{eq:basic-ineq-after-small-ball-argument}, we need a
second lemma that controls the empirical process in terms of the
critical radius $\radone_\maux$ defined by the fixed point
relation~\eqref{eq:defn-critical-radius-squared}.
\begin{lemma}
\label{lemma:relate-empirical-process-to-rade-complexity}
We have
\begin{align}
\Exs \Big[ \bm{1}_{\Event(\radone_\maux)} \cdot \frac{2}{\maux}
  \sum_{i = 1}^\maux \OutNoise_i \tfrac{\weightfunc^2(\State_i,
    \Action_i)}{ \propscore^2(\State_i, \Action_i)}
  \Delhat_\maux(\State_i, \Action_i) \Big] \leq \radone_\maux
\sqrt{\Exs \big[ \weightednorm{\Delhat_\maux}^2 \big]}.
  \end{align}
\end{lemma}
\noindent
See~\Cref{subsubsec:proof-lemma-relate-empirical-process-to-rade-complexity}
for the proof. \\

\medskip

With these two auxiliary lemmas in hand, we can now complete the proof
of the theorem itself.  In order to exploit the basic
inequality~\eqref{eq:basic-ineq-after-small-ball-argument}, we begin
by decomposing the MSE as $\Exs \big[ \weightednorm{\Delhat_\maux}^2
  \big] \leq \sum_{j=1}^3 \Term_j$, where
\begin{align*}
\Term_1 \defn \Exs \big[ \weightednorm{\Delhat_\maux}^2
  \bm{1}_{\Event_\omega \cap \Event(\radtwo_\maux) \cap
    \Event(\radone_\maux)} \big], \quad \Term_2 \defn \Exs \big[
  \weightednorm{\Delhat_\maux}^2 \bm{1}_{[\Event(\radtwo_\maux) \cap
      \Event(\radone_\maux)]^c} \big], \quad \mbox{and} \quad \Term_3
\defn \Exs \big[\weightednorm{\Delhat_\maux}^2
  \bm{1}_{\Event_\omega^c} \big].
\end{align*}
We analyze each of these terms in turn.
\paragraph{Analysis of $\Term_1$:}
Combining the bound~\eqref{eq:basic-ineq-after-small-ball-argument}
with~\Cref{lemma:relate-empirical-process-to-rade-complexity} yields
\begin{subequations}
\begin{align}
\Term_1 & \leq \tfrac{32}{\smallballprob \smallballcon^2 \maux} \Exs
\Big[ \bm{1}_{\Event(\radtwo_\maux)} \cdot \sum_{i = 1}^\maux
  \OutNoise_i \frac{\weightfunc^2(\State_i, \Action_i)}{
    \propscore^2(\State_i, \Action_i)} \Delhat_\maux(\State_i,
  \Action_i) \Big] + \tfrac{16}{\smallballprob \smallballcon^2} \Exs
\big[ \vecnorm{\Deltil}{\maux}^2 \big] \nonumber \\
& \leq \tfrac{32}{\smallballprob \smallballcon^2} \radone_\maux
\sqrt{\Exs \big[ \weightednorm{\Delhat_\maux}^2 \big]} +
\tfrac{16}{\smallballprob \smallballcon^2} \Exs \big[
  \vecnorm{\Deltil}{\maux}^2 \big] \nonumber \\
\label{eq:first-term-in-mse-decomp}
& = \tfrac{32}{\smallballprob \smallballcon^2} \radone_\maux
\sqrt{\Exs \big[ \weightednorm{\Delhat_\maux}^2 \big]} +
\tfrac{16}{\smallballprob \smallballcon^2} \weightednorm{\Deltil}^2,
\end{align}
where the final equality follows since $\Exs \big[
  \vecnorm{\Deltil}{\maux}^2 \big] = \weightednorm{\Deltil}^2$, using
the definition of the empirical $\Ltwospace$-norm, and the fact that
the approximation error $\Deltil$ is a deterministic function.

\paragraph{Bounding $\Term_2$:}
On the event $[\Event(\radtwo_\maux) \cap \Event(\radone_\maux)]^c =
\Event^c(\radtwo_\maux) \cup \Event^c(\radone_\maux)$, we are
guaranteed to have $\weightednorm{\Delhat_\maux}^2 \leq
\radone_\maux^2 + \radtwo_\maux^2$, and hence
\begin{align}
\label{eq:second-term-in-mse-decomp}  
\Term_2 & \leq \radone_\maux^2 + \radtwo_\maux^2.
\end{align}

\paragraph{Analysis of $\Term_3$:}
Since the function class $\funcClass$ is bounded, we have
\begin{align}
\label{eq:third-term-in-mse-decomp}  
\Term_3 & \leq \diameter^2_\omega(\funcClass \cup \{\treateff\}) \cdot
\Prob \big( \Event_\omega^c \big) \leq \diameter^2_\omega(\funcClass
\cup \{\treateff\}) \cdot e^{- c \smallballprob^2 \maux}
\end{align}
\end{subequations}
for a universal constant $c > 0$.\\

\medskip

Finally, substituting the
bounds~\eqref{eq:first-term-in-mse-decomp},~\eqref{eq:second-term-in-mse-decomp}
and~\eqref{eq:third-term-in-mse-decomp} into our previous inequality
\mbox{$\Exs \big[ \weightednorm{\Delhat_\maux}^2 \big] \leq
  \sum_{j=1}^3 \Term_j$} yields
\begin{align*}
  \Exs \big[ \weightednorm{\Delhat_\maux^2} \big] \leq
  \frac{32}{\smallballprob \smallballcon^2 } \radone_\maux \sqrt{\Exs
    \big[ \weightednorm{\Delhat_\maux}^2 \big]} +
  \frac{16}{\smallballprob \smallballcon^2} \weightednorm{\Deltil}^2 +
  (\radone_\maux^2 + \radtwo_\maux^2) +\diameter^2_\omega(\funcClass
  \cup \{\treateff\}) \cdot e^{ - c \smallballprob^2 \maux}.
\end{align*}
Note that this is a self-bounding relation for the quantity $\Exs
\big[ \weightednorm{\Delhat_\maux^2} \big]$.  With the choice $\maux =
\numobs / 2$, it implies the the MSE bound
\begin{align*}
\Exs \big[ \weightednorm{\muhat_{\numobs / 2} - \treateff}^2 \big] &
\leq 2 \Exs \big[ \weightednorm{\muhat_{\numobs / 2} - \plainmu}^2
  \big] + 2 \Exs \big[ \weightednorm{ \Delhat_{\numobs / 2}}^2 \big]
\\
& \leq \big(2 + \tfrac{2 c'}{\smallballcon \smallballprob^2} \big)
\weightednorm{\Deltil}^2 + \tfrac{c'}{\smallballcon^2
  \smallballprob^4} \radone_{\numobs / 2}^2 + \tfrac{c'}{\smallballcon
  \smallballprob^2} \radtwo_{\numobs / 2}^2 +
\diameter^2_\omega(\funcClass \cup \{\treateff\}) \cdot e^{ - c
  \smallballprob^2 \numobs / 2},
\end{align*}
for a pair $(c, c')$ of positive universal constants.  Combining
with~\Cref{thm:upper-sample-split-general} and taking the infimum over
$\plainmu \in \funcClass$ completes the proof.

%%%%%%%%%%%%%%%%%%%%%%%%%%%%%%%%%%%%%%%%%%%%%%%%%%%%%%%%%%%%%%%%%%%

\subsubsection{Proof of Lemma~\ref{lemma:relate-emp-norm-to-weighted-norm}}
\label{subsubsec:proof-lemma-relate-emp-norm-to-weighted-norm}

The following lemma provides a lower bound on the empirical norm,
valid uniformly over a given function class $\funcClassTemp \subseteq
\big\{ h / \weightednorm{h} \: \mid \: h \in \funcClass^* \backslash
\{0\} \big\}$.
\begin{lemma}
\label{lemma:concentration-for-weighted-sqr-norm}
For a failure probability $\failprob \in (0,1)$, we have
\begin{align}
\inf_{h \in \funcClassTemp} \vecnorm{h}{\maux}^2 \geq
\frac{\smallballprob \smallballcon^2}{4} - 4 \smallballcon
\radeComplexity_\maux(\Hclass) - c \smallballcon^2 \cdot \Big\{
\sqrt{\tfrac{\log(1 / \failprob) }{\maux} } + \tfrac{\log(1 /
  \failprob)}{\maux} \Big \}
\end{align}
with probability at least $1 - \failprob$.
\end{lemma}
\noindent
See~\Cref{subsubsec:proof-lemma-concentration-for-weighted-sqr-norm}
for the proof of this lemma. \\

Taking it as given for now, we proceed with the proof
of~\Cref{lemma:relate-emp-norm-to-weighted-norm}. For any
deterministic radius $r > 0$, we define the set
\begin{align*}
\Hclass_r \mydefn \Big\{ h / \weightednorm{h} \: \mid \; h \in
\funcClass^*, \; \mbox{and} \; \weightednorm{h} \geq r \Big\}.
\end{align*}
By construction, the sequence $\{\Hclass_r \}_{r > 0}$ consists of
nested sets---that is, $\Hclass_r \subseteq \Hclass_s$ for $r >
s$---and all are contained within the set $\big\{ h / \weightednorm{h}
\, \mid \, h \in \funcClass \backslash \{0\} \big\}$. By convexity of
the class $\funcClass$, for any $h \in \funcClass$ such that
$\weightednorm{h} \geq r$, we have $r \; h / \weightednorm{h} \in
\funcClass \cap \ball(r)$. Consequently, we can bound the Rademacher
complexity as
\begin{align*}
    \radeComplexity_\maux(\Hclass_r) = \Exs \Big[ \sup_{h \in
        \funcClassTemp_r} \sum_{i = 1}^\maux \rade_i \frac{
        \weightfunc(\State_i, \Action_i) h(\State_i, \Action_i)
      }{\propscore(\State_i, \Action_i) } \Big] & \leq \frac{1}{r}
    \Exs \Big[ \sup_{h \in \funcClass^* \cap \ball_\omega(r)} \sum_{i
        = 1}^\maux \rade_i \frac{ \weightfunc(\State_i, \Action_i)
        h(\State_i, \Action_i) }{\propscore(\State_i, \Action_i) }
      \Big] \\
 & = \frac{1}{r} \radeComplexity_\maux(\funcClass^* \cap
    \ball_\omega(r)).
\end{align*}
By combining this inequality
with~\Cref{lemma:concentration-for-weighted-sqr-norm}, we find that
\begin{align}
\label{eq:ratio-lower-bound-in-relate-emp-norm-to-weighted-norm-proof}
\inf_{f \in \funcClass \setminus \ball_\omega(r)}
\frac{\vecnorm{f}{\maux}^2}{\weightednorm{f}^2} \geq
\frac{\smallballprob \smallballcon^2}{4} - \frac{4 \smallballcon}{r}
\radeComplexity_\maux \big( \funcClass^* \cap \ball_\omega(r) \big) -
c \smallballcon^2 \cdot \Big\{ \sqrt{\frac{\log(1 / \failprob)
  }{\maux} } + \frac{\log(1 / \failprob)}{\maux} \Big\}
\end{align}
with probability at least $1 - \failprob$.  This inequality is valid
for any deterministic radius $r > 0$.

By the definition~\eqref{eq:defn-critical-radius-plain} of the
critical radius $\radtwo_\maux$,
inequality~\eqref{eq:defn-critical-radius-plain} holds for any
$\radtwo > \radtwo_\maux$.  We now set $r = \radtwo_\maux$ in
equation~\eqref{eq:ratio-lower-bound-in-relate-emp-norm-to-weighted-norm-proof}.
Doing so allows us to conclude that given a sample size satisfying
$\maux \geq \frac{1024 c^2}{\smallballprob^2} \log(1 / \failprob)$, we
have
\begin{align*}
  \frac{4 \smallballcon}{r_\maux} \radeComplexity_\maux \big(
  \funcClass^* \cap \ball_\omega(r_\maux) \big) \leq
  \frac{\smallballprob \smallballcon^2}{16}, \quad \mbox{and} \quad c
  \smallballcon^2 \cdot \Big\{ \sqrt{\frac{\log(1 / \failprob)
    }{\maux} } + \frac{\log(1 / \failprob)}{\maux} \Big\} \leq
  \frac{\smallballprob \smallballcon^2}{16}.
\end{align*}
Combining with
equation~\eqref{eq:ratio-lower-bound-in-relate-emp-norm-to-weighted-norm-proof}
completes the proof of~\Cref{lemma:relate-emp-norm-to-weighted-norm}.

\subsubsection{Proof of Lemma~\ref{lemma:relate-empirical-process-to-rade-complexity}}\label{subsubsec:proof-lemma-relate-empirical-process-to-rade-complexity}

Recall our notation $\OutNoise_i \defn \treateff(\State_i, \Action_i)
- \outcome_i$ for the outcome noise.  Since the set $\convSet$ is
convex, on the event $\Event(\radone_\maux)$, we have
\begin{align}
\label{eq:monotonicity-of-emp-proc-based-on-cvx-in-rade-proof}  
 \tfrac{1}{\weightednorm{\Delhat_\maux}} \; \sum_{i = 1}^\maux
 \OutNoise_i \frac{\weightfunc^2(\State_i, \Action_i)}{
   \propscore^2(\State_i, \Action_i)} \Delhat_\maux(\State_i,
 \Action_i) & \leq \tfrac{1}{\radone_\maux} \; \sup_{h \in
   \funcClass^* \cap \ball_\omega(\radone_\maux)} \sum_{i = 1}^\maux
 \OutNoise_i \frac{\weightfunc^2(\State_i, \Action_i)}{
   \propscore^2(\State_i, \Action_i)} h(\State_i, \Action_i).
\end{align}
Define the empirical process supremum
\begin{align*}
Z_\maux(\radone_\maux) & \mydefn \sup_{h \in \funcClass^* \cap
  \ball_\omega(\radone_\maux)} \frac{1}{\maux} \sum_{i = 1}^\maux
\OutNoise_i \frac{\weightfunc^2(\State_i, \Action_i)}{
  \propscore^2(\State_i, \Action_i)} h(\State_i, \Action_i).
\end{align*}
Since the all-zeros function $0$ is an element of $\funcClass^* \cap
\ball_\omega(\radone_\maux)$, we have $Z_\maux (\radone_\maux) \geq
0$. Equation~\eqref{eq:monotonicity-of-emp-proc-based-on-cvx-in-rade-proof}
implies that
\begin{align}
\Exs \Big[ \bm{1}_{\Event(\radone_\maux)} \cdot \tfrac{2}{\maux}
  \sum_{i = 1}^\maux \tfrac{\weightfunc^2(\State_i, \Action_i)}{
    \propscore^2(\State_i, \Action_i)} \big(\treateff(\State_i,
  \Action_i) - \outcome_i \big) \Delhat_\maux(\State_i, \Action_i)
  \Big] & \leq \Exs \Big[
  \frac{\weightednorm{\Delhat_\maux}}{\radone_\maux}
  Z_\maux(\radone_\maux) \Big] \nonumber \\
\label{eq:cauchy-schwarz-in-emp-proce-rade-proof}
& \leq \sqrt{\Exs \big[ \weightednorm{\Delhat_\maux}^2 \big]} \cdot
\sqrt{\radone_\maux^{-2} \Exs \big[ Z_\maux^2(\radone_\maux) \big]},
\end{align}
where the last step follows by applying the Cauchy--Schwarz
inequality.

Define the symmetrized random variable
\begin{align*}
\SymVar_\maux(\radone_\maux) \mydefn \sup_{h \in \funcClass^* \cap
  \ball_\omega(\radone_\maux)} \frac{1}{\maux} \sum_{i = 1}^\maux
\rade_i \frac{\weightfunc^2(\State_i, \Action_i)}{
  \propscore^2(\State_i, \Action_i)} \big(\treateff(\State_i,
\Action_i) - \outcome_i \big) h(\State_i, \Action_i),
\end{align*}
where $\{\rade_i\}_{i = 1}^\maux$ is an $\mathrm{i.i.d.}$ sequence of
Rademacher variables, independent of the data. By a standard
symmetrization argument (e.g., \S 2.4.1 in the
book~\cite{wainwright2019high}), there are universal constants $(c,
c')$ such that
\begin{align*}
    \Prob\big[Z_\maux(\radone_\maux) > t \big] & \leq c' \Prob
    \big[\SymVar_\maux(\radone_\maux) > c t \big], \quad \mbox{for any
      $t > 0$.}
\end{align*}
Integrating over $t$ yields the bound
\begin{align*}
\Exs[Z^2_\maux(\radone_\maux)] \leq c^2 c' \Exs
    [\SymVar^2_\maux(\radone_\maux)] = c^2 c'
    \SquaredRade^2_\maux(\radone_\maux) \; \stackrel{(i)}{=} c^2 \, c'
    \radone_\maux^2,
\end{align*}
where equality (i) follows from the definition of $\radone_\maux$.
Substituting this bound back into
equation~\eqref{eq:cauchy-schwarz-in-emp-proce-rade-proof} completes
the proof of
Lemma~\ref{lemma:relate-empirical-process-to-rade-complexity}.

%%%%%%%%%%%%%%%%%%%%%%%%%%%%%%%%%%%%%%%%%%%%%%%%%%%%%%%%%%%%%%%%

\subsection{Proof of~\Cref{thm:normal-approx-least-sqr}}
\label{subsec:proof-thm-normal-approx}

Define the function $\widebar{f}(\state, \action) \mydefn
\frac{\weightfunc(\state, \action)}{\propscore(\state, \action)}
\mubar(\state, \action) - \actinprod{\weightfunc(\state,
  \cdot)}{\mubar (\state, \cdot)}$, which would be optimal if $\mubar$
were the true treatment function.  It induces the estimate
\begin{align*}
\tauhat_{\numobs, \bar{f}} = \frac{1}{\numobs} \sum_{i = 1}^\numobs
\Big\{ \frac{\weightfunc(\State_i, \Action_i)}{\propscore(\State_i,
  \Action_i)} \big( \outcome_i - \mubar(\State_i, \Action_i) \big) +
\actinprod{\weightfunc(\State_i, \cdot)}{\mubar(\State_i, \cdot)}
\Big\},
\end{align*}
which has ($\numobs$-rescaled) variance $\numobs \cdot \Exs \big[
  \abss{\tauhat_{\numobs, \bar{f}} - \taustar}^2 \big] = \vstar^2 +
v^2(\mubar)$, where $\vstar^2$ is the efficient variance, and
$v^2(\mubar) \defn \var \Big( \frac{\weightfunc(\State,
  \Action)}{\propscore(\State, \Action)} \cdot \big( \treateff -
\mubar \big)(\State, \Action) \Big)$.  Let us now compare two-stage
estimator $\tauhat_{\numobs}$ with this idealized estimator.  We have
\begin{align*}
\Exs[|\tauhat_{\numobs, f} - \tauhat_{\numobs}|^2] & \leq
\frac{2}{\numobs^2} \Exs \Big[ \big| \sum_{i = 1}^{\numobs /
    2}(\widebar{f} - \fhat^{(2)}_{\numobs/2})(\State_i, \Action_i)
  \big|^2 \Big] + \frac{2}{\numobs} \Exs \Big[ \big| \sum_{i = \numobs
    / 2 + 1}^{\numobs}(\widebar{f} - \fhat^{(1)}_{\numobs /
    2})(\State_i, \Action_i) \big|^2 \Big]\\ &\leq \frac{4}{\numobs}
\Exs \big[ \weightednorm{\muhat_{\numobs / 2} - \mubar}^2 \big].
\end{align*}
Thus, we are guaranteed the Wasserstein bound
\begin{align*}
\Wass_1 \big( \sqrt{\numobs} \tauhat_\numobs, \sqrt{\numobs}
\tauhat_{\numobs, \bar{f}} \big) & \leq \Wass_2 \big( \sqrt{\numobs}
\tauhat_\numobs, \sqrt{\numobs} \tauhat_{\numobs, \bar{f}} \big) \leq
2 \sqrt{\Exs \big[ \weightednorm{\muhat_{\numobs / 2} - \mubar}^2
    \big]}.
\end{align*}
Consequently, by the triangle inequality for the Wasserstein distance,
it suffices to establish a normal approximation guarantee for the
idealized estimator $\tauhat_{\numobs, \bar{f}}$, along with control
on the error induced by approximating the function $\mubar$ using an
empirical estimator.

\paragraph{Normal approximation for $\tauhat_{\numobs, \bar{f}}$:}
We make use of the following non-asymptotic central limit theorem:
\begin{proposition}[\cite{ross2011fundamentals}, Theorem 3.2 (restated)]
\label{prop:non-asym-clt}
Given $\mathrm{i.i.d.}$ zero-mean random variables $\{\State_i
\}_{i=1}^\numobs$ with finite fourth moment, the rescaled sum
$W_\numobs \mydefn \sum_{i=1}^\numobs \State_i/\sqrt{\numobs}$
satisfies the Wasserstein bound
\begin{align*}
\Wass_1 \big( W_\numobs, Z \big) & \leq \tfrac{1}{\sqrt{\numobs}} \Big
\{ \tfrac{\Exs[|\State_1|^3]}{\Exs[\State_1^2]} + \sqrt{\tfrac{2
    \Exs[\State_1^4]}{\pi \Exs[\State_1^2]}} \Big \} \qquad
\mbox{where $Z \sim \mathcal{N}(0, \Exs[\State_1^2])$.}
\end{align*}   
\end{proposition}
Since we have $\Exs[|\State_1|^3 ] \leq \sqrt{\Exs[\State_1^2] \cdot
  \Exs[\State_1^4]}$, this bound implies that $\Wass_1 \big(W_\numobs,
Z \big) \leq \tfrac{2}{\sqrt{\numobs}} \cdot
\sqrt{\Exs[\State_1^4]/\Exs[\State_1^2]}$.  Applying this bound to the
empirical average $\tauhat_{\numobs, \bar{f}}$ yields
\begin{align*}
\Wass_1 \big(\sqrt{\numobs} \tauhat_{\numobs, \bar{f}}, \mathcal{Z}
\big) \leq \tfrac{2}{\sqrt{\numobs}} \cdot \sqrt{ \tfrac{\MomFour}{
    \vstar^2 + v^2(\mubar) }},
\end{align*}
as claimed.

\paragraph{Bounds on the estimation error $\weightednorm{\muhat_{\numobs / 2} - \mubar}$:}

From the proof of~\Cref{cor:least-sqr-split-estimator}, recall the
basic
inequality~\eqref{eq:basic-ineq-for-constrained-least-sqr-proof}---viz.
\begin{align}
\label{EqnPreBasic}
\vecnorm{\Delhat_\maux}{\maux}^2 \leq \frac{1}{\maux} \sum_{i =
  1}^\maux \OutNoise_i \tfrac{\weightfunc^2(\State_i, \Action_i)}{
  \propscore^2(\State_i, \Action_i)} \Delhat_\maux(\State_i,
\Action_i) + \inprod{\Delhat_\maux}{\Deltil}_\maux,
\end{align}
where $\OutNoise_i = \treateff(\State_i, \Action_i) - \outcome_i$ is
the outcome noise.

As before, we define the approximation error $\Deltil \mydefn
\treateff - \mubar$.  Since $\mubar = \arg \min_{h \in \funcClass}
\weightednorm{h - \treateff}$ is the projection of $\treateff$ onto
$\funcClass$, and $\muhat_\maux \in \funcClass$ is feasible for this
optimization problem, the first-order optimality condition implies
that $\weightedinprod{\Delhat_\maux}{\Deltil} \leq 0$.  By adding this
inequality to our earlier bound~\eqref{EqnPreBasic} and re-arranging
terms, we find that
\begin{align}
\vecnorm{\Delhat_\maux}{\maux}^2 \leq \frac{1}{\maux} \sum_{i =
  1}^\maux \tfrac{\weightfunc^2(\State_i,
  \Action_i)}{\propscore^2(\State_i, \Action_i)} \big(\treateff
(\State_i, \Action_i) - \outcome_i \big) \Delhat_\maux(\State_i,
\Action_i) + \big( \inprod{\Delhat_\maux}{\Deltil}_\maux -
\weightedinprod{\Delhat_\maux}{\Deltil}
\big).\label{eq:basic-ineq-with-exact-projection}
\end{align}
Now define the empirical process suprema
\begin{align*}
    Z_\maux(\radtwo) & \mydefn \sup_{h \in \funcClass^* \cap
      \ball_\omega(\radtwo)} \frac{1}{\maux} \sum_{i = 1}^\maux
    \frac{\weightfunc^2(\State_i, \Action_i)}{ \propscore^2(\State_i,
      \Action_i)} \big(\treateff(\State_i, \Action_i) - \outcome_i
    \big) h(\State_i, \Action_i), \quad \mbox{and} \\
\Zprime_\maux(\radone) & \mydefn \sup_{h \in \funcClass^* \cap
  \ball_\omega(\radone)} \frac{1}{\maux} \sum_{i = 1}^\maux \Big(
\frac{\weightfunc^2(\State_i, \Action_i)}{ \propscore^2(\State_i,
  \Action_i)} h(\State_i, \Action_i) \Deltil(\State_i, \Action_i) -
\weightedinprod{\Deltil}{h} \Big).
\end{align*}
From the proof of~\Cref{cor:least-sqr-split-estimator}, recall the
events
\begin{align*}
 \Event_\omega \mydefn \Big\{ \vecnorm{f}{\maux}^2 \geq
 \frac{\smallballprob \smallballcon^2}{16} \weightednorm{f}^2, \quad
 \mbox{for any } f \in \funcClass^* \setminus
 \ball_\omega(\radtwo_\maux) \Big\}, \quad \mbox{and} \quad \Event(r)
 & \mydefn \Big\{ \weightednorm{\Delhat_\maux} \geq r \Big\}.
\end{align*}

Introduce the shorthand $\utmp_\maux = \max \{ \radtwo_\maux,
\radone_\maux, \dradone_\maux \}$.  On the event $\Event_\omega \cap
\Event(\utmp_\maux)$, the basic
inequality~\eqref{eq:basic-ineq-with-exact-projection} implies that
\begin{align*}
\frac{\smallballprob \smallballcon^2}{16}
\weightednorm{\Delhat_\maux}^2 \leq \vecnorm{\Delhat_\maux}{\maux}^2 &
\stackrel{(i)}{\leq} Z_\maux(\weightednorm{\Delhat_\maux}) +
\Zprime_\maux(\weightednorm{\Delhat_\maux})\\
& \stackrel{(ii)}{\leq}
\frac{\weightednorm{\Delhat_\maux}}{\radtwo_\maux}
Z_\maux(\radtwo_\maux) +
\frac{\weightednorm{\Delhat_\maux}}{\radone_\maux}
\Zprime_\maux(\radone_\maux),
\end{align*}
where step (ii) follows from the non-increasing property of the
functions $r \mapsto r^{-1} Z_\maux(r)$ and $\radone \mapsto
\radone^{-1} \Zprime_\maux(\radone)$.

So there exists a universal constant $c > 0$ such that
\begin{align*}
    \Exs \Big[\weightednorm{\Delhat_\maux}^2 \bm{1}_{\Event_\omega
        \cap \Event (\utmp_\maux)}\Big] \leq \frac{c}{\smallballprob^2
      \smallballcon^4} \Big\{ \frac{1}{\radone_\maux^2} \Exs
    \big[Z_\maux^2(\radone_\maux) \big] + \frac{1}{\dradone^2_\maux}
    \Exs \big[ \big\{\Zprime_\maux(\dradone_\maux) \big \}^2 \big]
    \Big\}.
\end{align*}
Via the same symmetrization argument as used in the proof
of~\Cref{cor:least-sqr-split-estimator}, there exists a universal
constant $c > 0$ such that
\begin{align*}
\Exs[Z_\maux^2(\radone_\maux)] \leq c \SquaredRade^2_\maux \big(
\funcClass^* \cap \ball_\omega(\radone_\maux) \big), \quad \mbox{and}
\quad \Exs\Big[ \big(\Zprime_\maux(\dradone_\maux) \big)^2 \Big] \leq
c \SquaredDiffRade^2_\maux \big( \funcClass^* \cap
\ball_\omega(\dradone_\maux) \big).
\end{align*}
By the definition of the critical radius $\radone_\maux$, we have
\begin{align*}
\frac{1}{\radone_\maux} \SquaredRade_\maux \big( \funcClass^* \cap
\ball_\omega(\radone_\maux) \big) = \radone_\maux, \quad \mbox{and}
\quad \frac{1}{\dradone_\maux} \SquaredDiffRade_\maux
\big(\funcClass^* \cap \ball_\omega(\dradone_\maux) \big) =
\dradone_\maux.
\end{align*}
Combining with the moment bound above, we arrive at the conclusion:
\begin{align*}
    \Exs[\weightednorm{\Delhat_\maux}^2] & \leq \Exs
    \big[\weightednorm{\Delhat_\maux}^2 \bm{1}_{\Event_\omega \cap
        \Event (\utmp_\maux)} \big] + \Exs
    \big[\weightednorm{\Delhat_\maux}^2 \bm{1}_{\Event (\utmp_\maux)^c
      } \big] + \Exs \big[\weightednorm{\Delhat_\maux}^2
      \bm{1}_{\Event_\omega^c} \big] \\
& \leq \Big( 1 + \frac{c}{\smallballprob^4 \smallballcon^2} \Big)
    \cdot \big( \radtwo_\maux^2 + \radone_\maux^2 +\dradone_\maux^2
    \big) + \diameter^2_\omega(\funcClass \cup \{\treateff\}) \cdot
    \Prob(\Event_\omega^c) \\
& \leq \Big( 1 + \frac{c}{\smallballprob^4 \smallballcon^2} \Big)
    \cdot \big( \radtwo_\maux^2 + \radone_\maux^2 +\dradone_\maux^2
    \big) + \diameter_\omega^2(\funcClass) \cdot e^{- c
      \smallballprob^2 \maux}.
\end{align*}
Substituting into the Wasserstein distance bound completes the proof
of~\Cref{thm:normal-approx-least-sqr}.

%%%%%%%%%%%%%%%%%%%%%%%%%%%%%%%%%%%%%%%%%%%%%%%%%%%%%%%%%%%%%%%%
\section{Proofs of minimax lower bounds}

In this section, we prove the two minimax lower bounds---namely,
~\Cref{thm:minimax-local} and ~\Cref{thm:worst-case-shattering-dim}.

%%%%%%%%%%%%%%%%%%%%%%%%%%%%%%%%%%%%%%%%%%%%%%%%%%%%%%%%%%%%%%%

\subsection{Proof of~\Cref{thm:minimax-local}}
\label{subsec:proof-thm-minimax-local}

It suffices to show that the minimax risk $\minimaxRisk \equiv
\minimaxRisk\big( \class_{\delta}(\probInstanceStar) \big)$ satisfies
the following three lower bounds:
\begin{subequations}
  \begin{align}
\label{eq:minimax-lower-bound-local-v}     
\minimaxRisk & \geq \frac{c}{\numobs} \var_{\probxstar} \big(
\actinprod{\weightfunc(\State, \cdot)}{\treateff(\State, \cdot)} \big)
\quad \mbox{for $\numobs \geq 4 (\ctwofour')^2$,} \\
\label{eq:minimax-lower-bound-local-sigma}
\minimaxRisk & \geq \frac{c}{\numobs} \weightednorm{\sigma}^2 \quad
\mbox{for $\numobs \geq 16$,} \\
\label{eq:minimax-lower-bound-local-delta}    
\minimaxRisk & \geq \frac{c}{\numobs} \weightednorm{\delta}^2 \quad
\mbox{for $\numobs \in \big[ \ctwofour^2 , c'|\Xspace| / c_{\max}
    \big]$.}
\end{align}
\end{subequations}
Given these three inequalities, the minimax risk $\minimaxRisk$ can be
lower bounded by the average of the right-hand side quantities,
assuming that $\numobs$ is sufficiently large.  Since $c$ is a
universal constant, these bounds lead to the conclusion
of~\Cref{thm:minimax-local}.

Throughout the proof, we use $\Prob_{\treateff, \probx}$ to denote the
law of a sample $(\State, \Action, \outcome)$ under the problem
instance defined by outcome function $\treateff$ and data distribution
$\probx$. We further use $\Prob_{\treateff, \probx}^{\otimes \numobs}$
to denote its $\numobs$-fold product, as is appropriate given our
$\mathrm{i.i.d.}$ data $(\State_i, \Action_i, \outcome_i)_{i =
  1}^\numobs$.

%%%%%%%%%%%%%%%%%%%%%%%%%%%%%%%%%%%%%%%%%%%%%%%%%%%%%%%%%%%%%%%%%%%%

\subsubsection{Proof of the lower bound~\eqref{eq:minimax-lower-bound-local-v}}

The proof is based on Le Cam's two-point method: we construct a family
of probability distributions $\{ \probxtweak \: \mid \; \tweak > 0
\}$, each contained in the local neighborhood $\Nprob(\probxstar)$.
We choose the parameter $\tweak$ small enough to ensure that the
probability distributions $\Prob_{\probxtweak, \treateffzero}^{\otimes
  \numobs}$ and $\Prob_{\probxstar, \treateffzero}^{\otimes \numobs}$
are ``indistinguishable'', but large enough to ensure that the
functional values $\tau(\probxtweak, \treateffzero)$ and
$\tau(\probxstar, \treateffzero)$ are well-separated.  See \S
15.2.1--15.2.2 in the book~\cite{wainwright2019high} for more
background.

More precisely, Le Cam's two-point lemma guarantees that for any
distribution $\probxtweak \in \Nprob(\probxstar)$, the minimax risk is
lower bounded as
\begin{align}
\label{EqnLeCamTwoPoint}
\minimaxRisk \geq \frac{1}{4} \Big\{1 - \totalvariation
\Big(\Prob_{\treateffzero, \probxtweak}^{\otimes
  \numobs},\Prob_{\treateffzero, \probxstar}^{\otimes \numobs} \Big)
\Big\} \cdot \big \{ \avgtreat(\probxtweak, \treateffzero) -
\avgtreat(\probxstar, \treateffzero) \big \}^2,
\end{align}

Recall that throughout this section, we work with the sample size
lower bound
\begin{align}
    \numobs \geq 4
    (\ctwofour')^2.\label{eq:sample-size-req-in-minimax-v-lower-bound}
\end{align}
Now suppose that under the
condition~\eqref{eq:sample-size-req-in-minimax-v-lower-bound}, we can
exhibit a choice of $\tweak$ within the family $\{\probxtweak \mid
\tweak > 0 \}$ such that the functional gap satisfies the lower bound
\begin{subequations}
\begin{align}
\label{EqnPartAFuncLower}
 \avgtreat(\probxtweak, \treateffzero) - \avgtreat(\probxstar,
 \treateffzero) \geq \frac{1}{16 \sqrt{\numobs}} \sqrt{\var\big(
   \actinprod{\weightfunc(\State, \cdot)}{\treateff(\State, \cdot)}
   \big)},
\end{align}
whereas the TV distance satisfies the upper bound
\begin{align}
\label{EqnPartATVUpper}  
\totalvariation \Big(\Prob_{\treateffzero, \probxtweak}^{\otimes
  \numobs},\Prob_{\treateffzero, \probxstar}^{\otimes \numobs} \Big)
\leq \frac{1}{3}.
\end{align}
\end{subequations}
These two inequalities, in conjunction with Le Cam's two point
bound~\eqref{EqnLeCamTwoPoint}, imply the claimed lower
bound~\eqref{eq:minimax-lower-bound-local-v}. \\

With this overview in place, it remains to define the family
$\{\probxtweak \mid \tweak > 0 \}$, and prove the
bounds~\eqref{EqnPartAFuncLower} and~\eqref{EqnPartATVUpper}.

\paragraph{Family of perturbations:} Define
the real-valued function
\begin{align*}
h(\state) \mydefn \actinprod{\treateff(\state,
  \cdot)}{\weightfunc(\state, \cdot)} - \Exs_{\probxstar} \big[
  \actinprod{\treateff(\State, \cdot)}{\weightfunc(\State, \cdot)}
  \big],
\end{align*}
along with its truncated version
\begin{align*}
  \htrunc(\state) & \mydefn
\begin{dcases} h(\state) & \quad \mbox{if
      $|h(\state)| \leq 2 \ctwofour' \cdot \sqrt{\Exs_{\probxstar}
      [h^2(\State)]}$, and} \\
 \mathrm{sgn}(h(\state)) \cdot \sqrt{\Exs_{\probxstar}[h^2(\State)]} &
 \quad \mbox{otherwise.}
\end{dcases}
\end{align*}
For each $\tweak > 0$, we define the \emph{tilted probability measure}
\begin{align*}
  \probxtweak(\state) & \mydefn Z_\tweak^{-1} \probxstar(\state) \exp
  \big(\tweak \htrunc(\state) \big), \quad \mbox{where $Z_\tweak =
    \sum_{x \in \Xspace} \probxstar(\state) \exp \big(\tweak
    \htrunc(\state) \big)$.}
\end{align*}
It can be seen that the tilted measure satisfies the bounds
\begin{align*}
  \exp \big( - \tweak \vecnorm{\htrunc}{\infty} \big) \leq
  \frac{\probxtweak(\state)}{\probxstar(\state)} \leq \exp \big(
  \tweak \vecnorm{\htrunc}{\infty} \big) \qquad \mbox{for any $\state
    \in \Xspace$,}
\end{align*}
whereas the normalization constant is sandwiched as
\begin{align*}
\exp \big( - \tweak \vecnorm{\htrunc}{\infty} \big) \leq Z_\tweak \leq
\exp \big( \tweak \vecnorm{\htrunc}{\infty} \big).
\end{align*}
Throughout this section, we choose
\begin{subequations}
\begin{align}
  \label{EqnDefnTweak}
\tweak \defn ( 4 \vecnorm{\htrunc}{\Ltwospace(\probxstar)}
\sqrt{\numobs})^{-1},
\end{align}
\mbox{which ensures that}
\begin{align}
\tweak \vecnorm{\htrunc}{\infty} = \frac{1}{4 \sqrt{\numobs}} \cdot
\frac{\vecnorm{\htrunc}{\infty}}{\vecnorm{\htrunc}{\Ltwospace(\probxstar)}}
\stackrel{(i)}{\leq} \frac{1}{\sqrt{8 \numobs}} \frac{2 \ctwofour'
  \vecnorm{h}{\Ltwospace(\probxstar)}}{\vecnorm{h}{\Ltwospace(\probxstar)}}
\stackrel{(ii)}{\leq} \frac{1}{8},
\end{align}
\end{subequations}
where step (i) follows from the definition of the truncated function
$\htrunc$, and step (ii) follows from the sample size
condition~\eqref{eq:sample-size-req-in-minimax-v-lower-bound}.

\paragraph{Proof of the lower bound~\eqref{EqnPartAFuncLower}:}
First we lower bound the gap in the functional.  We have
\begin{align}
\avgtreat(\probxtweak, \treateffzero) - \avgtreat(\probxstar,
\treateffzero) &= \Exs_{\probxtweak}[\actinprod{\treateff(\State,
    \cdot)}{\weightfunc(\State, \cdot)}] -
\Exs_{\probxstar}[\actinprod{\treateff(\State,
    \cdot)}{\weightfunc(\State, \cdot)}] \nonumber \\
\label{eq:gap-expression-in-v-minimax-lower-bound}
& = \Exs_{\probxstar} \Big[ h(\State) e^{\tweak \htrunc(\State)} \Big]
~/~ \Exs_{\probxstar} \Big[ e^{\tweak \htrunc(\State)} \Big].
\end{align}
Note that $\abss{\tweak \htrunc (\State) } \leq 1/8$ almost surely by
construction. Using the elementary inequality $\abss{e^z - 1 - z} \leq
z^2$, valid for all $z \in [-1/4, 1/4]$, we obtain the lower bound
\begin{align}
  \Exs_{\probxstar} \Big[ h(\State) e^{\tweak \htrunc(\State)} \Big] &
  \geq \Exs_{\probxstar} \big[ h(\State) \big] + \tweak
  \Exs_{\probxstar} \big[ h(\State) \htrunc(\State) \big] - \tweak^2
  \Exs_{\probxstar} \big[ |h(\State)| \cdot |\htrunc(\State)|^2
    \big]\label{eq:gap-bound-elementary-in-minimax-v-proof}
\end{align}
Now we study the three terms on the right-hand-side of
equation~\eqref{eq:gap-bound-elementary-in-minimax-v-proof}. By
definition, we have $\Exs_{\probxstar} [h(\State)] = 0$. Since the
quantities $h(\State)$ and $\htrunc(\State)$ have the same sign almost
surely, the second term admits a lower bound
\begin{align*}
  \Exs_{\probxstar} \big[ h(\State) \htrunc(\State) \big] \geq \Exs
  \big[ \htrunc^2 (\State) \big] \geq \frac{1}{2} \Exs \big[ h^2
    (\State) \big],
\end{align*}
where the last step follows from
Lemma~\ref{lemma:simple-moment-lower-bound-under-trunc}.

Focusing on the third term in the
decomposition~\eqref{eq:gap-bound-elementary-in-minimax-v-proof}, we
note that Cauchy-Schwarz inequality yields
\begin{align*}
    \Exs_{\probxstar} \big[ |h(\State)| \cdot |\htrunc(\State)|^2
      \big] \leq \sqrt{\Exs \big[ h^2 (\State) \big] } \cdot
    \sqrt{\Exs \big[ h^4 (\State) \big]} \leq \sqrt{\ctwofour'} \cdot
    \Big\{ \Exs \big[ h^2 (\State) \big] \Big\}^{3/2},
\end{align*}
where the last step follows from the definition of the constant
$\ctwofour'$.

Combining these bounds with
equation~\eqref{eq:gap-bound-elementary-in-minimax-v-proof} and
substituting the choice~\eqref{EqnDefnTweak} of the parameter
$\tweak$, we obtain the following lower bound on the functional gap
\begin{align*}
\Exs_{\probxstar} \Big[ h(\State) e^{\tweak \htrunc(\State)} \Big] &
\geq \tweak \vecnorm{h}{\Ltwospace (\probxstar)}^2 - \tweak^2
\sqrt{\ctwofour'} \vecnorm{h}{\Ltwospace (\probxstar)}^3 \\
& \geq \frac{1}{8 \sqrt{\numobs}} \vecnorm{h}{\Ltwospace (\probxstar)}
- \frac{\sqrt{\ctwofour'}}{16 \numobs} \vecnorm{h}{\Ltwospace
  (\probxstar)}\\ &\geq \frac{3}{32 \sqrt{\numobs}}
\vecnorm{h}{\Ltwospace (\probxstar)},
\end{align*}
where the last step follows because $\numobs \geq 4(\ctwofour')^2$.

On the other hand, since $\abss{\tweak \htrunc (\State)} \leq 1/ 8$
almost surely, we have $ \Exs_{\probxstar} \big[ e^{\tweak
    \htrunc(\State)} \big] \leq 3/2$. Combining with the bound above
and substituting into the
expression~\eqref{eq:gap-expression-in-v-minimax-lower-bound}, we find
that we find that
\begin{align*}
 \avgtreat(\probxtweak, \treateffzero) - \avgtreat(\probxstar,
 \treateffzero) \geq \frac{3}{32 \sqrt{\numobs}}
 \vecnorm{h}{\Ltwospace(\probxstar)} ~/~ \Exs_{\probxstar} \big[
   e^{\tweak \htrunc(\State)} \big] \geq \frac{1}{16 \sqrt{\numobs}}
 \vecnorm{h}{\Ltwospace(\probxstar)},
\end{align*}
which is equivalent to the claim~\eqref{EqnPartAFuncLower}.

%%%%%%%%%%%%%%%%%%%%%%%%%%%%%%%%%%%%%%%%%%%%%%%%%%%%%%%%%%%%%%%%

\paragraph{Proof of the upper bound~\eqref{EqnPartATVUpper}:}
Pinsker's inequality ensures that
\begin{align}
\label{EqnPinsker}
\totalvariation \Big(\Prob_{\treateffzero, \probxtweak}^{\otimes
  \numobs},\Prob_{\treateffzero, \probxstar}^{\otimes \numobs} \Big)
\leq \sqrt{\frac{1}{2} \chisqdiv{\Prob_{\treateffzero,
      \probxtweak}^{\otimes \numobs}}{\Prob_{\treateffzero,
      \probxstar}^{\otimes \numobs}}},
\end{align}
so that it suffices to bound the $\chi^2$-divergence.  Beginning with
the divergence between $\probxtweak$ and $\probxstar$ (i.e., without
the tensorization over $\numobs$), we have
\begin{align}
\chisqdiv{\probxtweak}{\probxstar} = \var_{\probxstar}
\Big(\probxtweak(\State) / \probxstar(\State) \Big) & =
\frac{1}{Z_\tweak^2} \: \var_{\probxstar} \big( e^{\tweak
  \htrunc(\State)} - 1 \big) \nonumber \\
& \leq \exp \big( 2 \tweak \vecnorm{\htrunc}{\infty} \big) \cdot
\Exs_{\probxstar} \big[ |e^{\tweak \htrunc(\State)} - 1|^2 \big]
\nonumber \\
\label{EqnSitar}
& \leq \exp \big( 4 \tweak \vecnorm{\htrunc}{\infty} \big) \cdot
\tweak^2 \Exs_{\probxstar} \big[ \htrunc^2(\State) \big].
\end{align}
where the last step follows from the elementary inequality $|e^x - 1|
\leq e^{|x|} \cdot |x|$, valid for any $x \in \real$. Given the choice
of tweaking parameter $\tweak$, we have $\exp \big( 4 \tweak
\vecnorm{\htrunc}{\infty} \big) \leq 2$.

The definition of the truncated function $\htrunc$ implies that
$\Exs_{\probxstar} \big[ \htrunc^2(\State) \big] \leq
\Exs_{\probxstar}[h^2(\State)]$.  Combining this bound with our
earlier inequality~\eqref{EqnSitar} yields
\begin{align*}
\chisqdiv{\probxtweak}{\probxstar} \leq 2 \tweak^2 \Exs_{\probxstar}
\big[ \htrunc^2(\State) \big] \leq \frac{1}{8 \numobs},
\end{align*}
which certifies that $\probxtweak \in \Nprob(\probxstar)$, as required
for the validity of our construction.

Finally, by the tensorization property of the $\chi^2$-divergence, we
have
\begin{align*}
  \chisqdiv{\Prob_{\treateffzero, \probxtweak}^{\otimes
      \numobs}}{\Prob_{\treateffzero, \probxstar}^{\otimes \numobs}}
  \leq \Big(1 + \frac{1}{8 \numobs} \Big)^\numobs - 1 \; \leq \;
  \tfrac{3}{20}.
\end{align*}
Combining with our earlier statement~\eqref{EqnPinsker} of Pinsker's
inequality completes the proof of the upper
bound~\eqref{EqnPartATVUpper}.

%%%%%%%%%%%%%%%%%%%%%%%%%%%%%%%%%%%%%%%%%%%%%%%%%%%%%%%%%%%%%%%%%%%%%%%%%%%%%%%%%

\subsubsection{Proof of equation~\eqref{eq:minimax-lower-bound-local-sigma}}

The proof is also based on Le Cam's two-point method. Complementary to
equation~\eqref{eq:minimax-lower-bound-local-v}, we take the source
distribution $\probxstar$ to be fixed, and perturb the outcome
function $\treateff$. Given a pair $\plaintreateff_{(s)},
\plaintreateff_{(-s)}$ of outcome functions in the local neighborhood
$\neighborhood_\delta^{val}$, Le Cam's two-point lemma implies
\begin{align}
\label{EqnLeCamTwoPointVal}
\minimaxRisk \geq \frac{1}{4} \Big\{1 - \totalvariation
\Big(\Prob_{\plaintreateff_{(s)}, \probxstar}^{\otimes
  \numobs},\Prob_{\plaintreateff_{(-s)}, \probxstar}^{\otimes \numobs}
\Big) \Big\} \cdot \big \{ \avgtreat(\probxstar, \plaintreateff_{(s)})
- \avgtreat(\probxstar, \plaintreateff_{(-s)}) \big \}^2,
\end{align}
With this set-up, our proof is based on constructing a pair
$(\plaintreateff_{(s)}, \plaintreateff_{(-s)})$ of outcome functions
within the neighborhood $\neighborhood_\delta^{val} (\treateff)$ such
that
\begin{subequations}
  \begin{align} 
\label{eq:gap-bound-in-minimax-sigma-proof}   
\avgtreat(\probxstar, \plaintreateff_{(s)}) - \avgtreat(\probxstar,
\plaintreateff_{(-s)}) & \geq \frac{1}{2 \sqrt{\numobs}}
\weightednorm{\sigma}, \quad \mbox{and} \\
\label{eq:tv-bound-in-minimax-sigma-proof}
\totalvariation \Big(\Prob_{\plaintreateff_{(s)}, \probxstar}^{\otimes
  \numobs},\Prob_{\plaintreateff_{(-s)}, \probxstar}^{\otimes \numobs}
\Big) & \leq \frac{1}{3}.
\end{align}
\end{subequations}

%%%%%%%%%%%%%%%%%%%%%%%%%%%%%%%%%%%%%%%%%%%%%%%%%%%%%%%%%%%%%%%%%

\paragraph{Construction of problem instances:}
Consider the noisy Gaussian observation model
\begin{align}
\outcome_i ~\mid~ \State_i, \Action_i \sim \Normal \Big(
\treateff(\State_i, \Action_i), \treatsig^2(\State_i, \Action_i) \Big)
\qquad \mbox{for $i = 1, 2, \ldots, \numobs$.}
\end{align}
We construct a pair of problem instances as follows: for any $\tweak >
0$, define the functions
\begin{align*}
\treateff_{(\tweak)}(\state, \action) = \treateffzero(\state, \action)
+ \tweak \frac{\weightfunc(\state, \action)}{\propscore(\state,
  \action)} \sigma^2(\state, \action), \quad \mbox{and} \quad
\treateff_{(- \tweak)}(\state, \action) = \treateffzero(\state,
\action) - \tweak \frac{\weightfunc(\state,
  \action)}{\propscore(\state, \action)} \sigma^2(\state, \action)
\end{align*}
\mbox{for any $(\state, \action) \in \Xspace
  \times \actionspace$.}

Throughout this section, we make the choice $\tweak \mydefn
\frac{1}{4\weightednorm{\sigma} \sqrt{\numobs}}$. Under such choice,
the compatibility condition~\eqref{eq:size-of-neighborhood} ensures
that
\begin{align*}
|\plaintreateff_{(z \tweak)}(\state, \action) - \treateffzero(\state,
\action)| = \tweak \frac{\weightfunc(\state,
  \action)}{\propscore(\state, \action)} \sigma^2(\state, \action)
\leq \delta(\state, \action) \quad \mbox{for any $(\state, \action)
  \in \Xspace \times \actionspace$ and $z \in \{-1, 1\}$.}
\end{align*}
This ensures that both $\plaintreateff_{(\tweak)}$ and
$\plaintreateff_{(- \tweak)}$ belong to the neighborhood
$\neighborhood_{ \delta}^{val}(\treateffzero)$. It remains to prove
the two bounds required for Le Cam's two-point arguments.

\paragraph{Proof of equation~\eqref{eq:gap-bound-in-minimax-sigma-proof}:}

For the target linear functional under our construction, we note that
\begin{align*}
  \avgtreat( \probxstar, \plaintreateff_{(\tweak)}) -
  \avgtreat( \probxstar, \plaintreateff_{(- \tweak)}) = 2 \tweak
  \Exs_{\probxstar} \Big[ \actinprod{\frac{\weightfunc(\State,
        \cdot)}{\propscore(\State, \cdot)} \sigma^2(\State,
      \cdot)}{\weightfunc(\State, \cdot)} \Big] = 2 \tweak
  \weightednorm{\sigma}^2 = \frac{1}{2 \sqrt{\numobs}}
  \weightednorm{\sigma},
\end{align*}
which establishes the
bound~\eqref{eq:gap-bound-in-minimax-sigma-proof}.

\paragraph{Proof of the bound~\eqref{eq:tv-bound-in-minimax-sigma-proof}:}

It order to bound the total variation distance, we study the KL
divergence between the product distributions
$\Prob_{\plaintreateff_{(z \tweak)}, \probxstar}^{\otimes\numobs}$ for
$z \in \{-1, 1\}$. Indeed, we have
\begin{align}
\kull{\Prob_{\plaintreateff_{(\tweak)}, \probxstar}^{\otimes\numobs}
}{\Prob_{\plaintreateff_{(- \tweak)}, \probxstar}^{\otimes\numobs}} &
\overset{(i)}{=} \numobs \kull{\Prob_{\plaintreateff_{(\tweak)},
    \probxstar} }{\Prob_{\plaintreateff_{(- \tweak)}, \probxstar}}
\nonumber \\
\label{eq:tensorization-of-kl-in-sigma-lower-bound-proof}
& \overset{(ii)}{\leq} \numobs \Exs \Big[ \kull{\law(\outcome \mid
    \State, \Action) |_{\plaintreateff_{(\tweak)}}}{\law(\outcome \mid
    \State, \Action) |_{\plaintreateff_{(- \tweak)}} } \Big],
\end{align}
where in step (i), we use the tensorization property of KL divergence,
and in step (ii), we use convexity of KL divergence. The expectation
is taken with respect to $X \sim \probxstar$ and $\Action \sim
\propscore(\State, \cdot)$.

Noting that the conditional law $\law(\outcome \mid \State, \Action)
|_{\plaintreateff_{(z \tweak)}}$ is Gaussian under both problem
instances, we have
\begin{align*}
\kull{\law(\outcome \mid \state, \action)
  |_{\plaintreateff_{(\tweak)}}}{\law(\outcome \mid \state, \action)
  |_{\plaintreateff_{(- \tweak)}} } = \frac{4 \tweak^2
  \weightfunc^2(\state, \action)}{\propscore^2(\state, \action)}
\sigma^2(\state, \action).
\end{align*}
Substituting into
equation~\eqref{eq:tensorization-of-kl-in-sigma-lower-bound-proof}, we
find that $\kull{\Prob_{\plaintreateff_{(\tweak)},
    \probxstar}^{\otimes\numobs} }{\Prob_{\plaintreateff_{(- \tweak)},
    \probxstar}^{\otimes\numobs}} \leq 4 \numobs \tweak^2
\weightednorm{\sigma}^2$.  For a sample size $\numobs \geq 16$, with
the choice of the perturbation parameter $\tweak = \frac{1}{4
  \sqrt{\numobs} \weightednorm{\sigma}}$, an application of Pinsker's
inequality leads to the bound
\begin{align}
\label{eq:KL-bound-for-sigma-in-local-lower-bound-proof}  
\totalvariation \Big(\Prob_{\plaintreateff_{(\tweak)},
  \probxstar}^{\otimes\numobs} , \Prob_{\plaintreateff_{(- \tweak)},
  \probxstar}^{\otimes\numobs} \Big) \leq\sqrt{\frac{1}{2}
  \kull{\Prob_{\plaintreateff^{\tweak}, \probxstar}^{\otimes\numobs}
  }{\Prob_{\plaintreateff^{(- \tweak)}, \probxstar}^{\otimes\numobs}}}
\leq \frac{1}{2 \sqrt{2}},
\end{align}
which completes the proof of
equation~\eqref{eq:tv-bound-in-minimax-sigma-proof}.

%%%%%%%%%%%%%%%%%%%%%%%%%%%%%%%%%%%%%%%%%%%%%%%%%%%%%%%%%%%%%%%%%%%%

\subsubsection{Proof of equation~\eqref{eq:minimax-lower-bound-local-delta}}

The proof is based on Le Cam's mixture-vs-mixture method (cf. Lemma
15.9, ~\cite{wainwright2019high}). We construct a pair $(\Qprob_1,
\Qprob_{-1})$ of probability distributions supported on the
neighborhood $\Nval{\delta}(\treateffzero)$; these are used to define
two mixture distributions with the following properties:
\begin{itemize}
\item The mixture distributions have TV distance bounded as
\begin{align}
\label{eq:mixture-vs-mixture-tv-bound}    
\totalvariation \Bigg( \int \Prob_{\plainmu, \probxstar}^{\otimes
  \numobs} d \Qprob_1^* (\plainmu), \int \Prob_{\plainmu,
  \probxstar}^{\otimes \numobs} d \Qprob_{-1}^* (\plainmu) \Bigg) \leq
\frac{1}{4}.
\end{align}
See Lemma~\ref{lemma:tv-bound-in-local-minimax-delta-proof} for
details.
\item There is a large gap in the target linear functional when
  evaluated at functions in the support of $\Qprob_1^*$ and
  $\Qprob_{-1}^*$. See
  Lemma~\ref{lemma:gap-bound-in-local-minimax-delta-proof} for
  details.
\end{itemize}

For any binary function $\zeta : \Xspace \times \actionspace
\rightarrow \{-1, 1\}$, we define the perturbed outcome function
\begin{align*}
\plainmu_\zeta(\state, \action) \mydefn \treateffzero(\state, \action)
+ \zeta(\state, \action) \cdot \delta(\state, \action) \quad \mbox{for
  all $(\state, \action) \in \Xspace \times \actionspace$.}
\end{align*}
By construction, we have $\plainmu_\zeta \in \Nval{\delta}(\treateff)$
for any binary function $\zeta$.  Now consider the function
\begin{align*}
\rho(\state, \action) \mydefn \begin{dcases} \frac{\weightfunc(\state,
    \action) \delta(\state, \action)}{\weightednorm{\delta}
    \propscore(\state, \action)} & \frac{|\weightfunc(\state,
    \action)| \delta(\state, \action)}{\propscore(\state, \action)}
  \leq 2 \ctwofour \weightednorm{\delta} \\
  \mathrm{sgn} \big(g(\state, \action) \big),& \mbox{otherwise}.
    \end{dcases}
\end{align*}
It can be seen that $\Exs[\rho^2(\State, \Action)] \leq 1$ where the
expectation is taken over a pair $\State \sim \probxstar$ and $\Action
\sim \propscore(\State, \cdot)$.

For a scalar $\tweak \in \big( 0, \tfrac{1}{2\ctwofour} \big]$ and a
  sign variable $z \in \{-1, 1\}$, we define the probability
  distribution
\begin{align}
\Qprob_z^{\tweak} \mydefn \law(\plainmu_\zeta), \quad \mbox{where}
\quad \zeta \sim \prod_{x \in \Xspace, \action \in \actionspace}
\bernDistr \Big( \frac{1 + z \tweak \rho(\state, \action)}{2} \Big).
\end{align}
Having constructed the mixture distributions, we are ready to prove
the lower bound~\eqref{eq:minimax-lower-bound-local-delta}. The proof
relies on the following two lemmas on the properties of the mixture
distributions:

\begin{lemma}
\label{lemma:tv-bound-in-local-minimax-delta-proof}
The total variation distance between mixture-of-product distributions
is upper bounded as
\begin{align}
\label{eq:tv-bound-in-local-minimax-delta-proof}  
\totalvariation \Bigg( \int \Prob_{\plaintreateff,
  \probxstar}^{\otimes \numobs} d \Qprob_1^{\tweak}(\plaintreateff),
\int \Prob_{\plaintreateff, \probxstar}^{\otimes \numobs} d
\Qprob_{-1}^{\tweak}(\plaintreateff) \Bigg) \leq 2 \tweak
\sqrt{\numobs} + 4 \cdot e^{- \numobs / 4}.
\end{align}
\end{lemma}

\begin{lemma}
\label{lemma:gap-bound-in-local-minimax-delta-proof}
Given a state space with cardinality lower bounded as $|\Xspace| \geq
128 c_{\max} / \tweak^2$, we have
\begin{subequations}
\label{eqs:truncation-in-mixture-vs-mixture-proof}
\begin{align}
\Prob_{\plaintreateff \sim \Qprob_1^{\tweak} } \Big\{
\tau(\probxstar, \plaintreateff) \geq \tau(\probxstar, \treateffzero)
+ \frac{\tweak}{8} \weightednorm{\delta} \Big\} & \geq 1 - 2 \cdot
e^{-4}, \quad \mbox{and} \\
\Prob_{\plaintreateff \sim \Qprob_{-1}^{\tweak} } \Big\{
\tau(\probxstar, \plaintreateff) \leq \tau(\probxstar, \treateffzero)
- \frac{\tweak}{8} \weightednorm{\delta} \Big\} & \geq 1 - 2 \cdot
e^{-4}.
\end{align}
\end{subequations}
\end{lemma}
\noindent We prove these lemmas at the end of this section.

Taking these two lemmas as given, we now proceed with the proof of
equation~\eqref{eq:minimax-lower-bound-local-delta}. Based on
Lemma~\ref{lemma:gap-bound-in-local-minimax-delta-proof}, we define
two sets of functions as follows:
\begin{align*}
\Event_1 & \mydefn \Big\{ \plainmu_\zeta \, \mid \, \zeta \in \{-1,
1\}^{\Xspace \times \actionspace}, ~\tau(\probxstar, \plaintreateff_\zeta,
) \geq \tau( \probxstar, \treateffzero) + \frac{\tweak}{8}
\weightednorm{\delta} \Big\}, \quad \mbox{and} \\
\Event_{-1} & \mydefn \Big\{ \plainmu_\zeta \: \mid \: \zeta \in \{-1,
1\}^{\Xspace \times \actionspace},~ \tau(\probxstar, \plaintreateff_\zeta) \leq \tau(\probxstar, \treateffzero) - \frac{\tweak}{8}
\weightednorm{\delta} \Big\}.
\end{align*}
When the sample size requirement in
equation~\eqref{eq:minimax-lower-bound-local-delta} is satisfied,
Lemma~\ref{lemma:gap-bound-in-local-minimax-delta-proof} implies that
$\Qprob_z^{\tweak}(\Event_z) \geq 1 - e^{-4}$ for $z \in \{-1,
1\}$. We set $\tweak = \frac{1}{16 \sqrt{\numobs}}$, and define
\begin{align}
\Qprob_z^* \mydefn \Qprob_1^{\tweak} \big| \Event_z, \quad \mbox{for
  $z \in \{-1, 1 \}$.}
\end{align}
By construction, the probability distributions $\Qprob_1^*$ and
$\Qprob_{-1}^*$ have disjoint support, and for any pair
$\plaintreateff \in \support(\Qprob_1^*) \subseteq \Event_1$ and
$\plaintreateff' \in \support(\Qprob_{-1}^*) \subseteq \Event_{-1}$,
we have:
\begin{align}
 \label{eq:mixture-vs-mixture-gap-bound}  
\avgtreat \big(  \probxstar , \plaintreateff\big) \geq \avgtreat \big(
\probxstar, \treateffzero \big) + \frac{\weightednorm{\delta}}{128
  \sqrt{\numobs}}, \quad \mbox{and} \quad \avgtreat \big(
 \probxstar , \plaintreateff'\big) \leq \avgtreat \big(\probxstar, \treateffzero \big) - \frac{\weightednorm{\delta}}{128 \sqrt{\numobs}}.
\end{align}
Furthermore, combining the conclusions in
Lemma~\ref{lemma:tv-bound-in-local-minimax-delta-proof} and
Lemma~\ref{lemma:gap-bound-in-local-minimax-delta-proof} using
Lemma~\ref{lemma:total-variation-with-truncation}, we obtain the total
variation distance upper bound:
\begin{align*}
& \totalvariation \Bigg( \int \Prob_{\plaintreateff,
    \probxstar}^{\otimes \numobs} d \Qprob_1^* (\plaintreateff), \int
  \Prob_{\plaintreateff, \probxstar}^{\otimes \numobs} d \Qprob_{-1}^*
  (\plaintreateff) \Bigg) \\
& \leq \frac{1}{1 - 2 \cdot e^{-4}} \totalvariation \Bigg( \int
  \Prob_{\plaintreateff, \probxstar}^{\otimes \numobs} d
  \Qprob_1^{\tweak}(\plaintreateff), \int \Prob_{\plaintreateff,
    \probxstar}^{\otimes \numobs} \Qprob_{-1}^{\tweak}(\plaintreateff)
  \Bigg) + 4 \cdot e^{-4}\\
& \leq \frac{1/8 + 4 \cdot e^{- \numobs/ 4}}{1 - 2 \cdot e^{-4}} + 4
  \cdot e^{-4} \leq \frac{1}{4},
\end{align*}
which completes the proof of
equation~\eqref{eq:mixture-vs-mixture-tv-bound}.

Combining equation~\eqref{eq:mixture-vs-mixture-gap-bound}
and~\eqref{eq:mixture-vs-mixture-tv-bound}, we can invoke Le Cam's
mixture-vs-mixture lemma, and conclude that
\begin{align*}
  \minimaxRisk &\geq \frac{1}{4} \Big\{1 -\totalvariation \Big( \int
  \Prob_{\plaintreateff, \probxstar}^{\otimes \numobs} d \Qprob_1^*
  (\plaintreateff), \int \Prob_{\plaintreateff, \probxstar}^{\otimes
    \numobs} d \Qprob_{-1}^* (\plaintreateff) \Big)\Big\} \cdot
  \inf_{\substack{\plaintreateff \in \support
      (\Qprob_1)\\ \plaintreateff' \in \support(\Qprob_{-1})}} \big[
    \avgtreat(\probxstar, \plaintreateff) - \avgtreat(    \probxstar, \plaintreateff') \big]_+^2 \\
& \geq \frac{c \weightednorm{\delta}^2}{\numobs},
\end{align*}
for a universal constant $c > 0$. This completes the proof of
equation~\eqref{eq:minimax-lower-bound-local-delta}.

\paragraph{Proof of Lemma~\ref{lemma:tv-bound-in-local-minimax-delta-proof}:}

Our proof exploits a Poissonization device, which makes the number of
observations random, and thereby simplifies calculations. For $z \in
\{-1, 1\}$, denote the mixture-of-product distribution:
\begin{align*}
\Qprob_z^{(\tweak, \otimes \numobs)} \mydefn \int
\Prob_{\plaintreateff, \probxstar}^{\otimes \numobs} d
\Qprob_z^{\tweak}(\plaintreateff).
\end{align*}
We construct a pair $\big( \Qprob_1^{(\tweak, \mathrm{Poi})},
\Qprob_{-1}^{(\tweak, \poissonDistr)} \big)$ of mixture distributions
as follows: randomly draw the sample size $\nu \sim \poissonDistr(2
\numobs)$ independent of $\zeta$ and random sampling of data. For each
$z \in \{-1, 1\}$, we let $ \Qprob_z^{(\tweak, \poissonDistr)}$ be the
mixture distribution:
\begin{align*}
\Qprob_z^{(\tweak, \poissonDistr)} \mydefn \sum_{k = 0}^{+ \infty}
\Qprob_z^{(\tweak, \otimes k)} \cdot \Prob(\nu = k).
\end{align*}
By a known lower tail bound for a Poisson random variable
(c.f.~\cite{boucheron2013concentration}, \S 2.2), we have
\begin{align}
\Prob \big[ \underbrace{\nu \geq \numobs}_{=:
    \widetilde{\Event}_\numobs} \big] & \geq 1 - e^{- \numobs / 4},
\end{align}
We note that on the event $\widetilde{\Event}_\numobs$, the
probability law $ \Qprob_z^{(\tweak, \otimes \numobs)}$ is actually
the projection of the law $\Qprob_z^{(\tweak, \mathrm{Poi})} \big|
\widetilde{\Event}_\numobs $ on the first $\numobs$
observations. Consequently, we can
use~\Cref{lemma:total-variation-with-truncation} to bound the total
variation distance between the original mixture distributions using
that of the Poissonized models:
\begin{align}
\totalvariation \big( \Qprob_1^{(\tweak, \otimes \numobs)},
\Qprob_{-1}^{(\tweak, \otimes \numobs)}\big) &\leq \totalvariation
\Big( \Qprob_1^{(\tweak, \poissonDistr)} \big|
\widetilde{\Event}_\numobs , \Qprob_{-1}^{(\tweak, \poissonDistr)}
\big| \widetilde{\Event}_\numobs \Big) + 4 \Prob
\big(\widetilde{\Event}_\numobs^c \big) \nonumber \\
& \leq \frac{1}{\Prob \big(\widetilde{\Event}_\numobs
  \big)}\totalvariation \Big( \Qprob_1^{(\tweak, \poissonDistr)} ,
\Qprob_{-1}^{(\tweak, \poissonDistr)} \Big) + 4 \cdot e^{- \numobs /
  4} \nonumber \\
\label{eq:poissonization-error-bound}
& \leq 2 \totalvariation \Big( \Qprob_1^{(\tweak, \poissonDistr)} ,
\Qprob_{-1}^{(\tweak, \poissonDistr)} \Big) + 4 \cdot e^{- \numobs /
  4},
\end{align}
valid for any $\numobs \geq 4$.

It remains to bound the total variation distance between the
Poissonized mixture distributions.  We start by considering the
empirical count function
\begin{align*}
  M(\state, \action) \mydefn \sum_{i = 1}^\nu \bm{1} \Big\{\State_i =
  \state, ~\Action_i = \action \Big\} \qquad \mbox{for all $(\state,
    \action) \in \statespace \times \actionspace$,}
\end{align*}
Note that conditionally on the value of $\nu$, the vector $(M(\state,
\action))_{\state \in \Xspace, \action \in \actionspace}$ follows a
multinomial distribution. Since $\nu \sim \poissonDistr(2 n)$, we have
\begin{align*}
   \forall x \in \Xspace, ~ a \in \actionspace \quad M(\state,
   \action) \sim \poissonDistr \big(2 \numobs \probxstar(x)
   \propscore(\state, \action) \big), \quad \mbox{independent of each
     other}.
\end{align*}
For each $(\state, \action) \in \Xspace \times \actionspace$ and $z
\in \{-1, 1\}$, we consider a probability distribution
$\Qnew_z(x, \action)$ defined by the following sampling
procedure:
\begin{enumerate}[label=(\alph*)]
\item Sample $M(\state, \action) \sim \poissonDistr \big(2 \numobs
  \probxstar(x) \propscore(\state, \action) \big)$.
    \item Sample $\zeta(\state, \action) \sim \bernDistr \big( \frac{1
      + \tweak z \rho(\state, \action)}{2} \big)$.
    \item Generate a (possibly empty) set of $M(\state, \action)$
      independent observations from the conditional law of $Y$ given
      $X = \state$ and $\Action = \action$.
\end{enumerate}
By independence, for any $z \in \{-1, 1\}$, it is straightforward to
see that:
\begin{align*}
    \Qprob_z^{(\tweak, \poissonDistr)} = \prod_{(\state, \action)
      \in \statespace \times \actionspace} \Qnew_z(\state, \action),
\end{align*}
Pinsker's inequality, combined with the tensorization of the KL
divergence, guarantees that
\begin{align}
\label{eq:pinsker-and-tensorization-intv-bound-proof}  
\totalvariation \big(\Qprob_1^{(\tweak, \poissonDistr)},
\Qprob_{-1}^{(\tweak, \poissonDistr)} \big) \leq \sqrt{\frac{1}{2}
  \kull{\Qprob_1^{(\tweak, \poissonDistr) }}{
    \Qprob_{-1}^{(\tweak, \poissonDistr)} }} = \sqrt{\frac{1}{2}
  \sum_{x \in \Xspace, a \in \actionspace}
  \kull{\Qnew_{1}(\state,
    \action)}{\Qnew_{-1}(\state, \action)} }.
\end{align}
Note that the difference between the probability distributions
$\Qnew(\state, \action)$ and $\Qnew_{- 1}(\state,
\action)$ lies only in the parameter of the Bernoulli random variable
$\zeta(\state, \action)$, which is observed if and only if $M(\state,
\action) > 0$. By convexity of KL divergence, we have:
\begin{align*}
\kull{\Qnew_{1}(\state, \action)}{\Qnew_{-1}(\state,
  \action)} & \leq \Prob \Big( M(\state, \action) > 0 \Big) \cdot
\kull{\bernDistr \big( \frac{1 + \tweak \rho(\state, \action)}{2}
  \big)}{\bernDistr \big( \frac{1 - \tweak \rho(\state, \action)}{2}
  \big)} \\
& \leq 4 \big(1 - e^{-2 \numobs \probxstar(x) \propscore(\state,
  \action)} \big) \cdot \tweak^2 \rho^2(\state, \action)\\
& \leq 8 \numobs \probxstar(x) \propscore(\state, \action) \tweak^2
\rho^2(x) \\
  & \leq 8 \numobs \tweak^2 \probxstar(x) \frac{\weightfunc(\state,
  \action) \delta^2(\state, \action)}{\propscore(\state, \action)
  \weightednorm{\delta}^2}
\end{align*}
Substituting back to the decomposition
result~\eqref{eq:pinsker-and-tensorization-intv-bound-proof}, we
conclude that
\begin{align*}
\totalvariation \big(\Qprob_1^{(\tweak, \poissonDistr)},
\Qprob_{-1}^{(\tweak, \poissonDistr)} \big) \leq \sqrt{\frac{1}{2}
  \sum_{x \in \Xspace, a \in \actionspace} 8 \numobs \tweak^2
  \probxstar(x) \frac{\weightfunc^2(\state, \action) \delta^2(\state,
    \action)}{\propscore(\state, \action) \weightednorm{\delta}^2} }
\leq 2 \tweak \sqrt{\numobs}.
\end{align*}
Finally, combining with equation~\eqref{eq:poissonization-error-bound}
completes the proof.

%%%%%%%%%%%%%%%%%%%%%%%%%%%%%%%%%%%%%%%%%%%%%%%%%%%%%%%%%%%%%%%%%%%

\paragraph{Proof of Lemma~\ref{lemma:gap-bound-in-local-minimax-delta-proof}:} 

Under our construction, we can compute the expectation of the target
linear functional $\tau(\probInstance)$ under both distributions. In
particular, for $z = 1$, we have
\begin{align*}
 \Exs_{\plaintreateff \sim \Qprob_1^{\tweak} } \big[
   \tau(\probxstar, \plaintreateff) \big] & = \tau( \probxstar, \treateffzero) + \frac{ \tweak}{2} \cdot \Exs_{\probxstar} \Big[
   \int_\actionspace \delta(\State, \action) \weightfunc(\State,
   \action) \rho(\State, \action) d \actbasemsr(\action) \Big] \\
& \geq \tau(\probxstar, \treateffzero) + \frac{ \tweak}{2
   \weightednorm{\delta}} \cdot \Exs \Big[ \frac{\delta^2(\State,
     \Action) \weightfunc(\State, \Action)^2}{\propscore^2(\State,
     \Action)} \bm{1}\Big\{ \frac{|\weightfunc(\State, \Action)|
     \delta(\State, \Action)}{\propscore(\State, \Action)} \leq 2
   \ctwofour \weightednorm{\delta} \Big\} \Big],
\end{align*}
where the last expectation is taken with respect to $\State \sim
\probxstar$ and $\Action \sim \propscore(\State, \cdot)$.

Applying~\Cref{lemma:simple-moment-lower-bound-under-trunc} to the
random variable $\weightfunc(\State, \Action) \delta(\State, \Action)
/ \propscore(\State, \Action)$ yields
\begin{align*}
\Exs \Big[ \frac{\delta^2(\State, \Action) \weightfunc^2(\State,
    \Action)}{\propscore^2(\State, \Action)} \bm{1}\Big\{
  \frac{|\weightfunc(\State, \Action)| \delta(\State,
    \Action)}{\propscore(\State, \Action)} \leq 2 \ctwofour
  \weightednorm{\delta} \Big\} \Big] \geq \frac{1}{2} \Exs \Big[
  \frac{\delta^2(\State, \Action) \: \weightfunc^2(\State,
    \Action)}{\propscore^2(\State, \Action)} \Big].
\end{align*}

Consequently, we have the lower bound on the expected value under
$\Qprob_1^{\tweak}$
\begin{subequations}
\label{eqs:expectation-bounds-in-mixture-vs-mixture-proof}
\begin{align}
\Exs_{\plaintreateff \sim \Qprob_1^{\tweak} } \big[
  \tau(\probxstar, \plaintreateff) \big] \geq \tau(\probxstar, \treateffzero) + \frac{\tweak}{4} \weightednorm{\delta}.
\end{align}
Similarly, under the distribution $\Qprob_{-1}^{\tweak}$, we note
that:
\begin{align}
\Exs_{\plaintreateff \sim \Qprob_{-1}^{\tweak} } \big[
  \tau( \probxstar, \plaintreateff) \big] \leq \tau(\probxstar, \treateffzero) - \frac{\tweak}{4} \weightednorm{\delta}.
\end{align}
\end{subequations}

We now consider the concentration behavior of random function
$\plaintreateff \sim \Qprob_{z}^{\tweak}$ for each choice of $z \in
\{-1, 1\}$.  Since the random signs are independent at each
state-action pair \mbox{$(\state, \action) \in \Xspace
  \times \actionspace$,} we can apply Hoeffding's inequality: more
precisely, with with probability $1 - 2 e^{- 2t}$, we have
\begin{align*}
\abss{\tau(\probxstar, \plaintreateff) - \Exs_{\plaintreateff \sim
    \Qprob_{z}^{\tweak} } \big[ \tau( \probxstar, \plaintreateff)
    \big]} & \overset{(i)}{\leq} \sqrt{t \cdot \sum_{\state
    \in \Xspace, \action \in \actionspace} \probxstar^2(x)
  \weightfunc^2(\state, \action) \delta^2(\state, \action)} \\
& \leq \sqrt{ \frac{t c_{\max}}{|\Xspace|} \cdot \sum_{\state
    \in \Xspace, a \in \actionspace} \probxstar(x)
  \weightfunc^2(\state, \action) \delta^2(\state, \action)} \leq
\sqrt{\frac{t \cdot c_{\max}}{|\Xspace|}} \weightednorm{\delta},
\end{align*}
where in step (i), we use the compatibility condition
$\probxstar(\state) \leq \frac{c_{\max}}{|\Xspace|}$ for any $\state
\in \Xspace$.

Given a state space with cardinality lower bounded as \mbox{$|\Xspace|
  \geq 128 c_{\max} / \tweak^2$,} we can combine the concentration
bound with the expectation
bounds~\eqref{eqs:expectation-bounds-in-mixture-vs-mixture-proof} so
as to obtain
\begin{align*}
\Prob_{\plaintreateff \sim \Qprob_1^{\tweak} } \Big\{
\tau(\probxstar, \plaintreateff) \geq \tau(\probxstar, \treateffzero)
+ \tweak \sqrt{ \tfrac{t}{128}} \weightednorm{\delta} \Big\} &\geq 1 -
2 \cdot e^{- 2t}, \quad \mbox{and} \\
\Prob_{\plaintreateff \sim \Qprob_{-1}^{\tweak} } \Big\{
\tau(\probxstar, \plaintreateff) \leq \tau(\probxstar, \treateffzero)
- \tweak \sqrt{ \tfrac{t}{128}} \weightednorm{\delta} \Big\} & \geq 1
- 2 \cdot e^{- 2t}.
\end{align*}
Setting $t = 2$ completes the proof
of~\Cref{lemma:gap-bound-in-local-minimax-delta-proof}.

%%%%%%%%%%%%%%%%%%%%%%%%%%%%%%%%%%%%%%%%%%%%%%%%%%%%%%%%%%%%%%%%%

\subsection{Proof of~\Cref{thm:worst-case-shattering-dim}}
\label{subsec:proof-thm-minimax-shattering-dim}

Let the input distribution $\probxstar$ be the uniform distribution over the sequence $\{x_j\}_{j = 1}^\fatdim$. It suffices to show that
\begin{align}
    \inf_{\tauhat_\numobs}
\sup_{ \plainmu \in \funcClass} \Exs \big[ \abss{\tauhat_\numobs -
    \avgtreat (\probxstar, \plainmu)}^2 \big] \geq \frac{c}{\numobs} \; \Big \{
  \frac{1}{\fatdim} \sum_{j = 1}^\fatdim \sum_{\action
    \in \actionspace} \frac{\weightfunc^2(\state_j,
    \action)}{\propscore(\state_j, \action)} \delta_\action^2 \Big \}.\label{eq:worst-case-shattering-dim-with-unif-data}
\end{align}
Recall that we are given a sequence $\{x_j\}_{j = 1}^\fatdim$ such
that for each $\action \in \actionspace$, the function class
$\funcClass_\action$ shatters it at scale $\delta_\action$.  Let
$\{t_{j, \action}\}_{j = 1}^\fatdim$ be the sequence of function
values in the fat-shattering
definition~\eqref{eq:defn-fat-shattering}. Note that
since the class $\funcClass$ is convex, we have
  \begin{align*}
\bigotimes_{j = 1}^\fatdim \bigotimes_{\action \in \actionspace}
          [t_{j, \action} - \delta_\action, t_{j, \action} +
            \delta_\action] \subseteq \bigotimes_{\action
            \in \actionspace} \Big\{(f_\action(x_j))_{j \in [\fatdim]}
          \; \mid \; f_\action \in \funcClass_\action \Big\}.
\end{align*}
Note that this distribution satisfies the compatibility condition with
$c_{\max} = 1$ and the hyper-contractivity condition with a constant $\ctwofour = \ltwolfour{\weightfunc (\State, \Action) \delta_\Action / \propscore (\State, \Action)}$. Invoking equation~\eqref{eq:minimax-lower-bound-local-delta} over
the local neighborhood $\Nval{\delta}(t)$ yields the claimed
bound~\eqref{eq:worst-case-shattering-dim-with-unif-data}.

%%%%%%%%%%%%%%%%%%%%%%%%%%%%%%%%%%%%%%%%%%%%%%%%%%%%%%%%%%%%%%%%%%%

\section{Discussion}
\label{SecDiscussion}

We have studied the problem of evaluating linear functionals of the
outcome function (or reward function) based on observational data.  In
the bandit literature, this problem corresponds to off-policy
evaluation for contextual bandits.  As we have discussed, the
classical notion of semi-parametric efficiency characterizes the
optimal asymptotic distribution, and the finite-sample analysis
undertaken in this paper enriches this perspective.  First, our
analysis uncovered the importance of a particular weighted
$\Ltwospace$-norm for estimating the outcome function $\treateff$.
More precisely, optimal estimation of the scalar $\taustar$ is
equivalent to optimal estimation of the outcome function $\treateff$
under such norm, in the sense of minimax risk over a local
neighborhood. Furthermore, when the outcome function is known to lie
within some function class $\funcClass$, we showed that a sample size
scaling with the complexity of $\funcClass$ is necessary and
sufficient to achieve such bounds non-asymptotically.

Our result lies at the intersection of decision-making problems and
the classical semi-parametric theories, which motivates several
promising directions of future research on both threads:
\bcar
\item Our analysis reduces the problem of obtaining finite-sample
  optimal estimates for linear functionals to the nonparametric
  problem of estimating the outcome function under a weighted
  norm. Although the re-weighted least-square
  estimator~\eqref{eq:least-square-estimator-in-stage-1} converges to
  the best approximation of the treatment effect function in the
  class, it is not clear whether it always achieves the optimal
  trade-off between the approximation and estimation errors. How to
  optimally estimate the nonparametric component under weighted norm
  (so as to optimally estimate the scalar $\taustar$ in finite sample)
  for a variety of function classes is an important direction of
  future research, especially with weight functions.
\item The analysis of the current paper was limited to
  $\mathrm{i.i.d.}$ data, but similar issues arise with richer models
  of data collection.  There are recent lines of research on how to
  estimate linear functionals with adaptively collected data
  (e.g. when the data are generated from an exploratory bandit
  algorithm~\cite{zhan2021off,shin2020conditional,khamaru2021near}),
  or with an underlying Markov chain structure (e.g. in off-policy
  evaluation problems for reinforcement
  learning~\cite{jiang2016doubly,yin2020asymptotically,kallus2022efficiently,
    ZanWaiBru_neurips21}).  Many results in this literature build upon
  the asymptotic theory of semi-parametric efficiency, so that it is
  natural to understand whether extensions of our techniques could be
  used to obtain finite-sample optimal procedures in these settings.
\item The finite-sample lens used in this paper reveals phenomena in
  semi-parametric estimation that are washed away in the asymptotic
  limit.  This paper has focused on a specific class of
  semi-parametric problems, but more broadly, we view it as
  interesting to see whether such phenomena exist for other models in
  semi-parametric estimation. In particular, if a high-complexity
  object---such as a regression or density function---needs to be
  estimated in order to optimally estimate a low-complexity
  object---such as a scalar---it is important to characterize the
  minimal sample size requirements, and the choice of nonparametric
  procedures for the nuisance component that are finite-sample
  optimal.
  \ecar

\subsection*{Acknowledgements}
The authors thank Peng Ding and Fangzhou Su for helpful discussion.
This work was partially supported by Office of Naval Research Grant
ONR-N00014-21-1-2842, NSF-CCF grant 1955450, and NSF-DMS grant 2015454
to MJW, NSF-IIS grant 1909365 and NSF grant DMS-2023505 to MJW and PLB,
and ONR MURI award N000142112431 to PLB.

\medskip
%%%%%%%%%%%%%%%%%%%%%%%%%%%%%%%%%%%%%%%%%%%%%%%%%%%%%%%%%%%%%%%%%%%%

\bibliographystyle{alpha}

{\small{
\begin{singlespace}
\bibliography{references}
\end{singlespace}
}}

\appendix

%%%%%%%%%%%%%%%%%%%%%%%%%%%%%%%%%%%%%%%%%%%%%%%%%%%%%%%%%%%%%%%%%%%

\section{Proofs of auxiliary results in the upper bounds}

In this section, we state and prove some auxiliary results used in the
proofs of our non-asymptotic upper bounds.

\subsection{Some properties of the estimator $\tauhatgen{\numobs}{f}$}
\label{app:sec:proof-basic-exp-and-var-of-general-form}

In this appendix, we collect some properties of the estimator
$\tauhatgen{\numobs}{f}$ defined in
equation~\eqref{eq:general-form-of-unbiased-estimator}.

\begin{proposition}
  \label{prop:expectation-and-var-of-general-form-estimator}
Given any deterministic function $f \in \Ltwospace(\probx \times
\propscore)$ for any $\action \in \actionspace$, we have
$\Exs[\tauhatgen{\numobs}{f}] = \avgtreat_\weightfunc
(\probInstance)$. Furthermore, if $\actinprod{f(\state,
  \cdot)}{\propscore(\state, \cdot)} = 0$ for any $\state
\in \Xspace$, we have
\begin{multline}
\label{eq:exact-var-general-form-estimator}  
\numobs \cdot \Exs \Big[ \abss{\tauhatgen{\numobs}{f} -
    \avgtreat_\weightfunc(\probInstance)}^2 \Big] = \var_\probx \Big(
\actinprod{\weightfunc(\State, \cdot)}{\treateff(\State, \cdot)} \Big)
+ \int_{\actionspace} \Exs \Big[ \frac{\treatsig^2(\State, \action)
    \weightfunc^2(\State, \action)}{\propscore(\State, \action)} \Big]
d \basemsr(\action) \\
+ \int_\actionspace \Exs \Big[ \propscore(\State, \action) \abss{
    f(\State, \action) - \tfrac{\weightfunc(\State, \action)
      \treateff(\State, \action)}{\propscore(\State, \action)} +
    \actinprod{\weightfunc(\State, \cdot)}{\treateff(\State, \cdot)}
  }^2 \Big] d \basemsr(\action).
    \end{multline}
\end{proposition}
This decomposition immediately implies the claims given in the text.
The only portion of the MSE
decomposition~\eqref{eq:exact-var-general-form-estimator} that depends
on $f$ is the third term, and by inspection, this third term is equal
to zero if and only if
\begin{align*}
f(\state, \action) = \tfrac{\weightfunc(\state, \action)
  \treateff(\state, \action)}{\propscore(\state, \action)} -
\actinprod{\weightfunc(\state, \cdot)}{\treateff(\state, \cdot)}
\qquad \mbox{for all $(\state, \action) \in \Xspace
  \times \ActionSpace$.}
\end{align*}
    
\begin{proof}
Since the action $\Action_i$ follows the probability distribution
$\propscore(\State_i, \cdot)$ conditionally on $\State_i$, we have
$\Exs \big[ f(\State_i, \Action_i) \mid \State_i \big] =
\actinprod{\propscore(\State_i, \cdot)}{f(\State_i, \cdot)}$, and the
estimator $\tauhatgen{\numobs}{f}$ is always unbiased. Since the
function $f$ is square-integrable with respect to the measure $\probx
\times \propscore$, the second moment can be decomposed as follows:
\begin{align*}
& \Exs \Big[ \abss{ \frac{\weightfunc(\State_i,
        \Action_i)}{\propscore(\State_i, \Action_i)} \outcome_i -
      f(\State_i, \Action_i) + \int_\actionspace \propscore(\State_i,
      \action) f(\State_i, \action) d \actbasemsr(\action)}^2 \Big] \\
& = \int_\actionspace \Exs \Big[ \propscore(\State_i, \action) \cdot
    \abss{ \frac{\weightfunc(\State_i, \action)}{\propscore(\State_i,
        \action)} \outcome_i - f(\State_i, \action) }^2 \Big] d
  \actbasemsr(\action) \\
  & = \int_\actionspace \Exs \Big[ \frac{\treatsig^2(\State, \action)
      \weightfunc^2(\State, \action)}{\propscore(\State, \action)}
    \Big] d \actbasemsr(\action) + \int_\actionspace \Exs \Big[
    \propscore(\State, \action) \cdot \abss{ \frac{\weightfunc(\State,
        \action) \treateff(\State, \action)}{\propscore(\State,
        \action)} - f(\State, \action) }^2 \Big] d
  \actbasemsr(\action)
\end{align*}
Conditionally on the value of $\State$, we have the bias-variance
decomposition
\begin{multline*}
\int_{\actionspace} \propscore(\State, \action) \cdot \abss{
  \frac{\weightfunc(\State, \action) \treateff(\State,
    \action)}{\propscore(\State, \action)} - f(\State, \action) }^2
d\actbasemsr(\action) \\
= \actinprod{\weightfunc(\State, \cdot)}{\treateff(\State, \cdot)}^2 +
\int_\actionspace \propscore(\State, \action) \cdot \abss{ f(\State,
  \action) - \frac{\weightfunc(\State, \action) \treateff(\State,
    \action)}{\propscore(\State, \action)} +
  \actinprod{\weightfunc(\State, \cdot)}{\treateff(\State, \cdot)} }^2
d \actbasemsr(\action).
\end{multline*}
Finally, we note that
\begin{align*}
 \Exs \Big[ \actinprod{\weightfunc(\State, \cdot)}{\treateff(\State,
     \cdot)}^2 \Big] - \avgtreat^2 (\probInstance) =
 \Big( \Exs \Big[ \actinprod{\weightfunc(\State,
     \cdot)}{\treateff(\State, \cdot)} \Big] \Big)^2 = \var_\probx
 \Big( \actinprod{\weightfunc(\State, \cdot)}{\treateff(\State,
   \cdot)} \Big).
\end{align*}
Putting together the pieces completes the proof.
\end{proof}
%%%%%%%%%%%%%%%%%%%%%%%%%%%%%%%%%%%%%%%%%%%%%%%%%%%%%%%%%%%%%%%%%%%%%%%%%%%%%%%%%%%%%%%%%%%%%%

\subsection{Existence of critical radii}
\label{app:subsec-proof-existence-critical-radii}

In this section, we establish the existence of critical radii
$\radone_\maux(\plainmu)$ and $\radtwo_\maux(\plainmu)$ defined in
equations~\eqref{eq:defn-critical-radius-squared}
and~\eqref{eq:defn-critical-radius-plain}, respectively.
\begin{proposition}
\label{prop:existence-critical-radii}
Suppose that the compatibility condition~\ref{assume:convset} holds,
and that the Rademacher complexities
$\SquaredRade_\maux\big((\funcClass - \plainmu) \cap \ball_\omega(r_0)
\big)$ and $\radeComplexity_\maux \big((\funcClass - \plainmu) \cap
\ball_\omega(r_0) \big)$ are finite for some $r_0 > 0$.  Then:
\begin{itemize}
\item[(a)] There exists a unique scalar $\radone_\maux =
  \radone_\maux(\plainmu) > 0$ such that
  inequality~\eqref{eq:defn-critical-radius-squared} holds for any
  $\radone \geq \radone_\maux$, with equality when $\radone =
  \radone_\maux$, and is false when $\radone \in [0, \radone_\maux)$.
\item[(b)] There exists a scalar $\radtwo_\maux =
  \radtwo_\maux(\plainmu) > 0$ such that
  inequality~\eqref{eq:defn-critical-radius-plain} holds for any
  $\radtwo \geq \radtwo_\maux$.
\end{itemize}
\end{proposition}
\begin{proof}
Denote the shifted function class $\funcClass^* \mydefn \funcClass -
\plainmu$. Since the class $\funcClass$ is convex by assumption, for
positive scalars $r_1 < r_2$ and any function $f \in \funcClass^* \cap
\ball_\omega(r_2)$, we have $\frac{r_1}{r_2} f \in \funcClass^* \cap
\ball_\omega(r_1)$.
\begin{align*}
\frac{1}{r_2} \radeComplexity_\maux(\funcClass^* \cap
\ball_\omega(r_2)) \leq \frac{1}{r_2} \radeComplexity_\maux \Big(
\frac{r_2}{r_1} \cdot \big( \funcClass^* \cap \ball_\omega(r_1) \big)
\Big) = \frac{1}{r_1} \radeComplexity_\maux(\funcClass^* \cap
\ball_\omega(r_1)).
\end{align*}
So the function $r \mapsto r^{-1} \radeComplexity_\maux(\funcClass^*
\cap \ball_\omega(r))$ is non-increasing in $r$. A similar argument
ensures that the function $\radone \mapsto \radone^{-1}
\SquaredRade(\funcClass^* \cap \ball_\omega(\radone))$ is also
non-increasing in $\radone$.

Since the function class $\funcClass$ is compact in
$\Ltwospace_\omega$, we have $D \mydefn \diameter_\omega(\funcClass
\cup \{\treateff\}) < + \infty$, and hence
\begin{align*}
\radeComplexity_\maux(\funcClass^*) = \radeComplexity_\maux \big(
\funcClass^* \cap \ball_\omega(D) \big) \leq \frac{D}{r_0}
\radeComplexity_\maux \big( \funcClass^* \cap \ball_\omega(r_0) \big)
< + \infty,
\end{align*}
which implies that $\radeComplexity_\maux \big(\funcClass^* \cap
\ball_\omega(r) \big) < + \infty$ for any $r > 0$.  Similarly, the
Rademacher complexity $\SquaredRade\big(\funcClass^* \cap
\ball_\omega(\radone) \big)$ is also finite.

For the inequality~\eqref{eq:defn-critical-radius-squared}, the left
hand side is a non-increasing function of $\radone$, while the right
hand side is strictly increasing and diverging to infinity as $\radone
\rightarrow + \infty$. Furthermore, the right-hand-side is equal to
zero at $\radone = 0$, while the left-hand side is always finite and
non-negative for $\radone > 0$. Consequently, a unique fixed point
$\radone_\maux \geq 0$ exists, and we have
\begin{align*}
\radone^{-1} \SquaredRade\big(\funcClass^* \cap \ball_\omega(\radone)
\big) \begin{cases} < \radone, & \mbox{for $\radone > \radone_\maux$,
    and} \\
 > \radone, & \mbox{for $\radone \in (0, \radone_\maux)$.}
 \end{cases}
\end{align*}

As for inequality~\eqref{eq:defn-critical-radius-plain}, the
left-hand-side is non-increasing, and we have
\begin{align*}
\lim_{\radtwo \rightarrow + \infty} \radtwo^{-1}
\radeComplexity(\funcClass^* \cap \ball_\omega(\radtwo)) \leq
\lim_{\radtwo \rightarrow + \infty} \radtwo^{-1}
\radeComplexity(\funcClass^* ) = 0.
\end{align*}
So there exists $\radtwo_\maux \geq 0$ such that
inequality~\eqref{eq:defn-critical-radius-plain} holds for any
$\radtwo \geq \radtwo_\maux$.

\end{proof}

\subsection{Proof of~\Cref{lemma:concentration-for-weighted-sqr-norm}}
\label{subsubsec:proof-lemma-concentration-for-weighted-sqr-norm}

We define the auxiliary function
\begin{align*}
\phi(t) \mydefn \begin{cases} 0 & t \leq 1,\\ t - 1 & 1 \leq t \leq
  2,\\ 1 & t > 2.
    \end{cases}
\end{align*}
First, observe that for any scalar $u > 0$, we have
\begin{align*}
\frac{1}{\maux} \sum_{i = 1}^\maux \frac{\weightfunc^2(\State_i,
  \Action_i)}{\propscore^2(\State_i, \Action_i)} h^2(\State_i,
\Action_i) & \geq \frac{1}{\maux} \sum_{i = 1}^\maux u^2 \cdot
\indicator \Big[ \frac{|\weightfunc(\State_i, \Action_i) h(\State_i,
    \Action_i)|}{\propscore(\State_i, \Action_i)} \geq u \Big]\\ &\geq
\frac{1}{\maux} \sum_{i = 1}^\maux u^2 \cdot \phi \Big( \frac{
  |\weightfunc(\State_i, \Action_i) h(\State_i, \Action_i)|
}{\propscore(\State_i, \Action_i) u} \Big) =: Z_\maux^{(\phi)}(h).
\end{align*}
Second, for any function $h \in \funcClassTemp$, we have
\begin{align*}
   \Exs \Big[ Z_\maux^{(\phi)}(h) \Big] &= u^2 \cdot \sum_{a
     \in \actionspace} \Exs_{\probx} \Big[ \propscore(\State, \action)
     \phi \Big( \frac{ |\weightfunc(\State, \action) h(\State,
       \action)| }{\propscore(\State, \action) u} \Big) \Big]\\ &\geq
   u^2 \sum_{a \in \actionspace} \Exs_{\probx} \Big[
     \propscore(\State, \action) \cdot \indicator \big[ \frac{
         |\weightfunc(\State, \action) h(\State, \action)|
       }{\propscore(\State, \action) u} \geq 2 \big] \Big] \\
   & = u^2 \cdot \Prob_{X \sim \probx, \Action \sim \propscore(\State,
     \cdot)} \Big( \frac{ |\weightfunc(\State, \Action) h(\State,
     \Action)| }{\propscore(\State, \Action) }\geq 2 u \Big).
\end{align*}
Recall that the constant $\smallballcon$ is the constant factor in the
small-ball probability condition~\ref{assume:small-ball}. Choosing the
threshold $u \mydefn \frac{c_1}{2}$ and using the equality
$\weightednorm{h} = 1$, we see that the small-ball condition implies
that
\begin{align*}
    \Prob_{\State \sim \probx, \Action \sim \propscore(\State, \cdot)}
    \Big( \frac{ |\weightfunc(\State, \Action) h(\State, \Action)|
    }{\propscore(\State, \Action) }\geq 2 u \Big) \geq \smallballprob.
\end{align*}

Now we turn to study the deviation bound for
$Z_\maux^{(\phi)}(h)$. Using known concentration inequalities for
empirical processes~\cite{adamczak2008tail}---see
Proposition~\ref{prop:concentration-adamczak} in
Appendix~\ref{app:recall-emp-proc} for more detail---we are guaranteed
to have
\begin{multline*}
\sup_{h \in \funcClassTemp} \Big( Z_\maux^{(\phi)}(h) - \Exs \big[
  Z_\maux^{(\phi)}(h) \big] \Big) \leq 2 \Exs \sup_{h \in
  \funcClassTemp} \Big( Z_\maux^{(\phi)}(h) - \Exs \big[
  Z_\maux^{(\phi)}(h) \big] \Big) + c \smallballcon^2 \cdot \Big\{
\sqrt{\frac{\log(1 / \failprob) }{\maux} } + \frac{\log(1 /
  \failprob)}{\maux} \Big\}.
\end{multline*}
with probability at least $1 - \failprob$.

For the expected supremum term, standard symmetrization arguments lead
to the bound
\begin{align*}
\Exs \sup_{h \in \funcClassTemp} \Big( Z_\maux^{(\phi)}(h) - \Exs
\big[Z_\maux^{(\phi)}(h) \big] \Big) \leq
\frac{\smallballcon^2}{\maux} \cdot \Exs \Big[ \sup_{h \in
    \funcClassTemp} \sum_{i = 1}^\maux \rade_i \phi \Big( \frac{2 |h
    (\State_i, \Action_i) \weightfunc(\State_i,
    \Action_i)|}{\smallballcon \propscore(\State_i, \Action_i)} \Big)
  \Big].
\end{align*}
Note that since $\phi$ is a $1$-Lipschitz function, we may apply the
Ledoux-Talagrand contraction (e.g., equation (5.6.1) in the
book~\cite{wainwright2019high}) so as to obtain
\begin{align*}
\Exs \Big[ \sup_{h \in \funcClassTemp} \sum_{i = 1}^\maux \rade_i \phi
  \Big( \frac{2 |h(\State_i, \Action_i) \weightfunc(\State_i,
    \Action_i)|}{\smallballcon \propscore(\State_i, \Action_i)} \Big)
  \Big] \leq \frac{4}{\smallballcon } \Exs \Big[ \sup_{h \in
    \funcClassTemp} \sum_{i = 1}^\maux \frac{\rade_i \weightfunc
    (\State_i, \Action_i)}{\propscore(\State_i, \Action_i) } h
  (\State_i, \Action_i) \Big].
\end{align*}
Combining the pieces yields the lower bound
\begin{multline}
\frac{1}{\maux} \sum_{i = 1}^\maux \tfrac{\weightfunc^2(\State_i,
  \Action_i)}{\propscore^2(\State_i, \Action_i)} h^2(\State_i,
\Action_i) \\
\geq \tfrac{\smallballprob \smallballcon^2}{4} - \frac{4
  \smallballcon}{\maux} \Exs \Big[ \sup_{h \in \funcClassTemp} \sum_{i
    = 1}^\maux \frac{\rade_i \weightfunc(\State_i,
    \Action_i)}{\propscore(\State_i, \Action_i) } h(\State_i,
  \Action_i) \Big] - c \smallballcon^2 \cdot \Big\{ \sqrt{\tfrac{\log
    (1 / \failprob) }{\maux} } + \tfrac{\log(1 / \failprob)}{\maux}
\Big\},
\end{multline}
uniformly holding true over $h \in \funcClassTemp$, with probability
$1 - \failprob$, which completes the proof of the lemma.

%%%%%%%%%%%%%%%%%%%%%%%%%%%%%%%%%%%%%%%%%%%%%%%%%%%%%%%%%%%%%%%%%%%

\section{Proofs of the corollaries}

This section is devoted to the proofs of
Corollaries~\ref{cor:example-linear}---\ref{cor:example-monotonic},
as stated in~\Cref{SecExamples}.

%%%%%%%%%%%%%%%%%%%%%%%%%%%%%%%%%%%%%%%%%%%%%%%%%%%%%%%%%%%%%%%%%%%%%%%%%%%%%%%%%%%%%

\subsection{Proof of Corollary~\ref{cor:example-linear}}
\label{app:subsec-proof-cor-example-linear}

Let us introduce the shorthand \mbox{$f_\theta(\state, \action) \defn
  \inprod{\theta}{\phi(\state, \action)}$} for functions that are
linear in the feature map.  Moreover, for a vector $\thetabar \in
\real^\usedim$ and radius $r > 0$, we define the recentering function
$\mubar(\state, \action) \defn \inprod{\thetabar}{\phi(\state,
  \action)}$.

Our proof strategy is to bound the pair of critical radii
$(\radone_\maux, \radtwo_\maux)$, and we do so by controlling the
associated Rademacher complexities.  By a direct calculation, we find
that
\begin{align*}
(\funcClass - \mubar) \cap \ball_\omega(r) \subseteq \Big\{ f_\theta
  \; \mid \; \theta^\top \Sigma \theta \leq r^2 \Big\}, \quad
  \mbox{where } \Sigma \mydefn \Exs \Big[ \frac{\weightfunc^2 (\State,
      \Action)}{\propscore^2 (\State, \Action)} \phi (\State, \Action)
    \phi (\State, \Action)^\top \Big].
\end{align*}
We can therefore bound the Rademacher complexities as
\begin{align*}
\SquaredRade_\maux^2 \Big( (\funcClass - \mubar) \cap \ball_\omega (r)
\Big) & \leq \Exs \Big[ \sup_{\vecnorm{\theta}{\Sigma} \leq r}
  \Big\{\frac{1}{\maux} \inprod{\theta}{\sum_{i = 1}^\maux
    \frac{\varepsilon_i \weightfunc^2 (\State_i,
      \Action_i)}{\propscore^2 (\State_i, \Action_i)} (\outcome_i -
    \treateff (\State_i, \Action_i)) \phi (\State_i, \Action_i)}
  \Big\}^2 \Big] \\
& = \frac{r^2}{\maux} \trace \Big( \Sigma^{-1} \Gamma_\sigma \Big),
\end{align*}
and
\begin{align*}
\radeComplexity_\maux \Big( (\funcClass - \mubar) \cap \ball_\omega(r)
\Big) \leq \Exs \Big[ \sup_{\vecnorm{\theta}{\Sigma} \leq r}
  \frac{1}{\maux} \inprod{\theta}{\sum_{i = 1}^\maux
    \frac{\varepsilon_i \weightfunc (\State_i, \Action_i)}{\propscore
      (\State_i, \Action_i)} \phi(\State_i, \Action_i)} \Big] \leq r
\sqrt{\frac{\usedim}{\maux}}.
\end{align*}
By definition of the fixed point equations, the critical radii can be
upper bounded as
\begin{align*}
    \radone_\maux \leq \sqrt{ \maux^{-1} \trace \Big( \Sigma^{-1}
      \Gamma_\sigma \Big) }, \quad \mbox{and} \quad \radtwo_\maux
    \leq \begin{cases} + \infty, & \maux \leq
      \frac{1024}{\smallballcon^2 \smallballprob^2} \usedim \\ 0, &
      \maux > \frac{1024}{\smallballcon^2 \smallballprob^2}
      \usedim \end{cases}.
\end{align*}
Combining with Theorem~\ref{cor:least-sqr-split-estimator} completes
the proof of this corollary.

%%%%%%%%%%%%%%%%%%%%%%%%%%%%%%%%%%%%%%%%%%%%%%%%%%%%%%%%%%%%%%%%%%%%%%%%%%%%%%%%%%%%%

\subsection{Proof of~\Cref{cor:example-sparse}}
\label{app:subsec-proof-cor-example-sparse}

We introduce the shorthand $f_\theta(\state, \action) =
\inprod{\theta}{\phi(\state, \action)}$ for functions that are linear
in the feature map.  Given any vector $\thetabar \in \real^\usedim$
such that $\vecnorm{\thetabar}{1} = \RadOne$, define the set $S =
\support (\thetabar) \subseteq [\usedim]$ along with the function
$\mubar(\state, \action) = \inprod{\thetabar}{\phi(\state,
  \action)}$. For any radius $r > 0$ and vector $\theta \in
(\funcClass - \mubar) \cap \ball_\omega(r)$, we note that
\begin{align*}
    \vecnorm{\theta_{S^c}}{1} = \vecnorm{\theta_{S^c} +
      \thetabar_{S^c}}{1} = \vecnorm{\theta + \thetabar}{1} -
    \vecnorm{\theta_S + \thetabar_S}{1} \leq \RadOne -
    \vecnorm{\thetabar_S}{1} + \vecnorm{\theta_{S}}{1} \leq
    \vecnorm{\theta_S}{1}.
\end{align*}
Recalling that $\Sigma = \Exs \big[ \tfrac{\weightfunc^2 (\State,
    \Action)}{\propscore^2 (\State, \Action)} \phi (\State, \Action)
  \phi (\State, \Action)^\top \big]$, we have the inclusions
\begin{align}
(\funcClass - \mubar) \cap \ball_\omega(r) & \subseteq r \cdot \Big\{
  f_\theta \mid \vecnorm{\theta_{S^c}}{1} \leq \vecnorm{\theta_S}{1},
  ~ \vecnorm{\theta}{\Sigma} \leq 1 \Big\} \nonumber\\
  & \subseteq r \cdot \Big\{f_\theta ~\mid~ \vecnorm{\theta}{1} \leq 2
  \sqrt{|S| / \lammin (\Sigma)} \Big\} \nonumber\\
\label{eq:sparse-example-local-set-characterize}  
  & \subseteq 2 r \sqrt{|S| / \lammin (\Sigma)} \cdot \conv \Big(
\big\{ \pm \phi_j \big\}_{j=1}^\usedim \Big).
\end{align}
where the second step follows from the bound
\mbox{$\vecnorm{\theta_S}{1} \leq \vecnorm{\theta_S}{2} \sqrt{|S|}
  \leq \vecnorm{\theta}{\Sigma} \sqrt{|S| / \lammin (\Sigma)}$,} valid
for any $\theta \in \real^d$.

For each coordinate $j = 1, \ldots, \usedim$, we can apply the
Hoeffding inequality along with the sub-Gaussian tail
assumption~\eqref{eq:subgaussian-in-sparse-linear-regression} so as to
obtain
\begin{align*}
\Prob \Big[ \Big| \frac{1}{\maux} \sum_{i = 1}^\maux \frac{\rade_i
    \weightfunc (\State_i, \Action_i)}{\propscore (\State_i,
    \Action_i)} \phi (\State_i, \Action_i)^\top \coordinate_j \Big|
  \geq t \Big] \leq 2 e^{- \frac{2 \maux t^2}{\subgaussian^2}} \quad
\mbox{for any $t > 0$/}
\end{align*}
Taking the union bound over $j = 1,2, \ldots, \usedim$ and then
integrating the resulting tail bound, we find that
\begin{align*}
  \Exs \Big[ \max_{j = 1, \ldots, \usedim } \Big|\frac{1}{\maux}
    \sum_{i=1}^\maux \frac{\weightfunc (\State_i,
      \Action_i)}{\propscore (\State_i, \Action_i)} \rade_i
    \phi_j(\State_i, \Action_i) \Big| \Big] \leq \subgaussian
  \sqrt{\frac{\log \usedim}{\maux}}.
\end{align*}
Combining with
equation~\eqref{eq:sparse-example-local-set-characterize}, we conclude
that
\begin{align*}
\radeComplexity_\maux \big( (\funcClass - \mubar) \cap \ball_\omega(r)
\big) \leq 2 r \subgaussian \sqrt{\frac{|S| \cdot \log
    (\usedim)}{\maux \lammin(\Sigma)}} \qquad \mbox{for any
  $\mubar(\state, \action) = \inprod{\thetabar}{\phi(\state,
    \action)}$ with $\thetabar$ supported on $S$.}
\end{align*}
Consequently, defining the constant $c_0 = \frac{4096}{\smallballcon^2
  \smallballprob^2}$, when the sample size satisfies $\maux \geq c_0
|S|\frac{\subgaussian^2 \log (\usedim) }{\lammin (\Sigma)}$, the
critical radius $\radtwo_\maux$ is $0$.

Now we turn to bound the critical radius $\radone_\maux$. By the
sub-Gaussian
condition~\eqref{eq:subgaussian-in-sparse-linear-regression}, we have
the Orlicz norm bound
\begin{multline*}
\Big \|\frac{\weightfunc^2 (\State_i, \Action_i)}{\propscore^2
  (\State_i, \Action_i)} \rade_i \phi_j (\State_i, \Action_i)
(\outcome_i - \treateff (\State_i, \Action_i)) \Big \|_{\psi_1} \\
\leq \Big \| \frac{\weightfunc (\State_i, \Action_i)}{\propscore
  (\State_i, \Action_i)} \phi_j(\State_i, \Action_i) \Big\|_{\psi_1}
\cdot \vecnorm{\frac{\weightfunc (\State_i, \Action_i)}{\propscore
    (\State_i, \Action_i)} (\outcome_i - \treateff (\State_i,
  \Action_i)) }{\psi_1} \leq \subgaussian \varbound.
\end{multline*}
Invoking a known concentration inequality (see
Proposition~\ref{prop:concentration-adamczak} in
Appendix~\ref{app:recall-emp-proc}), we conclude that there exists a
universal constant $c_1 > 0$ such that
\begin{align*}
  \Prob \left( \abss{ \frac{1}{\maux} \sum_{i = 1}^\maux
    \frac{\weightfunc^2 (\State_i, \Action_i)}{\propscore^2
      (\State_i, \Action_i)} \rade_i \phi_j (\State_i, \Action_i)
    (\outcome_i - \treateff (\State_i, \Action_i))} \geq t \right)
  \leq 2 \exp \left( \frac{- c_1 \maux t^2}{\subgaussian^2
    \varbound^2 + t \subgaussian \varbound \log (\maux) } \right),
\end{align*}
for any scalar $t > 0$.

Taking the union bound over $j = 1,2, \ldots, \usedim$ and integrating
out the tail yields
\begin{align*}
  \Exs \left[ \max_{j \in [\usedim] } \abss{\frac{1}{\maux} \sum_{i
        = 1}^\maux \frac{\weightfunc (\State_i,
        \Action_i)}{\propscore (\State_i, \Action_i)} \rade_i \phi_j
      (\State_i, \Action_i)}^2 \right] \leq c_2 \subgaussian^2
  \varbound^2 \Big\{ \sqrt{\frac{\log \usedim}{\maux}} + \frac{\log
    \usedim \cdot \log \maux}{\maux} \Big\}^2,
\end{align*}
Given a sample size lower bounded as $\maux \geq \log^2 \usedim$, the
derivation above guarantees that the Rademacher complexity is upper
bounded as
\begin{align*}
    \SquaredRade \big( (\funcClass - \mubar) \cap \ball_\omega (r)
    \big) \leq c r \subgaussian \varbound \sqrt{\frac{|S| \cdot \log
        (\usedim)}{\maux \lammin (\Sigma)}},
\end{align*}
and consequently, the associated critical radius satisfies an upper
bound of the form \mbox{$\radone_\maux \leq c \subgaussian \varbound
  \sqrt{\frac{|S| \log (\usedim)}{\maux \lammin (\Sigma)}}$.}
Combining with Theorem~\ref{cor:least-sqr-split-estimator} completes
the proof of Corollary~\ref{cor:example-sparse}.

%%%%%%%%%%%%%%%%%%%%%%%%%%%%%%%%%%%%%%%%%%%%%%%%%%%%%%%%%%%%%%%%%%%%%%%%%%%%%%%%%%%%

\subsection{Proof of Corollary~\ref{cor:example-holder}}
\label{app:subsec-proof-cor-example-holder}

Clearly, the function class $\funcClass_k$ is symmetric and
convex. Consequently, for any $\mubar \in \funcClass_k$, we have
\begin{align*}
    (\funcClass_k - \mubar) \cap \ball_\omega (r) \subseteq (2
  \funcClass_k) \cap \ball_\omega (r).
\end{align*}
For any pair $\plainmu_1, \plainmu_2 \in (2 \funcClass) \cap
\ball_\omega (r)$, by the sub-Gaussian assumption in
equations~\eqref{eq:subgaussian-in-nonparametric}, we have that
\begin{align*}
    \Exs \Big[ \big( \frac{\weightfunc (\State_i,
        \Action_i)}{\propscore (\State_i, \Action_i)} \rade_i
      (\plainmu_1 - \plainmu_2) (\State_i, \Action_i) \big)^2 \Big] &=
    \weightednorm{\plainmu_1 - \plainmu_2}^2, \quad
    \mbox{and}\\ \vecnorm{\frac{\weightfunc (\State_i,
        \Action_i)}{\propscore (\State_i, \Action_i)} \rade_i
      (\plainmu_1 - \plainmu_2) (\State_i, \Action_i)}{\psi_1} &\leq
    \subgaussian \vecnorm{\plainmu_1 - \plainmu_2}{\infty}.
\end{align*}
By a known concentration inequality (see
Proposition~\ref{prop:concentration-adamczak} in
Appendix~\ref{app:recall-emp-proc}), for any $t > 0$, we have
\begin{align*}
  \Prob \left( \abss{\frac{1}{\maux} \sum_{i = 1}^\maux
    \frac{\weightfunc (\State_i, \Action_i)}{\propscore (\State_i,
      \Action_i)} \rade_i (\plainmu_1 - \plainmu_2) (\State_i,
    \Action_i)} \geq t \right) \leq 2 \exp \left( \frac{- c_1 \maux
    t^2}{\weightednorm{\plainmu_1 - \plainmu_2}^2 + t \subgaussian
    \vecnorm{\plainmu_1 - \plainmu_2}{\infty} \log (\maux)} \right),
\end{align*}
We also note that the Cauchy--Schwarz inequality implies that
\begin{align*}
    \Exs \Big[ \sup_{\weightednorm{\plainmu_1 - \plainmu_2} \leq
        \delta} \frac{1}{\maux} \sum_{i = 1}^\maux \frac{\weightfunc
        (\State_i, \Action_i)}{\propscore (\State_i, \Action_i)}
      \rade_i (\plainmu_1 - \plainmu_2) (\State_i, \Action_i) \Big]
    \leq \delta.
\end{align*}

By a known mixed-tail chaining bound (see
Proposition~\ref{prop:dirksen} and
equation~\eqref{eq:mixed-tail-chaining-without-donsker} in
Appendix~\ref{app:recall-emp-proc}), we find that
\begin{multline}
\label{eq:bounding-rade-complexity-using-chaining}  
\radeComplexity_\maux \big( (\funcClass_k - \mubar) \cap
\ball_\omega(r) \big) \leq \frac{c}{\sqrt{\maux}} \dudley_2 \big(
(2 \funcClass_k) \cap \ball_\omega(r), \weightednorm{\cdot};
     [\delta, r] \big) \\ + \frac{c \subgaussian \log
       \maux}{\maux} \dudley_1 \big( (2 \funcClass_k) \cap
     \ball_\omega (r), \vecnorm{\cdot}{\infty}; [\delta, 2] \big)
     + 2 \delta,
\end{multline}
for any scalar $\delta \in [0, 2]$. Observing the norm domination
relation $\weightednorm{f} \leq \subgaussian \vecnorm{f}{\infty}$ for
any function $f$, we have $\dudley_2 \big( (2 \funcClass_k) \cap
\ball_\omega(r), \weightednorm{\cdot}; [\delta, r] \big) \leq
\dudley_2 \big( 2 \subgaussian \funcClass_k, \vecnorm{\cdot}{\infty};
[\delta, r] \big)$. As a result, in order to control the
       right-hand-side of
       equation~\eqref{eq:bounding-rade-complexity-using-chaining}, it
       suffices to bound the covering number of the class
       $\funcClass_k$ under the $\vecnorm{\cdot}{\infty}$-norm.

In order to estimate the Dudley chaining integral for the localized
class, we begin with the classical
bound~\cite{kolmogorov1959varepsilon}
\begin{align*}
    \log N \big(\funcClass_k, \vecnorm{\cdot}{\infty}; \varepsilon
    \big) \leq \Big( \frac{c}{\varepsilon} \Big)^{\pdim / k},
\end{align*}
where $c > 0$ is a universal constant. Using this bound, we can
control the Dudley entropy integrals for any $\alpha \in \{1, 2\}$,
$q > 0$, and interval $[\delta, u]$ with $u \in \{r, 2\}$. In particular, for any interval $[\delta, u]$ of the
non-negative real line, we have
\begin{align}
\label{eq:dudley-bound-for-holder-class}  
  \dudley_\alpha \Big( q \funcClass_k, \vecnorm{\cdot}{\infty};
         [\delta, u] \Big) \leq \int_\delta^u \Big( \frac{c
           q}{\varepsilon} \Big)^{\frac{\pdim}{\alpha k}} d
         \varepsilon \leq c q^{\frac{\pdim}{\alpha k}}
         \cdot \begin{cases} \frac{\alpha k}{\alpha k - \pdim} u^{1 -
             \frac{\pdim}{\alpha k}} & \mbox{if $\pdim < \alpha k$,}
           \\
             \log \big(u / \delta \big) & \mbox{if $\pdim = \alpha k$,} \\
             \frac{\alpha k}{\pdim - \alpha k} \big( \frac{c}{\delta}
             \big)^{\frac{\pdim}{\alpha k} - 1} & \mbox{if $\pdim > \alpha
               k$.}
     \end{cases}
\end{align}
We set $\delta = \big( \frac{\subgaussian}{\maux} \big)^{k / \pdim}$,
and use the resulting upper bound on the Dudley integral to control
the Rademacher complexity; doing so yields
\begin{align*}
    \radeComplexity_\maux \big( (\funcClass_k - \mubar) \cap
    \ball_\omega(r) \big) \leq c_{\subgaussian, \pdim/k}
    \cdot \begin{cases} r^{1 - \frac{\pdim}{2 k}} / \sqrt{\maux} +
      \log \maux \cdot \maux^{- k / \pdim} & \mbox{if $\pdim < 2k$,}
      \\
      \log (\maux) / \sqrt{\maux} & \mbox{if $\pdim = 2k$,} \\
      \maux^{- k / \pdim} & \mbox{if $\pdim > 2k$.}
    \end{cases}
\end{align*}
Solving the fixed point equation~\eqref{eq:defn-critical-radius-plain}
yields
\begin{align*}
\radtwo_\maux \leq c'_{\subgaussian, \pdim/k} \maux^{- k / \pdim}
\cdot \log \maux,
\end{align*}
where the constant $c_{\subgaussian, \pdim/k}$ and $c_{\subgaussian,
  \pdim/k}'$ depend on the parameters $(\subgaussian, \pdim/k)$, along
with the small ball constants $(\smallballcon, \smallballprob)$.

Turning to the critical radius $\radone_\maux$, we note that each term
in the empirical process associated with the observation noise
satisfies
\begin{multline*}
    \Exs \Big[ \Big\{ \frac{\weightfunc^2 (\State_i,
        \Action_i)}{\propscore^2 (\State_i, \Action_i)} (\outcome_i -
      \treateff (\State_i, \Action_i)) \rade_i (\plainmu_1 -
      \plainmu_2) (\State_i, \Action_i) \Big\}^2 \Big] \\ \leq
    \sqrt{\Exs \Big[ \Big\{ \frac{\weightfunc (\State_i,
          \Action_i)}{\propscore (\State_i, \Action_i)} (\outcome_i -
        \treateff (\State_i, \Action_i)) \Big\}^4 \Big]} \cdot
    \sqrt{\Exs \Big[ \Big\{ \frac{\weightfunc (\State_i,
          \Action_i)}{\propscore (\State_i, \Action_i)} (\plainmu_1 -
        \plainmu_2) (\State_i, \Action_i) \Big\}^4 \Big]} \\
    \leq \varbound^2 \ctwofour \weightednorm{\plainmu_1 -
      \plainmu_2}^2,
\end{multline*}
and
\begin{multline*}
    \vecnorm{\frac{\weightfunc^2 (\State_i, \Action_i)}{\propscore^2
        (\State_i, \Action_i)} (\outcome_i - \treateff (\State_i,
      \Action_i)) \rade_i (\plainmu_1 - \plainmu_2) (\State_i,
      \Action_i)}{\psi_1}\\ \leq \vecnorm{\frac{\weightfunc (\State_i,
        \Action_i)}{\propscore (\State_i, \Action_i)} (\outcome_i -
      \treateff (\State_i, \Action_i)) }{\psi_2} \cdot
    \vecnorm{\frac{\weightfunc (\State_i, \Action_i)}{\propscore
        (\State_i, \Action_i)} (\plainmu_1 - \plainmu_2) (\State_i,
      \Action_i)}{\psi_2} \\
\leq \varbound \subgaussian \vecnorm{\plainmu_1 - \plainmu_2}{\infty}.
\end{multline*}
Following the same line of derivation in the bound for the Rademacher
complexity $\radeComplexity_\maux$, we use the mixed-tail chaining
bound to find that
\begin{multline*}
    \SquaredRade_\maux \big( (\funcClass_k - \mubar) \cap \ball_\omega
    (r) \big) \leq \frac{c \varbound \sqrt{\ctwofour}}{\sqrt{\maux}}
    \dudley_2 \big( (2 \funcClass_k) \cap \ball_\omega (r),
    \weightednorm{\cdot}; [\delta, r] \big) \\ + \frac{c
      \varbound\subgaussian \log \maux}{\maux} \dudley_1 \big( (2
    \funcClass_k) \cap \ball_\omega (r), \vecnorm{\cdot}{\infty};
              [\delta, 2] \big) + 2 \delta,
\end{multline*}
valid for all $\delta \in [0, 2]$.  The Dudley integral
bound~\eqref{eq:dudley-bound-for-holder-class} then implies
\begin{align*}
    \SquaredRade_\maux \big( (\funcClass_k - \mubar) \cap \ball_\omega
    (r) \big) \leq c_{\subgaussian, \pdim/k} \varbound \cdot
 \begin{cases}
      r^{1 - \frac{\pdim}{2 k}} / \sqrt{\maux} + \log \maux \cdot
      \maux^{- k / \pdim} & \mbox{if $\pdim < 2k$,} \\
      \log (\maux) / \sqrt{\maux} & \mbox{if $\pdim = 2k$,} \\
\maux^{- k/\pdim} & \mbox{if $\pdim > 2k$,}
    \end{cases}
\end{align*}
where the constant $c_{\subgaussian, \pdim/k}$ depends on the
parameters $(\subgaussian, \pdim/k)$ and the constant $\ctwofour$.
Solving the fixed point equation yields
\begin{align*}
\radone_\maux \leq c_{\subgaussian, \pdim/k} \varbound
\cdot \begin{cases} \maux^{- \frac{k}{2 k + \pdim}} & \mbox{if $\pdim
    < 2k$,} \\
  \maux^{-1/4} \sqrt{\log \maux} & \mbox{if $\pdim = 2k$,} \\
  \maux^{- \frac{k}{2\pdim}} & \mbox{if $\pdim > 2k$.}
    \end{cases}
\end{align*}
Combining with Theorem~\ref{cor:least-sqr-split-estimator} completes
the proof of Corollary~\ref{cor:example-holder}.

%%%%%%%%%%%%%%%%%%%%%%%%%%%%%%%%%%%%%%%%%%%%%%%%%%%%%%%%%%%%%%%%%%%%%%%%%%%%%%%%%%%%%

\subsection{Proof of~\Cref{cor:example-monotonic}}
\label{app:subsec-proof-cor-example-monotonic}

For any $\mubar \in \funcClass$, define the function class
\begin{align*}
    \funcClassTemp \mydefn \Big\{ (\state, \action) \rightarrow
    \frac{\weightfunc (\state, \action)}{\propscore (\state, \action)}
    f (\state, \action) \; \mid \; f \in \ball_\omega (r) \cap
    (\funcClass - \mubar) \Big\}.
\end{align*}
Clearly, the class $\funcClassTemp$ is uniformly bounded by $b$, and
for any $f \in \funcClass$, we have the upper bound \mbox{$\Exs \Big[
    \abss{\frac{\weightfunc (\State, \Action)}{\propscore (\State,
        \Action)} f(\State, \Action)}^2 \Big] = \weightednorm{f}^2
  \leq r^2$.}

Invoking a known bracketing bound on empirical processes (cf.
Prop.~\ref{prop:vandervart-wellner} in
Appendix~\ref{app:recall-emp-proc}), we have
\begin{align}
\Exs \Big[ \sup_{h \in \funcClassTemp} \frac{1}{\maux} \sum_{i =
    1}^\maux \rade_i h (\State_i, \Action_i) \Big] \leq
\frac{c}{\sqrt{\maux}} \DudleyBracket \big( \funcClassTemp,
\vecnorm{\cdot}{\Ltwospace}; [0, r] \big) \Big\{1 + \frac{b
  \DudleyBracket \big( \funcClassTemp, \vecnorm{\cdot}{\Ltwospace} ;
         [0,r] \big)}{r^2 \sqrt{\maux}}
\Big\}\label{eq:bracketing-bound-monotone-example}
\end{align}

For functions $\ell, f, u : [0,1] \rightarrow \real$, such that $f$ is
contained in the bracket $[\ell, u]$, we let:
\begin{align*}
  \widetilde{\ell} (\state, \action) \mydefn \frac{\weightfunc(\state,
    \action)}{\propscore (\state, \action)} \Big\{ \ell (\phi (\state,
  \action)) \bm{1}_{\weightfunc (\state, \action) > 0} + u (\phi
  (\state, \action)) \bm{1}_{\weightfunc (\state, \action) < 0} -
  \mubar (\state, \action) \Big\},\\ \widetilde{u} (\state, \action)
  \mydefn \frac{\weightfunc (\state, \action)}{\propscore (\state,
    \action)} \Big\{ u (\phi (\state, \action)) \bm{1}_{\weightfunc
    (\state, \action) > 0} + \ell (\phi (\state, \action))
  \bm{1}_{\weightfunc (\state, \action) < 0} - \mubar (\state,
  \action) \Big\}.
\end{align*}
It is easily observed that the function $(\state, \action) \mapsto
\tfrac{\weightfunc (\state, \action)}{\propscore (\state, \action)} (f
- \mubar) (\state, \action)$ lies in the bracket $[\widetilde{\ell},
  \widetilde{u}]$, and for any probability law $\mathbb{Q}$ on
$\Xspace \times \actionspace$, we have $\vecnorm{\widetilde{u} -
  \widetilde{\ell}}{\Ltwospace (\mathbb{Q})} \leq b \cdot \vecnorm{u -
  \ell}{\Ltwospace (\mathbb{Q}_\phi)}$, where $\mathbb{Q}_\phi$ is the
probability law of $\phi (\State, \Action)$ for $(\State, \Action)
\sim \mathbb{Q}$.

It is known (cf. Thm 2.7.5 in the book~\cite{vander1996}) that the
space of monotonic functions from $[0, 1]$ to $[0, 1]$ has
$\varepsilon$-bracketing number under any $\Ltwospace$-norm bounded by
$\exp \big( c / \varepsilon \big)$ for any $\varepsilon > 0$.
Substituting back into the bracketing entropy
bound~\eqref{eq:bracketing-bound-monotone-example} yields
\begin{align*}
  \radeComplexity_\maux \big( \ball_\omega (r) \cap (\funcClass -
  \mubar) \big) \leq c \Big\{ \sqrt{\frac{br}{\maux}} + \frac{b^2}{r
    \maux}\Big\}.
\end{align*}
From the definition of the fixed point equation, we can bound the
critical radius $\radtwo$ as
\begin{align*}
\radtwo_\maux \leq \frac{c b}{\maux} + \frac{c b}{\sqrt{\maux}},
\end{align*}
where $c > 0$ is a universal constant.

Turning to the squared Rademacher process associated with the outcome
noise, we construct the function class
\begin{align*}
\funcClassTemp' \mydefn \Big\{ (\state, \action, y) \rightarrow y
\cdot \frac{\weightfunc^2 (\state, \action)}{\propscore^2 (\state,
  \action)} f (\state, \action) \; \mid \; f \in \ball_\omega (r) \cap
(\funcClass - \mubar) \Big\}.
\end{align*}
For functions $\ell, f, u : [0,1] \rightarrow \real$, such that $f$ is
contained in the bracket $[\ell, u]$, we can similarly construct
\begin{align*}
 \widetilde{\ell}(\state, \action, y) \mydefn y \cdot
 \frac{\weightfunc^2 (\state, \action)}{\propscore^2 (\state,
   \action)} \Big\{ \ell (\phi (\state, \action)) \bm{1}_{y > 0} + u
 (\phi (\state, \action)) \bm{1}_{y < 0} - \mubar (\state, \action)
 \Big\},\\ \widetilde{u} (\state, \action, y) \mydefn y \cdot
 \frac{\weightfunc^2 (\state, \action)}{\propscore^2 (\state,
   \action)} \Big\{ u (\phi (\state, \action)) \bm{1}_{y > 0} + \ell
 (\phi (\state, \action)) \bm{1}_{y < 0} - \mubar (\state, \action)
 \Big\}.
\end{align*}
It is easily observed that the function $(\state, \action y) \mapsto y
\cdot \tfrac{\weightfunc (\state, \action)}{\propscore (\state,
  \action)} (f - \mubar) (\state, \action)$ lies in the bracket
$[\widetilde{\ell}, \widetilde{u}]$, and for any probability law
$\mathbb{Q}$ on $\Xspace \times \actionspace \times \real$, we have
$\vecnorm{\widetilde{u} - \widetilde{\ell}}{\Ltwospace (\mathbb{Q})}
\leq b^2 \cdot \vecnorm{u - \ell}{\Ltwospace (\mathbb{Q}_\phi)}$,
where $\mathbb{Q}_\phi$ is the probability law of $\phi(\State,
\Action)$ for $(\State, \Action, Y) \sim \mathbb{Q}$. Applying the
bracketing bound yields
\begin{align*}
\Exs \Big[ \sup_{h \in \funcClassTemp'} \frac{1}{\maux} \sum_{i =
    1}^\maux \rade_i h (\State_i, \Action_i, \outcome_i - \treateff
  (\State_i, \Action_i)) \Big] & \leq \frac{c}{\sqrt{\maux}}
\DudleyBracket \big( \funcClassTemp', \vecnorm{\cdot}{\Ltwospace} ; [0,
  br] \big) \Big\{1 + \frac{b \DudleyBracket \big( \funcClassTemp',
  \vecnorm{\cdot}{\Ltwospace} ; [0, br] \big)}{(br)^2 \sqrt{\maux}}
\Big\} \\
& \leq c b \Big( \sqrt{\frac{r}{\maux}} + \frac{1}{r \maux} \Big).
\end{align*}
Denote $Z_\maux \mydefn \sup_{h \in \funcClassTemp'} \frac{1}{\maux}
\sum_{i = 1}^\maux \rade_i h (\State_i, \Action_i, \outcome_i -
\treateff (\State_i, \Action_i))$. By a standard functional Bernstein
bound (e.g., Thm. 3.8 in the book~\cite{wainwright2019high}), we have
the tail bound
\begin{align*}
\Prob \Big[ Z_\maux \geq 2 \Exs [Z_\maux] + t \Big] & \leq 2 \exp
\left( \frac{- \maux t^2}{56 (br)^2 + 4 b^2 t} \right) \quad \mbox{for
  any $t > 0$}.
\end{align*}
Combining with the expectation bound, we conclude that
$\SquaredRade_\maux = \sqrt{\Exs [Z_\maux^2]} \leq 2 c b \big(
\sqrt{\frac{r}{\maux}} + \frac{1}{r \maux} \big)$.  By definition of
fixed point equation, the critical radius can be upper bounded
$\radone_\maux \leq c \big(\frac{b^2}{\maux} \big)^{1/3}$, and
substituting this bound into~\Cref{cor:least-sqr-split-estimator}
completes the proof of this corollary.

\subsection{Strong shattering for sparse linear models}
\label{subsec:strong-shatter-sparse}

In this section, we state and prove the claim from
Example~\ref{ExaSparseLinear} about the size of the fat shattering
dimension for the class of sparse linear models.

\begin{proposition}
\label{prop:sparse-hypercube-packing}
There is a universal constant $c > 0$ such that the function class
$\funcClass_s$ of $s$-sparse linear models over $\real^\pdim$
satisfies the strong shattering condition~\eqref{EqnStrongShatter}
with fat shattering dimension $\fatdim = c s \log(e \,\pdim / s)$ at
scale $\delta = 1$.
\end{proposition}
\begin{proof}
We assume without loss of generality (adjusting constants as needed)
that $\pdim/s = 2^k$ is an integer power of two.  Our argument
involves constructing a set of vectors by dividing the $\pdim$
coordinates into $s$ blocks. Let the matrix $A \in \{0, 1\}^{k \times
  2^k}$ be such that by sequentially writing down the elements in
$j$-th column, we get the binary representation of the integer $(j -
1)$, for $j = 1,2, \ldots, 2^k$. Let $(a_i^\top)_{1 \leq i \leq k}$ be
the row vectors of the matrix $A$. For $i \in [k]$ and $j \in s$, we
construct the $\pdim$-dimensional data vector as $x_{i, j} = a_i
\otimes e_j$, where the $e_j \in \real^s$ is the indicator vector of
$j$-th coordinate. The cardinality of this set is given by
\begin{align*}
\abss{\big\{ x_{i, j}: ~i \in [k], j \in s \big\}} = k s =
\frac{1}{\log 2} \cdot s \log(\pdim / s).
\end{align*}
It suffices to construct a hypercube packing for this set.  Given a
binary vector $v \in \{0, 1\}^{k}$, we let $J(v) \in \{1, 2, \ldots,
2^k \}$ such that the $J(v)$-th column of the matrix $A$ is equal to
$v$. (Note that our construction ensures that such a column always
exists and is unique.)

Given any binary vector $\zeta \in \{0, 1\}^{k \times s}$, we
construct the following vector:
\begin{align*}
\beta_\zeta \mydefn \sum_{i = 1}^s e \big( J(\zeta_{i, 1}, \zeta_{i,
  2}, \ldots, \zeta_{i, k}) \big) \otimes e_i
\end{align*}
where the function $e:[2^k] \rightarrow \real^{2^k}$ maps the integer
$j$ to the indicator vector of $j$-th coordinate.

We note that the vector $\beta$ is supported on $s$-coordinates, with
absolute value of each coordinate bounded by $1$.  Moreover, our
construction ensures that $i \in [k]$ and $j \in [s]$,
\begin{align*}
\beta_\zeta^\top x_{i, j} = a_i^\top e \big( J(\zeta_{i, 1}, \zeta_{i,
  2}, \ldots, \zeta_{i, k}) \big) = \zeta_{i, j}.
\end{align*}
Therefore, we have planted a hypercube $\prod_{i \in [k], j \in s}
\big(x_{i, j}, \{0, 1\} \big)$ in the graph of the function class
$\smallsuper{\funcClass}{sparse}_s$, which completes the proof of the
claim.
\end{proof}

%%%%%%%%%%%%%%%%%%%%%%%%%%%%%%%%%%%%%%%%%%%%%%%%%%%%%%%%%%%%%%%%%

\section{Some elementary inequalities and their proofs}

In this section, we collect some elementary results used throughout
the paper, as well as their proofs.

\subsection{Bounds on conditional total variation distance}

The following lemma is required for the truncation arguments used in
the proofs of our minimax lower bounds. In particular, it allows us to
make small modifications on a pair of probability laws by conditioning
on good events, without inducing an overly large change in the total
variation distance.
\begin{lemma}
\label{lemma:total-variation-with-truncation}
Let $(\mu, \nu)$ be a pair of probability distributions over the same
Polish space $\mathcal{S}$, and consider a subset $\Event \subseteq
\mathcal{S}$ such that $\min \big \{ \mu(\Event), \nu(\Event) \big \}
\geq 1 - \varepsilon$ for some $\varepsilon \in [0, 1/4]$.  Then the
conditional distributions $(\mu \mid \Event)$ and $(\nu \mid \Event)$
satisfy the bound
\begin{align}
    \totalvariation(\mu, \nu) - 4 \varepsilon \stackrel{(i)}{\leq}
    \totalvariation \big[ (\mu \mid \Event), (\nu \mid \Event) \big]
    \stackrel{(ii)}{\leq} \tfrac{1}{1 - \varepsilon}
    \totalvariation(\mu, \nu) + 2 \varepsilon.
  \end{align}
\end{lemma}
\begin{proof}
Recall the variational definition of the TV distance as the supremum
over functions $f: \statespace \rightarrow \real$ such that
$\vecnorm{f}{\infty} \leq 1$.  For any such function $f$, we have
\begin{align*}
\abss{\Exs_\mu[f(X)] - \Exs_\nu[f(X)]} & \leq \abss{\Exs_\mu[f(X)
    \bm{1}_{X \in \Event}] - \Exs_\nu[f(X) \bm{1}_{X \in \Event}]} +
\Exs_\mu[|f(X)| \; \bm{1}_{\Event^c}] + \Exs_\nu[|f(X)| \; \bm{1}_{X
    \in \Event^c}] \\
& \leq \abss{\frac{\Exs_\mu[f(X) \bm{1}_{X \in \Event}]}{\mu(\Event)}
  - \frac{\Exs_\nu[f(X) \bm{1}_{X \in \Event}]}{\nu(\Event)} } +
\abss{\frac{1}{\mu(\Event)} - \frac{1}{\nu(\Event)}} \Exs_\nu[|f(X)|]
+ 2 \varepsilon \\
& \leq \totalvariation \big( (\mu \mid \Event), ( \nu \mid \Event)
\big) + 4 \varepsilon,
\end{align*}
and re-arranging yields the lower bound (i).

On the other hand, in order to prove the upper bound (ii), we note
that
\begin{align*}
 \abss{\Exs_{\mu| \Event}[f(X)] - \Exs_{\nu | \Event}[f(X)]}  &=
 \frac{1}{\mu(\Event)} \abss{\Exs_\mu[f(X) \bm{1}_{X \in \Event}] -
   \Exs_\nu[f(X) \bm{1}_{X \in \Event}]
   \frac{\mu(\Event)}{\nu(\Event)} } \\ &\leq \frac{1}{\mu(\Event)}
 \abss{\Exs_\mu[f(X) \bm{1}_{X \in \Event}] - \Exs_\nu[f(X)
     \bm{1}_{X \in \Event}] } + \Exs_\nu[|f(X)|] \cdot \abss{
   \frac{\mu(\Event)}{\nu(\Event)} - 1}\\ &\leq \frac{1}{1 -
   \varepsilon} \totalvariation(\mu, \nu) + 2 \varepsilon,
\end{align*}
which completes the proof.
\end{proof}

%%%%%%%%%%%%%%%%%%%%%%%%%%%%%%%%%%%%%%%%%%%%%%%%%%%%%%%%%%%%%%%

\subsection{A second moment lower bound for truncated random variable}

The following lemma is frequently used in our lower bound
constructions.
\begin{lemma}
\label{lemma:simple-moment-lower-bound-under-trunc}
Let $X$ be a real-valued random variable with finite fourth moment,
and define the $(2$--$4)$-moment constant $\ctwofour \mydefn
\sqrt{\Exs[X]^4} / \Exs[X^2]$.  Then we have the lower bound
\begin{align*}
 \Exs \Big[X^2 \cdot \bm{1} \big\{ |X| \leq 2 \ctwofour
   \sqrt{\Exs[X^2]} \big\} \Big] \geq \frac{1}{2} \Exs[X^2].
  \end{align*}
\end{lemma}
\begin{proof}
Without loss of generality, we can assume that $\Exs[X^2] = 1$.
Applying Cauchy--Schwarz inequality implies that
\begin{align*}
\Exs \Big[ X^2 \bm{1} \big\{ |X| \geq 2 \ctwofour \big\} \Big] \leq
\sqrt{\Exs \big[ X^4 \big]} \cdot \sqrt{\Prob \Big( |X| \geq 2
  \ctwofour \Big) } \leq \ctwofour \cdot \sqrt{\Prob \Big( |X| \geq 2
  \ctwofour \Big)}.
\end{align*}
By Markov's inequality, we have
\begin{align*}
\Prob \Big( |X| \geq 2 \ctwofour \Big) \leq \frac{\Exs[X^2]}{4
  \ctwofour^2} = \frac{1}{4 \ctwofour^2}.
\end{align*}
Substituting back to above bounds, we conclude that $\Exs \big[X^2
  \bm{1} \big\{ |X| \geq 2 \ctwofour \big\} \big] \leq \frac{1}{2}$,
and consequently,
\begin{align*}
\Exs \Big[X^2 \bm{1} \big\{ |X| \leq 2 \ctwofour \big\} \Big] =
\Exs[X^2] - \Exs \Big[X^2 \bm{1} \big\{ |X| \geq 2 \ctwofour \big\}
  \Big] \geq \frac{1}{2},
\end{align*}
which completes the proof.
\end{proof}

%%%%%%%%%%%%%%%%%%%%%%%%%%%%%%%%%%%%%%%%%%%%%%%%%%%%%%%%%%%%%%%%%%%%%%%%%%%%%%%%%%%%%%%%%%

\section{Empirical process results from existing literature}
\label{app:recall-emp-proc}

In this appendix, we collect some known bounds on the suprema of
empirical processes.

\subsection{Concentration for unbounded empirical processes}

We use a concentration inequality for unbounded empirical processes.
It applies to a countable class $\funcClass$ of measurable functions,
and a supremum of the form
\begin{align*}
  Z & \mydefn \sup_{f \in \funcClass} \abss{\sum_{i = 1}^n f(X_i)}
\end{align*}
where $\{X_i \}_{i=1}^\numobs$ is a sequence of independent random
variables such that $\Exs[f(X_i)] = 0$ for any $f \in \funcClass$.

\begin{proposition}[Theorem 4 of~\cite{adamczak2008tail}, simplified]
  \label{prop:concentration-adamczak}
 There exists a universal constant $c > 0$ such that for any $t > 0$
 and $\alpha \geq 1$, we have
\begin{align*}
\Prob \left[ Z > 2 \Exs (Z) + t \right] & \leq \exp \left( \frac{-
  t^2}{4 v^2} \right) + 3 \exp \Big( - \Big( \frac{t}{c \vecnorm{\max
    \limits_{i = 1, \ldots, \numobs} \sup \limits_{f \in \funcClass}
    |f(X_i)| }{\psi_{1 /\alpha}}} \Big)^{1 / \alpha} \Big),
\end{align*}
where $v^2 \mydefn \sup_{f \in \funcClass} \sum_{i = 1}^\numobs \Exs
[f^2(X_i)]$ is the maximal variance.
\end{proposition}
The countability assumption can be easily relaxed for separable
spaces.  A useful special case of
Proposition~\ref{prop:concentration-adamczak} is by taking the class
$\funcClass$ to be a singleton and letting $\alpha = 1$, in which case
the bound becomes
\begin{align*}
    \abss{\frac{1}{\numobs} \sum_{i = 1}^\numobs f (X_i) - \Exs [f
        (X)] } \leq c \sqrt{\var \big(f (X) \big) \frac{\log (1 /
        \delta)}{\numobs}} + \frac{c \log \numobs}{\numobs} \vecnorm{f
      (X)}{\psi_1} \cdot \log (1 / \delta),
\end{align*}
with probability $1 - \delta$.

\subsection{Some generic chaining bounds}

We also use a known generic chaining tail bound.  It involves a
separable stochastic process $(Y_t)_{t \in T}$ and a pair $(d_1, d_2)$
of metrics over the index set $T$.  We assume that there exists some
$t_0 \in T$ such that $Y_{t_0} \equiv 0$.
\begin{subequations}
\begin{proposition}[Theorem 3.5 of Dirksen~\cite{dirksen2015tail}]\label{prop:dirksen}
Suppose that for any pair $s, t \in T$, the difference $Y_s - Y_t$
satisfies the mixed tail bound
\begin{align}
\label{eq:mixed-tail-condition}  
 \Prob \Big( \abss{Y_{s} - Y_{t}} \geq \sqrt{u} d_1(s, t) + u d_2(s,
 t) \Big) \leq 2 e^{- u} \quad \mbox{for any $u > 0$.}
    \end{align}
Then for any $\ell \geq 1$, we have the moment bound
\begin{align}
  \left\{ \Exs \Big[ \sup_{t \in T} \abss{Y_t}^\ell \Big]
  \right\}^{1/\ell} \leq c \Big( \talchain{2} (T, d_1) + \talchain{1}
  (T, d_2) \Big) + 2 \sup_{t \in T} \big( \Exs \abss{Y_t}^\ell
  \big)^{1/\ell},
\end{align}
where $\talchain{\alpha} (T, d)$ is the generic chaining functional of
order $\alpha$ for the metric space $(T, d)$.
\end{proposition}
\end{subequations}
For a set $T$ with diameter bounded by $r$ under the metric $d$, the
generic chaining functional can be upper bounded in terms of the
Dudley entropy integral as
\begin{align*}
  \talchain{\alpha}(T, d) & \leq c \dudley_\alpha \big(T, d; [0, r]
  \big) \quad \mbox{for each $\alpha \in \{1, 2\}$}
\end{align*}
(e.g., cf. Talagrand~\cite{talagrand2006generic}).  Furthermore,
suppose that the norm domination relation $d_1 (s, t) \leq a_0 d_2 (s,
t)$ holds true for any pair $s, t \in T$. Let $r_1, r_2$ be the
diameter of the set $T$ under the metrics $d_1, d_2$, respectively. If
we apply Proposition~\ref{prop:dirksen} to a maximal $\delta$-packing
for the set $T$ under metric $d_1$, we immediately have
\begin{multline}
\label{eq:mixed-tail-chaining-without-donsker}  
  \left\{ \Exs \Big[ \sup_{t \in T} \abss{Y_t}^p \Big] \right\}^{1/p}
  \leq c \Big\{ \dudley_2 \big(T, d_1; [\delta, r_1]\big) + \dudley_1
  (T, d_2; [\delta / a_0, r_2]) \Big\}\\ + \Big\{ \Exs
  \sup_{\substack{s, t\in T \\ d_1 (s, t) \leq \delta}} \abss{Y_s -
    Y_t}^p \Big\}^{1/p} + 2 \sup_{t \in T} \big( \Exs \abss{Y_t}^p
  \big)^{1/p}.
\end{multline}

\subsection{Bracketing entropy bounds}

Finally, we use the following bracketing integral bound for empirical
processes:
\begin{proposition}[Lemma 3.4.2 of~\cite{vander1996}]
\label{prop:vandervart-wellner}
Let $\funcClass$ be a class of measurable functions, such that $\Exs
[f^2(X)] \leq r^2$ and $|f (X)| \leq M$ almost surely for any $f \in
\funcClass$. Given $\numobs$ $\mathrm{i.i.d.}$ samples $\{X_i \}_{i =
  1}^\numobs$, we have
\begin{align*}
\Exs \Big[ \sup_{f \in \funcClass} \frac{1}{\numobs} \sum_{i =
    1}^\numobs f(X_i) - \Exs [f (X)] \Big] \leq
\frac{c}{\sqrt{\numobs}} \DudleyBracket \big( \funcClass,
\vecnorm{\cdot}{\Ltwospace}; [0, r] \big) \Big\{1 + \frac{M
  \DudleyBracket \big( \funcClass, \vecnorm{\cdot}{\Ltwospace} ; [0,
    r] \big)}{r^2 \sqrt{\numobs}} \Big\}.
\end{align*}
\end{proposition}

\end{document}

%% file: final_macros.tex
\def\ball{\mathbb{B}}

\newcommand{\real}{\ensuremath{\mathbb{R}}}

\newcommand{\order}[1]{\ensuremath{\mathcal{O}\parenth{#1}}}

\newcommand{\thetabar}{\ensuremath{\bar{\theta}}}

\newcommand{\Xspace}{\ensuremath{\mathbb{X}}}
\newcommand{\Yspace}{\ensuremath{\mathbb{Y}}}

\newcommand{\parenth}[1]{\left( #1 \right)}

\newcommand{\abss}[1]{\left| #1 \right |}

\newcommand{\law}{\ensuremath{\mathcal{L}}}

\newcommand{\mydefn}{\ensuremath{:=}}

\newcommand{\defn}{:=}

\newcommand{\vecnorm}[2]{\| #1\|_{#2}}
\newcommand{\inprod}[2]{\ensuremath{\langle #1 , \, #2 \rangle}}

\newcommand{\kull}[2]{\ensuremath{D_{\text{KL}}\left(#1\; \| \; #2 \right)}}

\newcommand{\Exs}{\ensuremath{{\mathbb{E}}}}
\newcommand{\Prob}{\ensuremath{{\mathbb{P}}}}

\DeclareMathOperator{\var}{var}
\DeclareMathOperator{\cov}{cov}
\DeclareMathOperator{\trace}{trace}

\newtheoremstyle{named}{}{}{\itshape}{}{\bfseries}{.}{.5em}{\thmnote{#3's }#1}
\theoremstyle{named}

\theoremstyle{plain}

\newtheorem{theorem}{Theorem}
\newtheorem{proposition}{Proposition}

\newtheorem{lemma}{Lemma}

\newtheorem{corollary}{Corollary}

\newlength{\widebarargwidth}
\newlength{\widebarargheight}
\newlength{\widebarargdepth}
\DeclareRobustCommand{\widebar}[1]{%
  \settowidth{\widebarargwidth}{\ensuremath{#1}}%
  \settoheight{\widebarargheight}{\ensuremath{#1}}%
  \settodepth{\widebarargdepth}{\ensuremath{#1}}%
  \addtolength{\widebarargwidth}{-0.3\widebarargheight}%
  \addtolength{\widebarargwidth}{-0.3\widebarargdepth}%
  \makebox[0pt][l]{\hspace{0.3\widebarargheight}%
    \hspace{0.3\widebarargdepth}%
    \addtolength{\widebarargheight}{0.3ex}%
    \rule[\widebarargheight]{0.95\widebarargwidth}{0.1ex}}%
  {#1}}

\makeatletter
\long\def\@makecaption#1#2{
        \vskip 0.8ex
        \setbox\@tempboxa\hbox{\small {\bf #1:} #2}
        \parindent 1.5em  %% How can we use the global value of this???
        \dimen0=\hsize
        \advance\dimen0 by -3em
        \ifdim \wd\@tempboxa >\dimen0
                \hbox to \hsize{
                        \parindent 0em
                        \hfil
                        \parbox{\dimen0}{\def\baselinestretch{0.96}\small
                                {\bf #1.} #2
                                }
                        \hfil}
        \else \hbox to \hsize{\hfil \box\@tempboxa \hfil}
        \fi
        }
\makeatother

\long\def\comment#1{}
\definecolor{battleshipgrey}{rgb}{0.52, 0.52, 0.51}
\definecolor{darkgray}{rgb}{0.66, 0.66, 0.66}
\definecolor{darkgreen}{rgb}{0.0, 0.2, 0.13}
\definecolor{darkspringgreen}{rgb}{0.09, 0.45, 0.27}
\definecolor{dukeblue}{rgb}{0.0, 0.0, 0.61}
\definecolor{olivedrab7}{rgb}{0.24, 0.2, 0.12}
\definecolor{darkblue}{rgb}{0.0, 0.0, 0.55}
\definecolor{darkscarlet}{rgb}{0.34, 0.01, 0.1}
\definecolor{candyapplered}{rgb}{1.0, 0.03, 0.0}
\definecolor{ao(english)}{rgb}{0.0, 0.5, 0.0}
\definecolor{applegreen}{rgb}{0.55, 0.71, 0.0}